\documentclass[leqno,11pt]{amsart}
\usepackage{amsmath,amssymb}
\usepackage{color}
\usepackage{comment}

\allowdisplaybreaks

\theoremstyle{plain}
\newtheorem{theorem}{Theorem}[section]
\newtheorem{lemma}[theorem]{Lemma}
\newtheorem{proposition}[theorem]{Proposition}
\newtheorem{corollary}[theorem]{Corollary}
\newtheorem*{theorem-A}{Theorem A}
\newtheorem*{theorem-B}{Theorem B}
\newtheorem*{theorem-C}{Theorem C}
\newtheorem*{theorem-D}{Theorem D}
\newtheorem*{theorem-E}{Theorem E}
\newtheorem*{theorem-F}{Theorem F}
\theoremstyle{definition}

\newtheorem{remark}[theorem]{Remark}

\newtheorem{example}[theorem]{Example}

\def\rn{\mathbb R\sp n}
\def\Rn{\mathbb R\sp n}
\def\R{\mathbb R}
\def\N{\mathbb N}
\def\S{\mathbb S}

\def\M{\mathcal M}

\def\limsup{\operatornamewithlimits{lim\,sup}}

\def\o{\Omega}

\def\med{\operatorname{med}}

\def\ddd{d_{E,\o}}

\newtoks\by
\newtoks\paper
\newtoks\book
\newtoks\jour
\newtoks\yr
\newtoks\pages
\newtoks\vol
\newtoks\publ
\newtoks\eds
\newtoks\proc
\newtoks\no
\def\ota{{\hbox{???}}}
\def\cLear{\by=\ota\paper=\ota\book=\ota\jour=\ota\yr=\ota
\pages=\ota\vol=\ota\publ=\ota}
\def\endpaper{\the\by, \textit{\the\paper},
{\the\jour} \textbf{\the\vol} (\the\yr), \the\pages.\cLear}
\def\endbook{\the\by, \textit{\the\book}, \the\publ.\cLear}
\def\endprep{\the\by, \textit{\the\paper}, \the\jour.\cLear}
\def\endproc{\the\by, \textit{\the\paper}, \the\publ, \the\pages.\cLear}
\def\name#1#2{#1 #2}
\def\et{ and }
\numberwithin{equation}{section}

\newcommand{\norm}[1]{{\left\vert\kern-0.25ex\left\vert\kern-0.25ex\left\vert #1
    \right\vert\kern-0.25ex\right\vert\kern-0.25ex\right\vert}}

\begin{document}

\date{\today}

\title[Fractional Orlicz-Sobolev embeddings]{Fractional Orlicz-Sobolev embeddings}

\author {Angela Alberico, Andrea Cianchi, Lubo\v s Pick and Lenka Slav\'ikov\'a}

\address{Angela Alberico, Istituto per le Applicazioni del Calcolo ``M. Picone''\\
Consiglio Nazionale delle Ricerche \\
Via Pietro Castellino 111\\
80131 Napoli\\
Italy} \email{a.alberico@iac.cnr.it}

\address{Andrea Cianchi, Dipartimento di Matematica e Informatica \lq\lq U. Dini"\\
Universit\`a di Firenze\\
Viale Morgagni 67/a\\
50134 Firenze\\
Italy} \email{andrea.cianchi@unifi.it}

\address{Lubo\v s Pick, Department of Mathematical Analysis\\
Faculty of Mathematics and Physics\\
Charles University\\
Sokolovsk\'a~83\\
186~75 Praha~8\\
Czech Republic} \email{pick@karlin.mff.cuni.cz}

\address{Lenka Slav\'ikov\'a, Mathematical Institute, University of Bonn, Endenicher Allee 60,
53115 Bonn, Germany
and
Department of Mathematical Analysis, Faculty of Mathematics and
Physics,  Charles University, Sokolovsk\'a~83,
186~75 Praha~8, Czech Republic}
\email{slavikova@karlin.mff.cuni.cz}
\urladdr{}


\begin{abstract}
The optimal Orlicz target space is exhibited for embeddings of fractional-order Orlicz--Sobolev spaces in $\rn$. An improved embedding with an  Orlicz--Lorentz target space, which is optimal in the broader class of all rearrangement-invariant spaces, is also established. Both  spaces of order  $s\in (0,1)$, and higher-order spaces are considered. Related Hardy type inequalities are proposed as well.  An extension theorem is proved, that enables us to derive embeddings for spaces defined in Lipschitz domains. Necessary and sufficient conditions for the compactness of fractional Orlicz-Sobolev embeddings are provided.
\end{abstract}

\date{\today}

\maketitle

\section{Introduction}

The present paper is aimed at offering optimal Sobolev--Poincar\'e type inequalities and related embeddings for fractional-order Orlicz-Sobolev spaces. These spaces extend the classical fractional Sobolev spaces introduced in \cite{Aronszajn, Gagliardo, Slobodeckij}. Given a number $s\in (0,1)$ and a Young function $A: [0, \infty ) \to[0, \infty]$, namely a convex function vanishing  at $0$, the fractional-order Orlicz-Sobolev space $W^{s,A}(\rn)$ is defined via a seminorm $|\,\cdot \,|_{s,A,\rn}$ built upon the functional defined as
\begin{equation}\label{intro1}
\int_{\rn} \int_{\rn}A\left(\frac{|u(x)-u(y)|}{|x-y|^s}\right)\frac{\,dx\,dy}{|x-y|^n}
\end{equation}
for a measurable function $u$ in $\rn$. The definition of the seminorm $|\,\cdot \,|_{s,A,\rn}$ is given, via the functional \eqref{intro1}, in analogy with
%the spirit of
the notion of  Luxemburg norm in Orlicz spaces.
The bases  for a theory of
 the spaces $W^{s,A}(\rn)$, motivated e.g. by the analysis of nonlinear fractional  Laplacians with non-polynomial kernels, have  recently been laid in \cite{deNBS, BonderSalort} under the
     $\Delta_2$-condition and a sublinear growth condition near zero  on the Young function $A$. Neither of these additional assumptions will be imposed throughout.
\par The standard Gagliardo functional and the associated seminorm $|\,\cdot \,|_{s,p,\rn}$, underlying the notion of the fractional Sobolev space $W^{s,p}(\rn)$ for $p \in [1, \infty)$, are recovered   by the choice $A(t)=t^p$.
A renewed interest in the area around fractional Sobolev spaces has flourished in the last two decades. This has been favoured by a  myriad of investigations on   nonlocal equations of elliptic and parabolic type, whose solutions naturally belong to the spaces
$W^{s,p}(\rn)$.
%for which the spaces $W^{s,p}(\rn)$ provide a natural function space framework.
%The spaces $W^{s,p}(\rn)$ provide a natural framework for the analysis of classes of nonlocal operators, which have been the subject of a myriad of investigations in the last  decade.  This has favoured a renewed interest on the area around fractional Sobolev spaces.
A touch of recent contributions in this connection is furnished by \cite{BaFiVa, BFV, BBM, BBM_2002, BouPonVanSch, BraCin, BraSal, BrezisMiro1, BrezisMiro2, BrezisNguyen, CaFaSoWe, CaSir, CaSil, CaSaSil, CaRoSa, CaVal, CaDiFiLaSl, CW, ChenWe, CFW, CostadeFigYang, DydaFrank, DydaVaha, FiFuMaMiMo, FMT, FJX, Heuer, KuMi, Ludwig, Mallick, MaranoMosco, MaSh2, MaSh1, MuNa1, PaPi, PaRu, R-OSe1, R-OSe2, SeegerTrebels, Tz, Zhou}.
Comprehensive treatments of the theory of fractional Sobolev spaces, as special instances of the more general Besov spaces, can be found e.g. in \cite{BesovIlinNik,  Leoni}. A self-contained presentation of their basic properties is provided in \cite{DPV}.
\par Embeddings for the spaces  $W^{s,p}(\rn)$ into Lebesgue spaces are classical. In particular, if $1\leq p < \frac ns$, then there exists a constant $C$ such that
\begin{equation}\label{intro2}
	\|u\|_{L^{\frac{np}{n-sp}}(\R^n)}
		\le C |u|_{s,p,\R^n}
		\end{equation}
for every measurable function $u$ decaying to $0$ (in a suitable sense) near infinity.
\\ An improved version of inequality \eqref{intro2} has been established in \cite{FrankSeiringer}. It asserts that, in fact,
\begin{equation}\label{intro3}
	\|u\|_{L^{\frac{np}{n-sp},p}(\R^n)}
		\le C |u|_{s,p,\R^n}\,,
\end{equation}
where the Lorentz space $L^{\frac{np}{n-sp},p}(\R^n)\subsetneq L^{\frac{np}{n-sp}}(\R^n)$.
\par
Our main results amount to  sharp counterparts of inequalities \eqref{intro2} and \eqref{intro3} for general seminorms $|\,\cdot \,|_{s,A,\rn}$. Given any $s\in (0,1)$ and any Young function $A$, with subcritical growth   corresponding to the assumption $p < \frac ns$ in the case of powers, we detect the optimal Orlicz target space $L^{A_{\frac ns}}(\rn)$ such that
\begin{equation}\label{intro4}
	\|u\|_{L^{A_{\frac ns}}(\rn)}
		\le C |u|_{s,A,\R^n}
\end{equation}
for some constant $C$ and every measurable function $u$ decaying to $0$   near infinity. An explicit formula for the Young function $A_{\frac ns}$ is provided, which only depends on $A$ and on the ratio $\frac ns$.
Here, and in what follows, the expression \lq\lq optimal target space" referred to an embedding or inequality means \lq\lq smallest possible" within a specified family, in the sense that if the embedding also holds with the optimal target space in question replaced by another space from the same family, then the former is continuously embedded into the latter.
\\ Inequality \eqref{intro4} is derived as a consequence of a stronger inequality
\begin{equation}\label{intro5}
	\|u\|_{L({\widehat A},\frac{n}{s})(\R^n)}
		\le C |u|_{s,A,\R^n}\,,
\end{equation}
where $L({\widehat A},\frac{n}{s})(\R^n)$ -- a space of Orlicz-Lorentz type depending only on $A$ and $\frac ns$  -- is optimal in the larger class of all rearrangement-invariant spaces.
\\ In particular, these results reproduce inequalities \eqref{intro2} and \eqref{intro3} when $A(t)=t^p$. The latter also provides new information about \eqref{intro3}, and  tells us that the space $L^{\frac{np}{n-sp},p}(\R^n)$ is indeed optimal in \eqref{intro3} among all rearrangement-invariant spaces. Let us mention that embeddings for the spaces $W^{s,A}(\rn)$, under additional technical assumptions on $A$ and into non-optimal (in general) target spaces have recently appeared in \cite{BO}.
\par
Optimal inequalities  for Orlicz-Sobolev spaces of  fractional-order $s>1$ are also presented. As customary, these spaces are defined on replacing the function $u$ in \eqref{intro1} by $\nabla ^{[s]}u$, the vector of all weak derivatives of $u$ whose order is the integer part $[s]$ of $s$. The  inequalities in question parallel \eqref{intro4} and \eqref{intro5}. They extend inequalities   \eqref{intro4} and \eqref{intro5} to any $s \in (0,n) \setminus \N$,
 and take the form
\begin{equation}\label{intro6}
	\|u\|_{L^{A_{\frac ns}}(\rn)}
		\le C \big|\nabla ^{[s]}u\big|_{s,A,\R^n}\,,
\end{equation}
and
\begin{equation}\label{intro7}
	\|u\|_{L({\widehat A},\frac{n}{s})(\R^n)}
		\le C \big|\nabla ^{[s]}u\big|_{s,A,\R^n}\,,
\end{equation}
respectively, for functions $u$ all of whose derivatives up to the order $[s]$  decay to $0$ near infinity. The sharp spaces  $L^{A_{\frac ns}}(\rn)$ and  $L({\widehat A},\frac{n}{s})(\R^n)$ appearing in \eqref{intro6} and \eqref{intro7} are defined exactly via the same formulas as in the case when $s\in (0,1)$, save that now are applied for  $s \in (1,n) \setminus \N$. Our conclusions  could thus be formulated via statements simultaneously covering the cases when $s\in (0,1)$ and $s\in (1,n) \setminus \N$. We prefer to enucleate
% premise
the results for $s\in (0,1)$ in a separate section for ease of presentation, and also because those for $s\in (1,n) \setminus \N$ call for  a  combination of the former and of inequalities for integer-order Orlicz-Sobolev spaces.
%\par An  interesting feature of the results announced above  is that, although relying on a substantially different approach, they
\par  The integer-order Orlicz-Sobolev embeddings have been established in  \cite{cianchi_IUMJ, cianchi_CPDE, cianchi-ibero, cianchi_forum}.
 Importantly, these embeddings are exactly matched by the fractional-order embeddings announced above, although the latter rely on a substantially different approach.
Indeed,  applying our formulas for the optimal spaces  $L^{A_{\frac ns}}(\rn)$ and  $L({\widehat A},\frac{n}{s})(\R^n)$ with $s \in \N$ recovers the optimal Orlicz and the optimal rearrangement-invariant space, respectively, in the Orlicz-Sobolev inequality of integer-order $s$.
%Importantly, although relying on a substantially different approach, the results announced above
%perfectly match those for integer-order Orlicz-Sobolev spaces established in \cite{cianchi_IUMJ, cianchi_CPDE, cianchi-ibero, cianchi_forum}. Actually,  applying our formulas for the optimal spaces  $L^{A_{\frac ns}}(\rn)$ and  $L({\widehat A},\frac{n}{s})(\R^n)$ with $s \in \N$ recovers the optimal Orlicz and the optimal rearrangement-invariant space in the Orlicz-Sobolev inequality of integer-order $s$.
\par  Closely related fractional-order Hardy type inequalities in $\rn$ are proposed as well. In fact, a crucial step in our approach is a Hardy inequality of order $s \in (0,1)$, which extends to the Orlicz realm a result from \cite{MaSh1}. The Hardy inequality for $s \in (1,n) \setminus \N$ is, by contrast, deduced as a consequence of inequality \eqref{intro7}.
\par Analogous inequalities and embeddings when $\rn$ is replaced by a  sufficiently regular bounded subset $\Omega$ -- a bounded Lipschitz domain -- are established. In order to treat this variant, we prove  an extension theorem for functions in the space $W^{s,A}(\Omega)$, a  generalization of   a well-known result for  fractional Sobolev spaces that can be found, for instance,  in  \cite[Theorem 5.4]{DPV}.
\par Compact embeddings for fractional-order Orlicz-Sobolev spaces are characterized as well.  A necessary and sufficient condition on a rearrangement-invariant space $Y(\Omega)$ for the embedding
$$W^{s,A} (\Omega) \to Y(\Omega)$$
to be compact is exhibited when $s\in (0,n) \setminus \N$ and $\Omega$ is a bounded Lipschitz domain.  As a consequence, the embedding
$$W^{s,A} (\Omega) \to L^B(\Omega)$$
is shown to be compact for an Orlicz space $L^B(\Omega)$ if and only if the Young function $B$ grows essentially more slowly near infinity than the function $A_{\frac ns}$ appearing in \eqref{intro6}. Local versions of these compactness results are also provided in the case when $\Omega=\rn$.
Like the other results of this paper, on setting $s\in \N$ in their statements, our fractional compact embeddings
perfectly tie up with    their integer-order counterparts proved in \cite{cianchi_IUMJ, cianchi_CPDE, cianchi-ibero, cianchi_forum}.
\par The material is organized as follows. Section \ref{rispaces} and Section \ref{fracorlicz} are devoted to notations, definitions and necessary background from the theory of Orlicz and rearrangement-invariant spaces, and the theory of integer and fractional-order Orlicz-Sobolev spaces, respectively. Several sharp one-dimensional Hardy type inequalities in Orlicz and rearrangement-invariant spaces of critical use in the proofs of our main results are collected in Section \ref{1d}. Some of them are known, but others are new. The main results are exposed in Sections \ref{hardy}--\ref{compact}. Section \ref{hardy} deals with the Hardy inequality for fractional Orlicz-Sobolev spaces on $\rn$ of order $s\in (0,1)$. Orlicz-Sobolev embeddings for spaces of order in the same range are presented in Section \ref{s<1}, whereas Section \ref{s>1} is concerned with embeddings of arbitrary order. Fractional Orlicz-Sobolev spaces on open subsets of
$\rn$ and their embeddings are the subject of Section \ref{domains}. The objective of the final Section \ref{compact} are criteria for the compactness of embeddings.

\section{Orlicz spaces and rearrangement-invariant   spaces}\label{rispaces}

 A  function  $A: [0,
\infty ) \to [0, \infty ]$ is called a Young function if it has
the form
\begin{equation}\label{young}
A(t) = \int _0^t a(\tau ) d\tau \quad \text{for $t \geq 0$}
\end{equation}
for some non-decreasing, left-continuous function $a: [0, \infty )
\to [0, \infty ]$ which is neither identically equal to $0$ nor to
$\infty$. Clearly, any convex (non trivial) function from $[0,
\infty )$ into $[0, \infty ]$, which is left-continuous and
vanishes at $0$, is a Young function.
\par\noindent
%Note that
%\begin{equation}\label{aA}
%t/2 \, a(t/2) \leq A(t) \leq t\, a(t) \quad \text{for $t\geq 0\,.$}
%\end{equation}
Note that, if $k \geq 1$, then
\begin{equation}\label{kt}
k\,A(t) \leq A(kt) \quad \text{for $t \geq 0$.}
\end{equation}
%A function $A$ is said to satisfy the $\Delta_2$--condition  if there exists a constant $c$ such that
%\begin{equation}\label{delta2}
%A(2t) \leq c A(t) \quad \text{for $t \geq 0$.}
%\end{equation}
The Young conjugate $\widetilde{A}$ of $A$  is defined by
%\begin{equation}\label{2.8}
$$
\widetilde{A}(t) = \sup \{\tau t-A(\tau ):\,\tau \geq 0\}  \quad \text{for $t\geq 0$.}
$$
The following representation formula for $\widetilde{A}$ holds:
\begin{equation}\label{youngconj}
\widetilde A(t) = \int _0^t a^{-1}(\tau ) d\tau \quad \text{for $t \geq 0$.}
\end{equation}
 Here, $a^{-1}$ denotes the left-continuous inverse of the
function $a$ appearing in \eqref{young}.
%One can show that
%\begin{equation}\label{AAtilde}
%t \leq A^{-1}(t) \widetilde A^{-1}(t) \leq 2\, t \quad \text{for $t\geq 0$,}
%\end{equation}
%where $A^{-1}$ and $\widetilde A^{-1}$ stand for the generalized
%right-continuous inverses of $A$ and $\widetilde A$, respectively.
%\par\noindent
%A Young function $A$ is said to belong to the class $\Delta _2$
%globally if there exists a positive constant $C$ such that
%\begin{equation}\label{delta2}
%A(2t) \leq C A(t)
%\end{equation}
%for $t \geq 0$. If there exists $t _0 >0$ such that inequality
%\eqref{delta2} just holds for  $t \geq t_0$, then we say that $A \in
%\Delta _2$ near infinity. The function $A$ is said to belong to the
%class $\nabla _2$ globally if there exists a  constant $K>1$ such
%that
%\begin{equation}\label{nabla2}
%A(Kt) \geq 2K A(t)
%\end{equation}
%for $t \geq 0$. Membership of $A$ to $\nabla _2$ near infinity is
%defined accordingly. One has that
%\begin{equation}\label{deltanabla}
%A \in \Delta _2 \quad \hbox{[near infinity]} \quad \hbox{if and only
%if} \quad \widetilde A \in \nabla _2  \quad \hbox{[near infinity]}.
%\end{equation}
\\ A  Young function $A$ is said to dominate another Young function $B$
globally   if there exists a positive
 constant $C$  such that
\begin{equation}\label{B.5bis}
B(t)\leq A(C t) \quad \text{for $ t\geq 0.$}
\end{equation}
The function $A$ is said to dominate $B$ near infinity[resp. near zero] if there
exists $t_0> 0$ such that \eqref{B.5bis} holds for $t \geq t_0$ [$t \leq t_0$]. The functions $A$ and $B$ are called equivalent globally [near infinity] [near zero] if they dominate each other globally [near infinity] [near zero]. Plainly, the function $A$ dominates [is equivalent to] $B$ globally if and only if
$A$ dominates [is equivalent to] $B$ near zero and near infinity.
\\
The function $B$ is said to grow essentially more slowly near infinity than $A$ if
\begin{equation}\label{nov 110}
\lim _{t \to \infty} \frac{B(\lambda t)}{A(t)} =0
\end{equation}
for every $\lambda >0$. Note that condition \eqref{nov 110} is equivalent to
\begin{equation}\label{nov 110bis}
\lim _{t \to \infty} \frac{A^{-1}(t)}{B^{-1}(t)} =0.
\end{equation}
\\
The growth of a Young function $A$ can be compared with that of a power function via its Matuszewska-Orlicz indices.
Recall that the upper Matuszewska-Orlicz  index $I(A)$ of a finite-valued Young function $A$ is defined as
\begin{equation}\label{index}I(A) = \lim_{\lambda \to \infty} \frac{\log \Big(\sup _{t>0}\frac{A(\lambda t)}{A(t)}\Big)}{\log \lambda}.
\end{equation}
The Matuszewska-Orlicz index  $I_\infty (A)$  of $A$ near infinity  is defined analogously, with
$\sup _{t>0}\frac{A(\lambda t)}{A(t)}$ replaced by $\limsup _{t \to \infty}\frac{A(\lambda t)}{A(t)}$.
\par
Let $\Omega$ be  a measurable subset of $\rn$, with $n\geq 1$. We set
\begin{equation}\label{M}
\mathcal{M}(\Omega)=\{u:\Omega \to \R : \text{$u$ is  measurable}\}\,,
\end{equation}
and
\begin{equation}\label{M+}
\mathcal{M}_+(\Omega)=\{ u\in \mathcal{M}(\Omega) : u \geq 0\}\,.
\end{equation}
The notation $\mathcal{M}_d(\Omega)$ is adopted for the subset of $\mathcal{M}(\Omega)$ of those functions $u$ that decay near infinity, in the sense that all their level sets $\{|u|>t\}$ have finite Lebesgue measure for $t>0$. Namely,
\begin{equation}\label{Md}
\mathcal{M}_d(\Omega)=\{ u\in \mathcal{M}(\Omega) : |\{|u|>t\}|<\infty\,\, \text{for every $t>0$}\}\,,
\end{equation}
where $|E|$ stands for the Lebesgue measure of a set $E\subset \rn$.
Of course, $\mathcal{M}_d(\Omega)= \mathcal{M}(\Omega)$ provided that $|\Omega|<\infty$.
\par
The Orlicz space $L^A (\Omega )$, associated with a Young function
$A$,  on  a measurable subset $\Omega$ of $\rn$, is the Banach
function space of those real-valued measurable functions $u$ in
$\Omega$ for which the
 Luxemburg norm
\begin{equation*}
 \|u\|_{L^A(\Omega )}= \inf \left\{ \lambda >0 :  \int_{\Omega }A
\left( \frac{|u|}{\lambda} \right) dx \leq 1 \right\}\,
\end{equation*}
is finite. In particular, $L^A (\Omega )= L^p (\Omega )$ if $A(t)=
t^p$ for some $p \in [1, \infty )$, and $L^A (\Omega )= L^\infty
(\Omega )$ if $A(t)=0$ for $t\in [0, 1]$ and $A(t) = \infty$ for
$t>0$.
\\ When convenient for specific choices of $A$, we shall also adopt the notation $A(L)(\o)$ to denote the Orlicz space $L^A(\o)$.
\par\noindent
The H\"older type inequality
\begin{equation}\label{holder}
\int _\Omega |u v|\,dx \leq 2\|u\|_{L^A(\Omega )}
\|v\|_{L^{\widetilde A}(\Omega )}
\end{equation}
holds for every $u \in L^A(\Omega )$ and $v\in L^{\widetilde
A}(\Omega )$.
\par\noindent If $A$ dominates $B$ globally, then
\begin{equation}\label{normineq}
\|u \|_{L^B(\Omega )} \leq C \|u \|_{L^A(\Omega )}
\end{equation}
for every $u \in L^A(\Omega )$, where $C$ is the same constant as in
\eqref{B.5bis}.  If $|\Omega|<\infty$   and $A$ dominates $B$ near infinity, then inequality
\eqref{normineq} continues to hold for some constant $C$ depending also on $A$, $B$ and $|\Omega|$. Thus, if $A$ is globally equivalent to $B$, then $L^A(\o)= L^B(\o)$, up to equivalent norms.  The same is true even if $A$ and $B$ are just equivalent near infinity, provided that $|\o|<\infty$.

\par The Orlicz spaces are members of the more general class of rearrangement-invariant spaces, whose definition is based upon that of decreasing rearrangement of a function.
\\
The \emph{decreasing
rearrangement} $u^{\ast}$ of a function $u\in \mathcal M (\Omega)$ is the (unique) non-increasing,
right-continuous function from $[0,\infty)$ into $[0,\infty]$
which is equidistributed with $u$. In formulas,
\begin{equation*}
u^{\ast}(r) = \inf \{t\geq 0: |\{x\in \o: |u(x)|>t \}|\leq r \} \quad \text{for $t\geq 0$.}
\end{equation*}
Moreover, we define the function $u^{\ast \ast}: (0, \infty)  \to [0, \infty]$ as
$$
u^{\ast \ast}(r) =
\frac{1}{r}\int_{0}^{r}u^{\ast}(\varrho) d\varrho \quad  \text{for $r> 0$.}
$$
Notice that $u^* \leq u^{\ast \ast}$.
 The
\emph{Hardy-Littlewood inequality} states that
\begin{equation}\label{B.0}
\int_{\o}|uv| \,dx \leq \int_{0}^{|\o|}u^{\ast}v^{\ast}\,dr
\end{equation}
for all functions $u, v \in \mathcal M(\o)$. As a consequence, one also has that
\begin{equation}\label{HLA}
\int_{\o}A(|uv|) \,dx \leq \int_{0}^{|\o|}A(u^{\ast}v^{\ast})\,dr
\end{equation}
for every Young function $A$.
 \par\noindent A
Banach function space $X(\o)$, in the sense of Luxemburg \cite[Chapter 1, Section 1]{BS},   is called a \emph{rearrangement-invariant space}   if
%\todo[inline]{A: I would go for rearrangement-invariant everywhere, as we did in our recent papers}
%\todo[inline]{Angela: OK Andrea! Done!}
\begin{equation}\label{B.1}
 \|u\|_{X(\o)} = \|v \|_{X(\o)} \quad \text{whenever $u^*=v^*$.}
 \end{equation}
\par\noindent
The \emph{associate space} $X^{'}(\o)$ of $X(\o)$ is the rearrangement-invariant
space of all  e functions in $\mathcal M(\o)$ for which the norm
\begin{equation}\label{B.2}
 \|v \|_{X^{'}(\o)} =
\sup_{u \neq 0} \frac{\int _\o|uv| dx}{\|u \|_{X(\o)}}
\end{equation}
is finite.
%One has
 %\begin{equation}\label{B.2'}
%(X^{'})^{'}(\o)= X(\o)\,.
%\end{equation}
Notice that, given two rearrangement-invariant spaces  $X(\o)$ and $Y(\o)$,
\begin{equation}\label{nov103}
X(\o) \to Y(\o) \quad \text{if and only if} \quad Y'(\o) \to X'(\o)
\end{equation}
with the same embedding constants. Here, and in  what follows, the arrow $\lq\lq \to "$ stands for continuous embedding.
\\ If $X(\o)= L^A(\o)$ for some Young function $A$, then
\begin{equation}\label{nov104}
{(L^A)}'(\o) = L^{\widetilde A}(\o)\,,
\end{equation}
up to equivalent norms, with absolute equivalence constants.
\\  Let  $\Omega$ be a  measurable set  in $ \rn$. With any function $u : \o \to \R$,  we can associate the function $\mathcal E_0(u) : \rn \to \R$  defined as
\begin{equation}\label{E0}
\mathcal E_0(u) (x)=  \begin{cases} u(x) \quad & \text{if $x \in \Omega$}
\\ 0 \quad & \text{if $x \in \rn \setminus \Omega$}\,.
\end{cases}
\end{equation}
The map $u \mapsto \mathcal E_0(u)$ plainly defines  a linear operator.  Given a rearrangement-invariant space
$X(\rn)$, we denote by $X(\Omega)$ the rearrangement-invariant space on $\Omega$ equipped with the   norm defined as
\begin{equation}\label{Xr}
\|u\|_{X(\Omega)}= \|\mathcal E_0(u)\|_{X(\rn)}
\end{equation}
for every function $u \in \mathcal M(\Omega)$.
Note that, if $X(\rn)= L^A(\rn)$ for some Young function $A$, then the space  $X(\o)$ defined   as in \eqref{Xr}  agrees with  the Orlicz space $L^A(\o)$.
\par\noindent
The \emph{representation space} $\overline{X}(0,|\o|)$ of a rearrangement-invariant space $X(\o)$
is the unique rearrangement-invariant space on $(0,|\o|)$ satisfying
\begin{equation}\label{B.3}
\|u \|_{X(\o)} = \|u^{\ast} \|_{\overline{X}(0,|\o|)}
\end{equation}
for every $u\in X(\o)$.
\par\noindent
If $|\o|<\infty$, then
\begin{equation}\label{B.3'}
L^\infty (\o) \to X(\o) \to L^1(\o)
\end{equation}
for every rearrangement-invariant  space $X(\o)$.
\\
Given any $\lambda>0$ and $L>0$, the \textit{dilation operator} $E_{\lambda}$, defined at
$f\in \M(0,L)$ by
\begin{equation}\label{dilation}
  (E_{\lambda} f)(t)=
\begin{cases}
  f(t/\lambda)\quad&\textup{if} \quad 0< t/\lambda\ \leq L
\\
  0&\textup{if}\quad  L< t/\lambda,
  \end{cases}
\end{equation}
is bounded on any rearrangement-invariant~space $X(0,L)$, with norm
not exceeding $\max\{1, 1/\lambda\}$.
\\
Assume that $|\o|<\infty$ and let $X(\o)$ and $Y(\o)$ be rearrangement-invariant spaces.
We say that the space $X(\o)$ is almost-compactly embedded into $Y(\o)$ if
\begin{equation}\label{E:almost_compact_definition}
\lim_{L\to 0^+} \sup_{\|u\|_{X(\o)}\leq 1} \|\chi_{(0,L)}u^*\|_{\overline{Y}(0,|\o|)}=0.
\end{equation}
Here, and in what follows, $\chi_E$ denotes the characteristic function of a set $E$.
By~\cite[Theorem 3.1]{S}, equation~\eqref{E:almost_compact_definition} is equivalent to the following condition:
\begin{equation}\label{E:convergence}
\text{if $\{u_i\}$ is a bounded sequence in $X(\o)$ such that $u_i \rightarrow 0$ a.e., then } \lim_{i\to \infty} \|u_i\|_{Y(\o)}=0.
\end{equation}
In the special case of Orlicz spaces $L^A(\o)$ and $L^B(\o)$, one has that
\begin{equation}\label{dec230}
 \text{$L^A(\o)$ is almost-compactly embedded into $L^B(\o)$ if and only if
$B$ grows essentially more slowly than $A$,}
\end{equation}
see,   e.g.,\cite[Theorem 4.17.7]{PKJF}).

\par The Orlicz-Lorentz  spaces are  a family of function spaces  that extends that of the Orlicz spaces. Given  a Young function $A$ and a number $q\in \R$, we denote by
 $L(A,q)(\o)$ the Orlicz-Lorentz space   of all functions $u \in \mathcal M(\o)$ for which the quantity
\begin{equation}\label{aug300}
	\|u\|_{L(A, q)(\o)}
		= \|r^{-\frac{1}{q}}u^{*}(r)\|_{L^A(0,|\o|)}
\end{equation}
is finite. Under suitable assumptions on $A$ and $q$, this quantity is a norm, and $L(A,q)(\o)$, equipped with this norm, is a (non-trivial) rearrangement-invariant space. This is certainly the case when $q>1$ and
\begin{equation}\label{aug310}
\int^\infty \frac{A(t)}{t^{1+q}}\, dt < \infty\,,
\end{equation}
see \cite[Proposition 2.1]{cianchi-ibero}.
\\ The  spaces $L(A, q)(\o)$  come into play in the description of the associate spaces of another closely related family of Orlicz-Lorentz type spaces. They are denoted by
$L[A,q](\Omega)$, and consist of all functions $u\in \mathcal M (\o)$ that make the functional
\begin{equation}\label{aug301}
	\|u\|_{L[A, q](\o)}
		= \|r^{-\frac{1}{q}}u^{**}(r)\|_{L^A(0,|\o|)}
\end{equation}
finite. One can verify that, if $q<-1$, then  this functional is a rearrangement-invariant norm that renders $L[A, q](\o)$ a rearrangement-invariant space provided that   either $|\o|<\infty$, or $|\o|=\infty$ and
\begin{equation}
   \label{jan30}
\int_0 \frac{A(t)}{t^{1+(-q)'}}\, dt < \infty\,,
\end{equation}
where $(-q)'=\tfrac q{q+1}$, the H\"older conjugate of $-q$.
%
%$q>1$ and condition %\eqref{aug310} holds, or $q<0$.
%\note[inline]{Lubos: this is true only if $|\Omega|<\infty$, otherwise (P4) might be violated. Does the assumption $|\Omega|<\infty$ apply to this paragraph? It is mentioned several times above but it is not clear to me whether it reaches here.}
%\todo[inline]{A: we need these norms especially when $\o=\rn$, and hence $|\o|=\infty$. It is good that you pointed this out. In fact, the  previous version was wrong. The condition which ensures property (P4) is \eqref{jan30}. In the case of interest to us, namely $L[\widetilde A, -\tfrac ns]$, this condition is equivalent to \eqref{E:0''}, which we always keep in force.}
For special choices of the function $A$, the space $L(A, q)(\o)$ agrees, up to equivalent norms, with customary Lorentz type spaces. Assume, for instance, that $|\o|<\infty$ and that
$$\text{$A(t)$ is equivalent to $t^p (\log t)^\alpha(\log\log t)^\beta$ near infinity.}$$
for some powers $p$, $\alpha$ and $\beta$. Then, depending on the relations among $p$, $q$ and $\alpha$, the space  $L(A, q)(\o)$ agrees with the Lorentz space $L^{\sigma, p}(\o)$, the Lorentz-Zygmund space $L^{\sigma,p;  \gamma}(\o)$  or with the generalized Lorentz-Zygmund space $L^{\sigma,p;  \gamma, \delta}(\o)$, for a suitable choice of the parameters  $\sigma, p \in (0, \infty]$ and  $\gamma, \delta \in \R$.
%\todo[inline]{A: Lubos, please check these definitions and add something about the range for the parameters (minimal information for our needs in the examples)}
%\todo[inline]{Lubos: the definitions and relations are OK as long as $|\Omega|<\infty$, otherwise double indices are needed. Since $\log\log2<0$, I suggest a slight modification below.}
%\todo[inline]{A: we are using these spaces only when $|\o|<\infty$ in the examples.}
Recall that $L^{\sigma, p}(\o)$, $L^{\sigma,p;  \gamma}(\o)$  and $L^{\sigma,p;  \gamma, \delta}(\o)$ are the spaces of those functions $u\in \mathcal M(\o)$ for which the quantity
\begin{equation}\label{dec271}
\|u\|_{L^{\sigma, p} (\o)}=
 \big\|u^*(r) r^{\frac 1\sigma - \frac 1p}\big\|_{L^p(0, |\o|)},
\end{equation}
\begin{equation}\label{dec272}
\|u\|_{L^{\sigma, p;   \gamma} (\o)}=
 \big\|u^*(r) r^{\frac 1\sigma - \frac 1p}(\log (1+ |\o|/r))^\gamma \big\|_{L^p(0, |\o|)},
\end{equation}
\begin{equation}\label{dec273}
\|u\|_{L^{\sigma, p;   \gamma, \delta} (\o)}=
 \big\|u^*(r) r^{\frac 1\sigma - \frac 1p}(\log (1+ |\o|/r))^\gamma (\log(1+\log (1+ |\o|/r)))^\delta \big\|_{L^p(0, |\o|)},
\end{equation}
respectively, is finite. Notice that $L^p(\o)=L^{p,p}(\o)$, $L^{\sigma,p; 0}(\o)=L^{\sigma, p}(\o)$ and
$L^{\sigma, p;   \gamma, 0} (\o)=
L^{\sigma, p;   \gamma} (\o)$.
The full range of parameters  $\sigma, p, \gamma$  for which $L^{\sigma,p;  \gamma}(\o)$ is nontrivial is exhibited in~\cite[Remark 9.10.2(a)]{PKJF}. A characterization of the parameters for which the functional defined by  \eqref{dec272} is (equivalent to) a norm, and $L^{\sigma,p;  \gamma}(\o)$ equipped with this norm is  a rearrangement-invariant space, can be found in~\cite[Theorem 9.10.4]{PKJF}. This will always be the case in our use of the spaces $L^{\sigma,p;  \gamma}(\o)$, as well as of that of the spaces $L^{\sigma,p;  \gamma,\delta}(\o)$, for which an analogous characterization is stated in~\cite[Lemma 9.3.1, Remark 9.3.2 and Lemma 9.5.6]{PKJF}.
%, where it is also stated that, in these cases, the associate space of $L^{\sigma,p;  \gamma}(\o)$ is always equivalent to $L^{\sigma',p'; -\gamma}(\o)$.
%Analogous results concerning the spaces $L^{\sigma,p;  \gamma,\delta}(\o)$ are stated in~\cite[Lemma 9.5.6, Lemma 9.5.6 and Theorem 9.6.13]{PKJF}.}

%
%\begin{equation}\label{dec270}
%L(A, q)(\o) = L^{\frac{pq}{q-p},p; \frac \alpha p}(\o),
%\end{equation}
%up to equivalent norms,  where the latter is a Lorentz-Zygmund space. Recall that, if $\sigma ??$, $p???$ and $\alpha ???$, then the Lorentz-Zygmund space  $L^{\sigma,p;  \alpha}(\o)$ is endowed with the norm given by
%\begin{equation}\label{dec271}
%\|u\|_{L^{\sigma, p;   \alpha} (\o)}=
% \big\|u^*(r) r^{\frac 1\sigma - \frac 1p}(\log (1+ |\o|/r))^\alpha \big\|_{L^p(0, |\o|)}
%\end{equation}
% for $u\in \mathcal M(\o)$. Moreover, if $p=q$,

\section{Fractional Orlicz-Sobolev spaces}\label{fracorlicz}

Assume  that $\Omega$ is an open subset of $\rn$. Given $m \in \N$ and a  Young function $A$, we
denote by $V^{m,A}(\Omega )$ the homogeneous Orlicz-Sobolev space given by
\begin{equation}\label{homorliczsobolev}
V^{m,A}(\Omega ) = \{ u \in W^{m,1}_{\rm loc} (\Omega):\,   |\nabla ^m u| \in L^A(\Omega)\}.
\end{equation}
Here, $\nabla ^m u$ denotes the vector of all weak derivatives of $u$ of order $m$. If $m=1$, we also simply write $\nabla u$ instead of $\nabla^1 u$.
The notation
$W^{m,A}(\Omega )$  is adopted for the classical Orlicz-Sobolev space defined by
\begin{equation}\label{orliczsobolev}
W^{m,A}(\Omega ) = \{u \in V^{m,A}(\Omega): |\nabla ^k u| \in L^A(\Omega), \, k=0,1, \dots, m-1\}\,,
\end{equation}
where $\nabla ^0u$ has to be interpreted as $u$.
The space $W^{m,A}(\Omega)$ is a Banach space equipped with the
norm
$$\|u\|_{W^{m,A}(\Omega )} = \sum _{k=0}^m  \|\nabla^k u
\|_{L^A(\Omega)}\,.$$
\par Now, let $s\in (0,1)$. The seminorm $|u|_{s,A, \o}$ of  a function $u \in \mathcal M (\Omega)$ is given by
\begin{equation}\label{aug340}
|u|_{s,A, \o}
		= \inf\left\{\lambda>0: \int_{\o} \int_{\o}A\left(\frac{|u(x)-u(y)|}{\lambda|x-y|^s}\right)\frac{\,dx\,dy}{|x-y|^n}\le1\right\}\,.
\end{equation}
The homogeneous
 fractional Orlicz-Sobolev space $V^{s,A}(\o)$ is defined as
\begin{equation}\label{aug341}
	V^{s,A}(\o) = \big\{u \in \mathcal M (\Omega):  |u|_{s,A, \o}<\infty\}\,.
\end{equation}
The definitions of the seminorm $|u|_{s,A, \o}$ and of the space $V^{s,A}(\o)$ carry over to vector-valued functions $u$ just by replacing the absolute value of $u(x)-u(y)$ by the norm  of the same expression on the right-hand side of equation \eqref{aug340}.
\\
The subspace $V^{s,A}(\o) \cap \mathcal M_d(\o)$ of those functions in $V^{s,A}(\o)$ that decay near infinity is denoted by $V^{s,A}_d(\o)$. Thus,
\begin{equation}\label{nov100}
	V^{s,A}_d(\o) =  \{ u\in V^{s,A}(\o): |\{|u|>t\}|<\infty\,\, \text{for every $t>0$}\} \,.
\end{equation}
The definition of $V^{s,A}(\o)$ is extended to all $s\in (0, \infty) \setminus \N$ in a customary way.
Denote by $[s]$ the integer part of $s$, and set $\{s\}= s-[s]$, the fractional part of $s$.
Then we set
\begin{equation}\label{aug343}
V^{s,A}(\Omega ) = \{ W^{[s],1}_{\rm loc} (\Omega):\,  \nabla ^{[s]}u \in V^{\{s\}, A}(\o)\}\,.
\end{equation}
In analogy with \eqref{nov100}, we extend definition \eqref{nov100} to every $s\in (0, \infty) \setminus \N$ on setting
\begin{equation}\label{nov100higher}
	V^{s,A}_d(\o) =  \{ u\in V^{s,A}(\o): |\{|\nabla ^k u|>t\}|<\infty\,\, \text{for $k=0,1, \dots ,[s]$, and for  every $t>0$}\} \,.
\end{equation}
The functional $\big|\nabla ^{[s]}u\big|_{\{s\},A, \Omega}$ defines a norm on the space $V^{s,A}_d(\o)$.
\\ If $|\o|< \infty$ and $s \in (0,1)$, we also define the space
\begin{equation}\label{perpspace}
V^{s,A}_{\perp}(\Omega) = \{u\in V^{s,A}(\Omega): u_\o =0\}\,,
\end{equation}
where
$$u_\Omega = \frac 1{|\o|}\int _\Omega u\; dx\,,$$
the mean value of $u$ over $\o$. Definition  \eqref{perpspace} is extended to any $s\in (0, \infty) \setminus \N$ on setting
\begin{equation}\label{sep100}
V^{s,A}_\perp (\o) = \{u \in V^{s,A}(\o): (\nabla ^ku)_\o =0, \, k=1, \dots , [s]\}\,.
\end{equation}
\\
The fractional-order Orlicz-Sobolev space $W^{s,A}(\Omega)$ is defined, for $s\in (0, \infty) \setminus \N$ and any open set $\o$, as
\begin{equation}\label{aug344}
W^{s,A}(\Omega ) = \{ u \in V^{s,A}(\Omega ):\,  u \in W^{[s],A}(\Omega)\},
\end{equation}
and is a   Banach space equipped with the norm
$$\|u\|_{W^{s,A}(\o)} = \|u\|_{W^{[s],A}(\Omega)} + \big|\nabla ^{[s]}u\big|_{\{s\},A, \Omega}.$$
 Clearly, $W^{s,A}(\o) \to V^{s,A}_d(\o)$, and, as a consequence of Proposition \ref{poinc}, Section \ref{domains}, $W^{s,A}(\o) = V^{s,A}_d(\o)$ if $\Omega$ is bounded. The space $ V^{s,A}_d(\o)$ naturally arises as a natural  maximal domain space  for  various embeddings of ours to hold.
\par
For the sake of completeness, let us recall that inclusion relations hold between integer-order and fractional-order Orlicz-Sobolev spaces.
If $s\in (0,1)$ and $A$ is a Young function, then
\begin{equation}
    \label{incl1}
    W^{1,A}(\rn) \to W^{s,A}(\rn).
\end{equation}
Moreover, denote by
$\overline A$ the Young function defined as
\begin{equation}\label{Abar}
\overline A (t) = \int_0^t \int_{\mathbb S^{n-1}}A(|x_1| \tau)\, d\mathcal H^{n-1}(x)\frac {d\tau}\tau \quad \text{for $t \geq 0$,}
\end{equation}
where $x=(x_1, \dots , x_n)$, $\S^{n-1}$ stands for the $(n-1)$--dimensional unit sphere in $\rn$, and $\mathcal H^{n-1}$ for the $(n-1)$--dimensional Hausdorff measure.
Then the function $\overline A$ is equivalent to $A$, and
if $u\in W^{1,A}(\rn)$, then there exists $\lambda_0>0$ such that
\begin{equation}\label{na1}
\lim_{s\to 1^-} (1-s) \int_{\rn} \int_{\rn} A\left(\frac{|u(x) - u(y)|}{\lambda |x-y|^s}\right)\; \frac {dx\,dy}{|x-y|^n}= \int_{\rn} \overline{A} \left(\frac{|\nabla u|}{\lambda}\right)\; dx
\end{equation}
for every $\lambda \geq \lambda_0$.
\\
If, in particular, $u$ belongs to the subspace of $W^{1,A}(\rn)$ of those functions  such that
\begin{equation}\label{nov1}
\int_{\rn} A\Big(\frac{|u|}{\lambda}\Big)\,dx + \int_{\rn} A\Big(\frac{|\nabla u|}{\lambda}\Big)\,dx < \infty
\end{equation}
for every $\lambda >0$,
then equation \eqref{na1} also holds for every $\lambda>0$. Recall that the subspace of functions $u$ fulfilling \eqref{nov1} agrees with the closure of $C^\infty_0(\rn)$ in $W^{1,A}(\rn)$. It coincides with the whole of $W^{1,A}(\rn)$ if and only if $A$ fulfills the so called $\Delta_2$-condition.
\\ Embedding \eqref{incl1} and an analogue of equation  \eqref{na1} hold with $\rn$ replaced by any bounded Lipschitz domain.
\\ As a consequence of embedding \eqref{incl1}, one also has that
\begin{equation}
    \label{higherincl1}
    W^{[s]+1,A}(\rn) \to W^{s,A}(\rn)
\end{equation}
for every $s \in (0,\infty) \setminus \N$.
\\
In the classical case when $A(t)=t^p$ for some $p\geq 1$,  embedding \eqref{incl1} and   equation  \eqref{na1} have been established in \cite{BBM}.  For functions $A$ satisfying the $\Delta_2$-condition and with the function  $\overline A$ in a somewhat implicit form, they are proved in \cite{BonderSalort}. The present general version can be found in \cite{ACPS_limit1}.
\par
We conclude this section with  a fractional-order P\'olya--Szeg\H{o} principle on the decrease of the functional \eqref{intro1} under symmetric rearrangement of  functions $u$. Recall that the symmetric rearrangement $u^\bigstar$ of a function $u \in \mathcal M_d (\rn)$ is defined as
$$u^\bigstar (x) = u^*(\omega_n |x|^n) \quad \text{for $x \in \rn$,}$$
where $\omega_n$ denotes the Lebesgue measure of the unit ball in $\rn$. Thus, $u^\bigstar$ is radially decreasing about $0$ and is equidistributed with $u$.
The fractional P\'olya--Szeg\H{o} principle goes back to \cite{almgren-lieb,Bae} in the case when $A$ is a power. The result for Young functions $A$ satisfying the $\Delta_2$-condition and functions $u \in W^{s,A}(\rn)$ is the subject of \cite{deNBS}. The general version stated in Theorem \ref{P:L1} below can be proved via the same route. The necessary variant is sketched after its statement.

\begin{theorem}{\rm{\bf [Fractional
P\'olya--Szeg\H{o} principle]}} \label{P:L1}
Let  $s \in (0,1)$ and let $A$ be a Young function. Assume that $u\in\mathcal M_d (\rn)$.
Then
\begin{equation}\label{E:28'}
	\int_{\R^n}\int_{\R^n} A\left(\frac{|u(x)-u(y)|}{|x-y|^{s}}\right)\frac{dx\,dy}{|x-y|^n}
		\ge \int_{\R^n}\int_{\R^n} A\left(\frac{|u^{\bigstar}(x)-u^{\bigstar}(y)|}{|x-y|^{s}}\right)\frac{dx\,dy}{|x-y|^n}.
\end{equation}
\end{theorem}

\begin{proof}[Sketch of proof]
The proof follows along the same lines as that of~\cite[Theorem~3.7]{deNBS}. One step of the proof of that theorem requires that $u$ be approximated by a subsequence  of polarizations of $u$ that converges to $u^{\bigstar}$  a.e. in $\rn$. This is guaranteed if $u$ is just nonnegative and belongs to the space $\mathcal M_d (\o)$. Actually,   as observed in \cite[Section 4.1]{VS}, under these assumptions, there exists a sequence of polarizations of $u$ (with respect to a sequence of hyperplanes independent of $u$) that converges to $u^\bigstar$ in measure. The assumption that $u$ be nonnegative is not a restriction, since
\begin{equation}\label{aug21}\int_{\R^n}\int_{\R^n} A\left(\frac{|u(x)-u(y)|}{|x-y|^{s}}\right)\frac{dx\,dy}{|x-y|^n} \geq \int_{\R^n}\int_{\R^n} A\left(\frac{||u(x)|-|u(y)||}{|x-y|^{s}}\right)\frac{dx\,dy}{|x-y|^n},
\end{equation}
and $|u|^\bigstar = u^\bigstar$.
\end{proof}

%%%%%%%%%%%%%%%%%%%%%%%%%%%%%%%%%%%%%%%%%%%%%%%%%%%%

\section{One-dimensional Hardy type  inequalities in Orlicz spaces}\label{1d}

The results recalled in the first part of this section concern optimal target norms in inequalities for the integral operator $T_s$ defined, for $n\in\N$,  $s\in(0,n)$ and $L\in (0, \infty]$, as
\begin{equation}\label{T}
    T_sf(r)=\int_{r}^{L}\varrho^{-1+\frac{s}{n}}f(\varrho)\,d\varrho \quad\text{for $r\in [0,L]$,}
\end{equation}
for $f\in\M_+(0,L)$.
\par We begin with the optimal Orlicz target space corresponding to an Orlicz domain space $L^A(0,L)$, where
 $A$ is a~Young function such that
\begin{equation}\label{E:0'}
	\int^{\infty}\left(\frac{t}{A(t)}\right)^{\frac{s}{n-s}}\,dt = \infty
\end{equation}
and
\begin{equation}\label{E:0''}
	\int_{0}\left(\frac{t}{A(t)}\right)^{\frac{s}{n-s}}\,dt < \infty.
\end{equation}	
Such an Orlicz target is defined in terms of the Young function
 $A_{\frac{n}{s}}$ given by
\begin{equation}\label{Ans}
A_{\frac{n}{s}} (t) = A(H^{-1}(t)) \quad \text{for $t\geq 0$,}
\end{equation}
where
\begin{equation}\label{H}
H(t) = \bigg(\int _0^t \bigg(\frac \tau{A(\tau)}\bigg)^{\frac
{s}{n-s}} d\tau\bigg)^{\frac {n-s}n} \quad \text{for $t \geq0$.}
\end{equation}

\par\noindent {\bf Theorem A.}
\, %\cite{cianchi_CPDE}
\emph{Let $n\in\N$, $s\in(0,n)$, $L\in (0, \infty]$. Assume that $A$ is a~Young function  fulfilling  conditions \eqref{E:0'} and  \eqref{E:0''}. Let $A_{\frac ns}$ be the Young function defined by \eqref{Ans}. Then there exists a constant $C=C(\frac ns)$ such that
\begin{equation}\label{aug230}
	\bigg\|\int_{r}^{L}f(\varrho)\varrho^{-1+\frac{s}{n}}\,d\varrho \bigg\|_{L^{A_{\frac ns}}(0, L)}
		\leq C\|f\|_{L^A(0, L)}
\end{equation}
for every function $f\in L^A(0, L)$. Moreover, $L^{A_{\frac ns}}(0, L)$ is the optimal Orlicz target space in \eqref{aug230}.}

\medskip
\par\noindent
{\rm Inequality \eqref{aug230} is equivalent to \cite[inequality (2.7)]{cianchi_CPDE}, with $n$ replaced by $n/s$. The optimality of the space
$L^{A_{\frac ns}}(0, \infty)$ follows from \cite[Lemma 1]{cianchi_IUMJ}, where such an optimality is proved  with $A_{\frac ns}$ replaced by an equivalent Young function. Such an equivalence is shown in \cite[Lemma 2]{cianchi_pacific}. }

\smallskip
\par
We next focus an inequality parallel to \eqref{aug230}, but with a target space which is optimal among all rearrangement-invariant spaces.
Let $n$, $s$, $L$ and $A$ be as in Theorem A.  Denote by  $\widehat A$ the Young function given by
\begin{equation}\label{E:1}
	\widehat A (t)=\int_0^t\widehat a (\tau)\,d\tau\quad\text{for $t\in[0,\infty)$},
\end{equation}
where
\begin{equation}\label{E:2}
	{\widehat a\,}^{-1}(r) = \left(\int_{a^{-1}(r)}^{\infty}
		\left(\int_0^t\left(\frac{1}{a(\varrho)}\right)^{\frac{s}{n-s}}\,d\varrho\right)^{-\frac{n}{s}}\frac{dt}{a(t)^{\frac{n}{n-s}}}
				\right)^{\frac{s}{s-n}}
					\quad\text{for $r\ge0$}.
\end{equation}
Let $L(\widehat A,\frac{n}{s})(0,L)$ be the  Orlicz-Lorentz space equipped with the norm   defined as in  \eqref{aug300}, namely as
\begin{equation}\label{E:29}
	\|f\|_{L(\widehat A,\frac{n}{s})(0, L)}
		= \|r^{-\frac{s}{n}}f^{*}(r)\|_{L^{\widehat A}(0, L)}
\end{equation}
for $f \in \mathcal M(0,L)$.
Assumption \eqref{E:0''} ensures that condition \eqref{aug310}
is certainly fulfilled with $A$ replaced by $\widehat A$ and $q$ by $\frac ns$.
This is a consequence of
\cite[Propositions 2.1 and 2.2]{cianchi-ibero}.
Thus, $L(\widehat A,\frac{n}{s})(0,L)$ is actually a rearrangement-invariant space.
\\ By \cite[Lemma 2.3]{cianchi-ibero},
assumption \eqref{E:0''} is equivalent to condition \eqref{jan30}, with $A$ replaced by $\widetilde A$ and $q$ by $-\frac ns$. Hence, the space  $L[\widetilde A, -\frac ns](0,L)$, endowed with the norm defined as in \eqref{aug301} by
\begin{equation}\label{jan90}
	\|f\|_{L[\widetilde A, -\frac ns](0,L)}
		= \|r^{\frac{s}{n}}f^{**}(r)\|_{L^{\widehat A}(0, L)}
\end{equation}
for $f \in \mathcal M(0,L)$,
is also a
rearrangement-invariant space.
Moreover,
one has that
\begin{equation}\label{E:36}
	L[\widetilde A,-\tfrac{n}{s}](0,L) = L({\widehat A},\tfrac{n}{s})'(0,L),
\end{equation}
up to equivalent norms,
where $ L({\widehat A},\frac{n}{s})'(0,L)$ denotes the associate space of
  $L({\widehat A},\frac{n}{s})(0,L)$. Property \eqref{E:36} is  stated and established in \cite[Lemma 4.5]{cianchi_MZ} for $L<\infty$; the proof in the case when $L=\infty$ is completely analogous.

\bigskip
\par\noindent {\bf Theorem B.}
\, %\cite{cianchi-ibero}
\emph{Let $n\in\N$, $s\in(0,n)$, $L\in (0, \infty]$. Assume that $A$ is a~Young function  fulfilling  conditions \eqref{E:0'} and  \eqref{E:0''}, and  let ${\widehat A}$ be the Young function defined by \eqref{E:1}. Then there exists a constant $C=C(\frac ns)$ such that
\begin{equation}\label{aug231}
	\bigg\|\int_{r}^{L}f(\varrho)\varrho^{-1+\frac{s}{n}}\,d\varrho \bigg\|_{L({\widehat A}, \frac ns)(0, L)}
		\leq C\|f\|_{L^A(0, L)}
\end{equation}
for every function $f\in L^A(0, L)$. Moreover, $L({\widehat A}, \frac ns)(0, L)$ is the optimal rearrangement-invariant target space in \eqref{aug231}.}

\bigskip
\par \noindent
{\rm
Inequality \eqref{aug231}
 agrees with inequality \cite[inequality (3.1)]{cianchi-ibero}, with $n$ replaced with $n/s$. The optimality of the space $L({\widehat A}, \frac ns)(0, L)$ follows from   \cite[inequalities (4.6)--(4.8)]{cianchi-ibero}.}

\smallskip
\par
Let us notice that the function $A$ always dominates $\widehat A$. Moreover, $A$ is equivalent to  $\widehat A$  if and only if $A(t)$ grows less than the power $t^{\frac ns}$ in the sense that its Matuszewska-Orlicz index $I(A)$, defined by \eqref{index}, is smaller $\frac ns$. An analogous property holds near infinity in connection with the index $I_\infty (A)$. These assertions are proved in
 \cite[Propositions 5.1 and 5.2]{cianchi-ibero}, and collected in the following proposition.
\medskip

\par\noindent {\bf Proposition C.}  \,
%\cite{cianchi-ibero}%\begin{proposition}\label{L:3}
\emph{
Let $n\in\N$ and $s\in(0,n)$. Assume that  $A$ is a~Young function  fulfilling  conditions \eqref{E:0'} and  \eqref{E:0''}, and  let   $A_{\frac ns}$ be the Young function defined by \eqref{E:1}. Then there exists a~constant $c= c(n/s)$ such that
\begin{equation}\label{E:14'}
	{\widehat A}(t)\le A(ct)\quad \text{for $t>0$}.
\end{equation}
Moreover, the function ${\widehat A}$ is equivalent to $A$ globally {\rm [}near infinity{\rm ]} if and only if $I(A) < \frac ns$ {\rm [}$I_\infty(A) < \frac ns${\rm ]}.
}

%\begin{proof}
%The assertions of Proposition \ref{L:3}  are contained  in~\cite[Propositions~5.1 and 5.2]{cianchi-ibero}.
%\end{proof}

\medskip
\par\noindent
As a consequence of Proposition C, the space $L({\widehat A}, \frac ns)(0, L)$ reduces to  $L(A, \frac ns)(0, L)$ if $A(t)$ is subcritical with respect to $t^{\frac ns}$ in the sense of Matuszewska-Orlicz indices.

 \medskip
\par\noindent {\bf Proposition D.} %\label{L:3bis}
\emph{
Let $n\in\N$, $s\in(0,n)$. Assume that  $A$ is a~Young function  fulfilling  conditions \eqref{E:0'} and  \eqref{E:0''}.
\\ (i) If $I(A) < \frac ns$, then $ L({\widehat A}, \frac ns)(0, \infty) = L(A, \frac ns)(0, \infty)$, up to equivalent norms.
\\  (ii) If $I_\infty(A) < \frac ns$, then $ L({\widehat A}, \frac ns)(0, L) = L(A, \frac ns)(0, L)$, up to equivalent norms, for every $L>0$.
}

 \medskip
\par\noindent

The embedding of the space
$L({\widehat A},\tfrac{n}{s})(0,L)$ into $L^{A_{\frac{n}{s}}}(0,L)$ is a trivial consequence of the optimality of the former in inequality \eqref{aug231} in the class of all rearrangement-invariant spaces, which includes, in particular, the Orlicz spaces. This fact is stated in Proposition \ref{P:1} below. A direct proof of this proposition is however given, which shows that the norm of the embedding in question is independent of $A$ and $L$, a piece of information of use in the proofs of our main results.

\begin{proposition}\label{P:1}
Let $n\in\N$, $s\in(0,n)$ and  $L\in (0, \infty]$. Assume that $A$ is a~Young function  fulfilling  conditions \eqref{E:0'} and  \eqref{E:0''}. Let
$A_{\frac ns}$
and ${\widehat A}$ be  the Young functions defined as in \eqref{Ans} and \eqref{E:1}, respectively.
Then
\begin{equation}\label{E:35}
	L({\widehat A},\tfrac{n}{s})(0,L) \to L^{A_{\frac{n}{s}}}(0,L).
\end{equation}
Moreover, the norm of embedding \eqref{E:35} depends only on $\tfrac ns$.
\end{proposition}

\begin{proof}
Denote by $L[\widetilde A,-\frac{n}{s}](0,L)$ the Orlicz--Lorentz space endowed with the norm defined as in \eqref{aug301}.
%As a consequence of ~\cite[Theorem~4.2]{cianchi-ibero},
%One has that
%\begin{equation}\label{E:36}
%	L[\widetilde A,-\tfrac{n}{s}](0,L) = L({\widehat A},\tfrac{n}{s})'(0,L),
%\end{equation}
%up to equivalent norms,
%where $ L({\widehat A},\frac{n}{s})'(0,L)$ denotes the associate space of
 % $L({\widehat A},\frac{n}{s})(0,L)$. Property \eqref{E:36} is  stated and established in \cite[Lemma 4.5]{cianchi_MZ} for $L<\infty$; the proof in the case when $L=\infty$ is completely analogous.
Thanks to equation \eqref{E:36} and to
property \eqref{nov103}, embedding~\eqref{E:35} is equivalent to
\begin{equation}\label{E:37}
	\big(L^{A_{\frac{n}{s}}}\big)' (0,L)\to  L({\widehat A}, \tfrac{n}{s})'(0,L),
\end{equation}
with the same embedding constants.
Also, by property \eqref{nov104},
\begin{equation}\label{nov102}
	\big(L^{A_{\frac{n}{s}}}\big)'(0,L) = L^{\widetilde{A_{\frac{n}{s}}}}(0,L),
\end{equation}
up to equivalent norms, with absolute equivalence constants. Owing to equations
by~\eqref{E:36} --\eqref{nov102}, embedding~\eqref{E:35} will follow if we show that
\begin{equation}\label{E:38}
	L^{\widetilde{A_{\frac{n}{s}}}}(0,L) \to L[\widetilde A,-\tfrac{n}{s}](0,L).
\end{equation}
Embedding~\eqref{E:38} is equivalent to the inequality
\begin{equation}\label{E:39}
	\|r^{\frac{s}{n}}f^{**}(r)\|_{L^{\widetilde A}(0,L)}
		\le C \|f^{*}\|_{L^{\widetilde{A_{\frac{n}{s}}}}(0,L)}
			\end{equation}
		for some constant $C$ and for every function $f\in\mathcal M_+  (0,L)$. Moreover, the norm of embedding \eqref{E:38} equals the optimal constant $C$ in inequality \eqref{E:39}.
 Inequality
~\eqref{E:39} is in its turn equivalent to the inequality
\begin{equation}\label{E:40}
	\left\|\int_{r}^{L}g(\varrho)\varrho^{-1+\frac{s}{n}}\,d\varrho\right\|_{L^{{A_{\frac{n}{s}}}}(0,L)}
		\le C' \|g\|_{L^{A}(0,L)}
\end{equation}
for every function $g\in\mathcal M_+  (0,L)$, and for some constant $C'$ equivalent to $C$, up to absolute multiplicative constants.  Indeed,
\begin{align*}
\sup_{f \in L^{\widetilde{A_{\frac{n}{s}}}}(0,L)} &
\frac{\|r^{\frac{s}{n}}f^{**}(r)\|_{L^{\widetilde A}(0,L)}}
{ \|f^{*}\|_{L^{\widetilde{A_{\frac{n}{s}}}}(0,L)}}  \approx
\sup _{f \in L^{\widetilde{A_{\frac{n}{s}}}}(0,L)}  \sup_{g \in L^A(0,L)} \frac{\int_0^Lr^{\frac{s}{n}}f^{**}(r) g(r)\, dr}{\|g\|_{L^A(0,L)} \|f^{*}\|_{L^{\widetilde{A_{\frac{n}{s}}}}(0,L)}}
\\ \nonumber & =
 \sup_{g \in L^A(0,L)}\sup _{f \in L^{\widetilde{A_{\frac{n}{s}}}}(0,L)}  \frac{\int_0^Lf^{*}(r) \int_r ^L\varrho^{-1+\frac{s}{n}}g(\varrho)\, d\varrho\, dr}{\|g\|_{L^A(0,L)} \|f^{*}\|_{L^{\widetilde{A_{\frac{n}{s}}}}(0,L)}}
\approx  \sup_{g \in L^A(0,L)} \frac{\left\|\int_{r}^{L}g(\varrho)\varrho^{-1+\frac{s}{n}}\,d\varrho\right\|_{L^{{A_{\frac{n}{s}}}}(0,L)}}{\|g\|_{L^{A}(0,L)}},
\end{align*}
where the relations \lq\lq$\approx$" hold up to absolute multiplicative constants.
Inequality~\eqref{E:40} is nothing but \eqref{aug230}, and hence the latter holds with a constant $C'$ depending only on $\tfrac ns$.
%
% is proved in~\cite{cianchi_CPDE}, more precisely it follows from inequality (2.1) of that paper with $n$ replaced with $\frac{n}{s}$.
\end{proof}

A characterization of the optimal rearrangement-invariant target space $X_s(0,L)$, corresponding to any given rearrangement-invariant domain space $X(0,L)$, for the operator $T_s$ defined by \eqref{T} is contained in the next result. Notice that, in view of Theorem B, if $X(0,L)=L^A(0,L)$, then $X_s(0,L)= L(\widehat A, s)$.

\medskip
\par\noindent  {\bf Theorem E.}\,
{\emph Let  $s\in(0,n)$ and $L\in (0, \infty]$.
Let  $\|\cdot\|_{X(0,L)}$ be a rearrangement-invariant norm. If $L=\infty$, assume, in addition, that
\begin{equation}\label{E:nontriv}
    \|(1+r)^{-1+\frac{s}{n}}\|_{X'(0,\infty)}< \infty.
\end{equation}
(i) Define the functional $\|\cdot\|_{Z(0,L)}$ as
\begin{equation}\label{E:sigma}
    \|f\|_{Z(0,L)} =\left\|r^{\frac{s}{n}}f^{**}(r)\right\|_{X'(0,L)}
\end{equation}
for $f\in\M(0,L)$.
%If
%\begin{equation}\label{E:nontriv}
%    (1+r)^{-1+\frac{s}{n}} \in X'(0,L),
%\end{equation}
Then $\|\cdot\|_{Z(0,L)}$ is a~rearrangement-invariant~norm on $(0,L)$. Denote by $\|\cdot\|_{X_{s}(0,L)}$ the norm defined by
\begin{equation*}
    \|f\|_{X_{s}(0,L)}= \|f\|_{Z'(0,L)}\
\end{equation*}
for $f\in\M(0,L)$.
Then
\begin{equation}\label{E:hardy}
\bigg\|\int_{r}^{L}\varrho^{-1+\frac{s}{n}}f(\varrho)\,d\varrho \bigg\|_{X_s(0,L)} \leq C \|f\|_{X(0,L)}
\end{equation}
for every $f \in X(0,L)$.
Moreover, $X_{s}(0,L)$ is the optimal rearrangement-invariant target space in \eqref{E:hardy}.
\\
(ii) Let $s_1, s_2>0$ be such that $s_1+s_2<n$. Then
\begin{equation}\label{E:iteration}
    (X_{s_1})_{s_2}(0,L)=X_{s_1+s_2}(0,L).
\end{equation}}

\par\noindent
In the case when $L<\infty$, Part (i) of Theorem E is established in  \cite[Theorem~4.5]{EKP} and Part (ii) in \cite[Theorem~3.5]{CP-Trans}. The proof for $L=\infty$ follows along the same lines; the result of Part (i)  is stated in \cite[Theorem~4.4]{EMMP} and that of Part (ii) in \cite[Proposition~4.3]{Mih}.

\begin{remark}  Condition \eqref{E:nontriv} is indispensable if $L=\infty$. Indeed, if \eqref{E:nontriv} fails, then the functional $\|\cdot\|_{Z(0,L)}$ given by ~\eqref{E:sigma} is trivial in the sense that $\|f\|_{Z(0,L)}<\infty$ only if $f=0$. Furthermore, the  inequality
\begin{equation}\label{E:hardyfail}
\bigg\|\int_{r}^{\infty}\varrho^{-1+\frac{s}{n}}f(\varrho)\,d\varrho \bigg\|_{Y(0,\infty)} \leq C \|f\|_{X(0,\infty)}
\end{equation}
does not hold for any rearrangement-invariant  space $Y(0,\infty)$.
\end{remark}

The last two results of this section are new. They amount to Hardy type inequalities in integral form in Orlicz spaces. Of course, they can be  equivalently stated in the form of the boundedness of suitable integral operators in the relevant Orlicz spaces.

\begin{lemma}\label{L:1}
Let $s\in(0,1)$ and let $L\in(0,\infty]$. Assume that $A$ is any Young function. Then
\begin{equation}\label{E:4}
	\int_0^L A\left(r^{-s}\int_0^{{r}}f(\varrho)\,d\varrho \right)\,\frac{dr}{r}
		\le
	\int_0^L A\left(\frac{1}{s}r^{1-s}f(r)\right)\,\frac{dr}{r}
\end{equation}
for every function $f\in \mathcal M_+ (0,L)$.
\end{lemma}

\begin{proof}
The change of variables $r=e^{-\xi}$ and $\varrho =e^{-\eta}$ turns inequality~\eqref{E:4} into
%\todo[inline]{A: I would change $\theta$ to another letter}
%\todo[inline]{Lubos: $\theta$ changed to $\xi$}
\begin{equation*}
	\int_{\log\frac{1}{L}}^{\infty}A\left(e^{s\xi}\int_{\xi}^{\infty}f(e^{-\eta})e^{-\eta}\,d\eta\right)\,d\xi
		\le
	\int_{\log\frac{1}{L}}^{\infty}A\left(\frac{1}{s}e^{(s-1)\xi}f(e^{-\xi})\right)\, d\xi.
\end{equation*}
Of course, $\log\frac{1}{L}$ has to be understood as $-\infty$ if $L=\infty$. On setting
\begin{equation*}
	g(\xi)=e^{(s-1)\xi}f(e^{-\xi})\quad\text{for $\xi >\log\tfrac{1}{L}$,}
\end{equation*}
the last inequality can be rewritten as
\begin{equation}\label{E:5}
	\int_{\log\frac{1}{L}}^{\infty}A\left(e^{s\xi}\int_{\xi}^{\infty}g(\eta)e^{-s\eta}\,d\eta\right)\,d\xi
		\le
	\int_{\log\frac{1}{L}}^{\infty}A\left(\frac{1}{s}g(\xi)\right)\, d\xi.
\end{equation}
Define the operator
\begin{equation*}
	Hg(\xi)=e^{s\xi}\int_{\xi}^{\infty}g(\eta)e^{-s\eta}\,d\eta \quad \text{for $\xi >\log\tfrac{1}{L}$,}
\end{equation*}
for $g\in\M_+\left(\log \tfrac{1}{L},\infty\right)$.
We claim that $H\colon L^{\infty}\left(\log \tfrac{1}{L},\infty\right) \to L^{\infty}\left(\log \tfrac{1}{L},\infty\right)$,
%\begin{equation*}
%	H\colon L^{\infty}\left(\log \tfrac{1}{L},\infty\right) \to L^{\infty}\left(\log \tfrac{1}{L},\infty\right)
%\end{equation*}
and
\begin{equation}\label{E:6}
	\|H\|_{L^{\infty}\to L^{\infty}}\le\frac{1}{s}.
\end{equation}
Indeed,
\begin{align*}
	\|Hg\|_{L^{\infty}\left(\log \frac{1}{L},\infty\right)}
		& =
			\sup_{t\in\left(\log \frac{1}{L},\infty\right)} e^{s\xi}\int_{\xi}^{\infty}g(\eta)e^{-s\eta}\,d\eta
			\\
		& \le
			\|g\|_{L^{\infty}\left(\log \frac{1}{L},\infty\right)} \sup_{\xi\in\left(\log \frac{1}{L},\infty\right)}
			e^{s\xi}\left[\frac{e^{-\eta s}}{-s}\right]_{\eta=\xi}^{\eta=\infty}
			= \frac{1}{s}\|g\|_{L^{\infty}\left(\log \frac{1}{L},\infty\right)}
\end{align*}
for $g \in L^{\infty}\left(\log \frac{1}{L},\infty\right)$.
Moreover, $H\colon L^{1}\left(\log \tfrac{1}{L},\infty\right) \to L^{1}\left(\log \tfrac{1}{L},\infty\right)$,
%\begin{equation*}
%	H\colon L^{1}\left(\log \tfrac{1}{L},\infty\right) \to L^{1}\left(\log \tfrac{1}{L},\infty\right)
%\end{equation*}
and
\begin{equation}\label{E:7}
	\|H\|_{L^{1}\to L^1}\le\frac{1}{s},
\end{equation}
inasmuch as
\begin{align*}
	\|Hg\|_{L^{1}\left(\log \frac{1}{L},\infty\right)}
		& =
			\int_{\log \frac{1}{L}}^{\infty} e^{s\xi}\int_{\xi}^{\infty}g(\eta)e^{-s\eta}\,d\eta\,d\xi
			 =
			\int_{\log\frac{1}{L}}^{\infty}g(\eta)e^{-\eta s}\int_{\log\frac{1}{L}}^{\eta}e^{s\xi}\,d\xi\,d\eta
			\\
		&
%= \int_{\log\frac{1}{L}}^{\infty}g(\eta)e^{-\eta s}\left(\frac{e^{\eta s}-e^{s\log\frac{1}{L}}}{s}\right)\,d\eta
			 \le \frac{1}{s} \int_{\log\frac{1}{L}}^{\infty}g(\eta)\,d\eta
			= \frac{1}{s}\|g\|_{L^1\left(\log \frac{1}{L},\infty\right)}
\end{align*}
for $g \in L^1\left(\log \frac{1}{L},\infty\right)$.
From~\eqref{E:6}--\eqref{E:7}, via an interpolation theorem by  Calder\'on \cite[Chapter 3, Theorem~2.12]{BS}, one can infer that
\begin{equation}\label{E:8}
	H\colon L^{A}\left(\log \tfrac{1}{L},\infty\right) \to L^{A}\left(\log \tfrac{1}{L},\infty\right)
\end{equation}
with
\begin{equation}\label{nov111}
	\|H\|_{L^{A}\to L^A}\le\frac{1}{s}.
\end{equation}
Notice that in deducing \eqref{E:8} and \eqref{nov111} one makes use of the fact that $L^{A}\left(\log \tfrac{1}{L},\infty\right)$ is a rearrangement-invariant space, and hence an exact interpolation space between $L^1\left(\log \tfrac{1}{L},\infty\right)$ and $L^\infty\left(\log \tfrac{1}{L},\infty\right)$ (see~\cite[Chapter~3, Theorem~2.12]{BS}). Therefore,
\begin{equation}\label{E:9}
	\|Hg\|_{L^{A}\left(\log \frac{1}{L},\infty\right)}
		\le \left\|\frac{1}{s}g\right\|_{L^{A}\left(\log \frac{1}{L},\infty\right)}
\end{equation}
for every $g\in L^{A}\left(\log \tfrac{1}{L},\infty\right)$.
An application of inequality~\eqref{E:9} with $A$ replaced by $A_M$, defined by $A_M(t)=\frac{A(t)}{M}$ for $t>0$, where
\begin{equation*}
	M = \int_{\log\frac{1}{L}}^{\infty}A\Big(\frac{1}{s}g(r)\Big)\,dr,
\end{equation*}
yields~\eqref{E:5} via the definition of Luxemburg norm.
\end{proof}

\begin{lemma}\label{L:2}
Let $n\in\N$, $s\in(0,n)$, $L\in (0, \infty]$. Assume that $A$ is a~Young function  fulfilling  conditions \eqref{E:0'} and  \eqref{E:0''}, and  let ${\widehat A}$ be the Young function defined by \eqref{E:1}. Then there exists a~constant  $C= C(n/s)$ such that
\begin{equation}\label{E:10}
	\int_{0}^{L}{\widehat A}\left(r^{-s}\int_{r}^{L}f(\varrho)\,d\varrho\right) r^{n-1}\, dr
		\le
	\int_{0}^{L} A\left(Cr^{1-s}f(r)\right)r^{n-1}\,dr
\end{equation}
for every function $f \in \mathcal M_+(0, L)$. Moreover, $\lim _{s\to s_0} C(n/s)<\infty$ for any $s_0 \in [0,n)$.
\end{lemma}
%\todo[inline]{Lenka: Is the limit at $1$, or at $n$?}
%\todo[inline]{A: the limit is at any value smaller than $n$}
\begin{proof} On replacing, if necessary, $f$ by $f\chi _{(0,L)}$, it suffices consider the case when $L=\infty$.
The change of variables $t=r^n$ and $\tau=\varrho^n$ turns inequality~\eqref{E:10} into
\begin{equation}\label{E:11}
	\int_{0}^{\infty}{\widehat A}\left(\frac{\xi^{-\frac{s}{n}}}{n}
		\int_{\xi}^{\infty}f(\eta^{\frac{1}{n}})\eta^{-\frac{1}{n'}}\,d\eta\right)\,d\xi
		\le \int_0^{\infty} A\left(C \xi^{\frac{1-s}{n}}f(\xi^{\frac{1}{n}})\right)\,d\xi.
\end{equation}
On setting
\begin{equation*}
	g(\xi) = \xi^{\frac{1-s}{n}}f(\xi^{\frac{1}{n}}) \quad \text{for $\xi>0$},
\end{equation*}
inequality~\eqref{E:11} reads
\begin{equation}\label{E:12}
	\int_0^{\infty}{\widehat A}\left(\frac{\xi^{-\frac{s}{n}}}{n}\int_{\xi}^{\infty}g(\eta)\eta^{-1+\frac{s}{n}}\,d\eta\right)\,d\xi
		\le \int_{0}^{\infty} A(Cg(\xi))\,d\xi
\end{equation}
for $g \in \mathcal M_+(0,\infty)$.
Denote by $T$ the operator defined as
\begin{equation*}
	Tg(\xi)=\xi^{-\frac{s}{n}}\int_{\xi}^{\infty}g(\eta)\eta^{-1+\frac{s}{n}}\,d\eta
		\quad\text{for $\xi>0$,}
\end{equation*}
for $g\in\M_+(0,\infty)$.
We have that $T\colon L^1(0,\infty)\to L^1(0,\infty)$,
with
\begin{equation}\label{E:13}
	\|T\|_{L^1\to L^1}\le\frac{n}{n-s},
\end{equation}
inasmuch as
\begin{align*}
	\|Tg\|_{L^1(0,\infty)} &= \int_{0}^{\infty} \xi^{-\frac{s}{n}}\int_{\xi}^{\infty}g(\eta)\eta^{-1+\frac{s}{n}}\,d\eta \,d\xi
			 = \int_{0}^{\infty}g(\eta)\eta^{-1+\frac{s}{n}}\int_{0}^{\eta}\xi^{-\frac{s}{n}}\,d\xi\,d\eta
			= \frac{n}{n-s} \|g\|_{L^1(0,\infty)}
\end{align*}
for $g \in L^1(0,\infty)$.
Also, $T\colon L^{\frac{n}{s},1}(0,\infty)\to L^{\frac{n}{s},\infty}(0,\infty)$,
with
\begin{equation}\label{E:14}
	\|T\|_{L^{\frac{n}{s},1}\to L^{\frac ns,\infty}}\le1.
\end{equation}
Actually,
\begin{align*}
	\|Tg\|_{L^{\frac{n}{s},\infty}(0,\infty)}
		&= \sup_{\xi\in(0,\infty)} \xi^{\frac{s}{n}}(Tg)^*(\xi)
			= \sup_{\xi\in(0,\infty)} \xi^{\frac{s}{n}}Tg(\xi)
				= \int_{0}^{\infty} g(\eta)\eta^{-1+\frac{s}{n}}\,d\eta
					\le \|g\|_{L^{\frac{n}{s},1}(0,\infty)}
\end{align*}
for $g \in L^{\frac{n}{s},1}(0,\infty)$,
where the second  equality follows from the fact that $(Tg)^*=Tg$, for $Tg$ is a~decreasing function,  and the inequality follows from the Hardy--Littlewood inequality, since the function $\eta\mapsto \eta^{-1+\frac{s}{n}}$ is decreasing. Owing to the boundedness  properties \eqref{E:13} and \eqref{E:14} of the operator $T$, inequality~\eqref{E:12} follows from  \cite[Theorem~3.1]{cianchi-ibero}. The finiteness of the limit of the constant $C(n/s)$ can be checked via a close inspection of the proof of that theorem.
%\todo[inline]{A: Check the dependence of the constant in Theorem~3.1}
\end{proof}

\section{A fractional Hardy type inequality: case $s \in (0,1)$}\label{hardy}

%In this section we prove key results that will be used in the proofs of the main results, namely the appropriate versions of the hardy inequality and of the P\'olya--Szeg\H{o} principle.

%\subsection{The Hardy inequalities}
This section is devoted to a proof of the Hardy type inequality stated below. Apart from its own interest, it is critical in our approach to the other main results of this paper.

\begin{theorem}{\rm{\bf [Fractional Orlicz--Hardy inequality]}}\label{T:1}
Let $n\in\N$ and $s\in(0,1)$. Assume that  $A$ is a~Young function satisfying conditions \eqref{E:0'} and \eqref{E:0''} and let ${\widehat A}$ be the Young function given by \eqref{E:1}.
 Then, there exists a~constant
$C=C(n,s)$
such that
\begin{equation}\label{E:3bis}
	\left\|\frac{|u(x)|}{|x|^s}\right\|_{L^{\widehat A}(\rn)} \leq C
		|u|_{s,A,\rn}
\end{equation}
for every   function $u \in   V^{s,A}_d (\rn)$. Moreover,
$\lim _{s\to 1^-}C(n,s)< \infty$. In particular,
%\todo[inline]{A: As a further result, we have to try to imitate Shaposhnikova-Mazya and prove an optimal dependence on $s$ to pass to the limit as $s\to 0$.}
\begin{equation}\label{E:3}
	\int_{\R^n}{\widehat A}\left(\frac{|u(x)|}{|x|^s}\right)\,dx
		\le (1-s) \int_{\R^n} \int_{\R^n}
			A\left(C\frac{|u(x)-u(y)|}{|x-y|^s}\right)\,\frac{dx\, dy}{|x-y|^n}
\end{equation}
for every   function $u \in \mathcal M_d (\rn)$.
\end{theorem}

The following example provides us with an application of Theorem \ref{T:1} to a Young function whose behaviour near zero and near infinity is of power-logarithmic type. Although quite simple, this model Young function enables us to recover the results known until now and to exhibit genuinely new inequalities. This model function will  also be called into play    in order to illustrate the results of the next sections.

\begin{example}\label{exhardy}
Consider a Young function $A$ such that
\begin{equation}\label{dec250}
A(t) \,\,\text{is equivalent to}\,\, \begin{cases} t^{p_0} (\log \frac 1t)^{\alpha_0} & \quad \text{near zero}
\\
t^p  (\log t)^\alpha & \quad \text{near infinity,}
\end{cases}
\end{equation}
where either $p_0>1$ and $\alpha_0 \in \R$, or $p_0=1$ and $\alpha_0 \leq 0$, and either $p>1$ and $\alpha \in \R$ or $p=1$ and $\alpha \geq 0$. Here, equivalence is meant in the sense of Young functions.
\\ Let $n \in \N$ and $s\in (0,1)$. The function $A$ satisfies assumption \eqref{E:0'} if
\begin{equation}\label{dec251}
\text{either $1\leq p< \frac ns$ and $\alpha$ is as above, or $p=\frac ns$ and $\alpha \leq \frac ns -1$,}
\end{equation}
and satisfies assumption \eqref{E:0''} if
\begin{equation}\label{dec252}
\text{either $1\leq p_0< \frac ns$ and $\alpha_0$ is as above, or $p_0=\frac ns$ and $\alpha_0 > \frac ns -1$.}
\end{equation}
Theorem \ref{T:1} tells us that,
under assumptions \eqref{dec251} and \eqref{dec252}, inequalities \eqref{E:3bis} and \eqref{E:3} hold if
\begin{equation}\label{dec253}
\widehat A(t)\,\,  \text{is equivalent to}\,\, \begin{cases} t^{p_0}(\log \frac 1t)^{\alpha_0} & \quad \text{ if $1\leq {p_0}< \frac ns$ }
\\
t^{\frac ns}  (\log \frac 1t)^{\alpha_0 - \frac ns} & \quad  \text{if ${p_0}=\frac ns$ and $\alpha_0> \frac ns -1$}
\end{cases} \quad \text{near zero,}
\end{equation}
and
\begin{equation}\label{dec254}
\widehat A(t) \,\, \text{is equivalent to} \,\, \begin{cases} t^p (\log t)^\alpha & \quad  \text{ if $1\leq p< \frac ns$ }
\\
t^{\frac ns}  (\log t)^{\alpha - \frac ns} &  \quad \text{if  $p=\frac ns$ and $\alpha < \frac ns -1$}
\\
t^{\frac ns}  (\log t)^{-1} (\log (\log t))^{-\frac ns}&  \quad \text{if  $p=\frac ns$ and $\alpha = \frac ns -1$}
\end{cases} \quad \text{near infinity.}
\end{equation}
In particular, the   choices ${p_0}=p<\tfrac ns$ and $\alpha_0 =\alpha=0$ yield $\widehat A(t) =t^p$, and inequalities \eqref{E:3bis} and \eqref{E:3} recover (apart from the specific form of the constant involved) \cite[Inequality (3)]{MaSh1}.
\end{example}

\medskip

The following preliminary results will be of use in  the proof of Theorem~\ref{T:1}.

\begin{proposition}\label{poinc0}
Let $s \in (0,1)$ and let $A$ be a Young function. Assume that   $u\in V^{s,A}(\rn)$ and
$|\{|u|>0\}|<\infty$. Then there exists a constant
 $C=C(n,s, |\{|u|>0\}|)$ such that
\begin{equation}\label{poinc01}
\|u\|_{L^A(\rn)} \leq C |u|_{s,A, \rn}.
\end{equation}
\end{proposition}
\begin{proof}
Let us set $U=\{u^\bigstar>0\}$ and  $d= {\rm diam}(U)$.  By Theorem \ref{P:L1}, given any $\lambda >0$, we have that
\begin{align}\label{poinc1}
\int_{\rn} \int_{\rn}A\left(\frac{|u(x)-u(y)|}{\lambda  |x-y|^s}\right)\frac{\,dx\,dy}{|x-y|^n}&\geq
\int_{\rn} \int_{\rn}A\left(\frac{|u^\bigstar (x)-u^\bigstar(y)|}{\lambda |x-y|^s}\right)\frac{\,dx\,dy}{|x-y|^n}
\\ \nonumber &
\geq
\int_{U} \int_{\{y\in \rn:  2(d+1)\geq |x-y| \geq d+1\}}A\left(\frac{u^\bigstar (x)}{\lambda |x-y|^s}\right)\frac{\,dy\,dx}{|x-y|^n}
\\ \nonumber &=  \int_{U} \int_{\{y\in \rn:  2(d+1)\geq |x-y| \geq d+1\}}A\left(\frac{(2(d+1))^s}{\lambda |x-y|^s}\frac {u^\bigstar (x)}{(2(d+1))^s}\right)\frac{\,dy\,dx}{|x-y|^n} \\ \nonumber &
\geq \int _U A\left(\frac {u^\bigstar (x)}{\lambda (2(d+1))^s}\right)   \int_{\{y\in \rn:  2(d+1)\geq |x-y| \geq d+1\}}\frac{(2(d+1))^s}{|x-y|^{s+n}}\,dy\,dx
\\ \nonumber & =
\frac{c(2^s-1)}{s} \int _U A\left(\frac {u^\bigstar (x)}{\lambda(2(d+1))^s}\right)  \,dx
 \\ \nonumber & =
\frac{c(2^{s}-1)}{s} \int _{\rn} A\left(\frac {|u(x)|}{\lambda(2(d+1))^s}\right)  \,dx
\end{align}
for some positive constant $c=c(n)$.
Note that the third inequality holds, owing to property \eqref{kt}.
Hence, inequality \eqref{poinc01} follows.
\end{proof}

%\note[inline]{Lubos: it seems to me that in the last two lines of~\eqref{poinc1} the factor $(d+1)^s$ should disappear from the denominator.}

\begin{corollary}\label{P:aux} Let $s \in (0,1)$ and let $A$ be a Young function. Assume that  $u\in V^{s,A}_d(\rn)$.
% satisfies
%$|\{|u|>t\}|<\infty$ for every $t>0$.
% If
%$$\int_{\R^n}\int_{\R^n} A\left(\frac{|u(x)-u(y)|}{|x-y|^{s}}\right)\frac{dx\,dy}{|x-y|^n}< \infty,$$
Then
\begin{equation}\label{aug20}
\int _E |u|\, dx <\infty
%u \in L^1_{\rm loc} (\rn).
\end{equation}
for every set $E\subset \rn$ with $|E|<\infty$. In particular, $u \in L^1_{\rm loc} (\rn)$.
\end{corollary}
\begin{proof}
%Owing to inequality \eqref{aug21}, we may assume, without loss of generality, that $u \geq 0$.
Define the function $\overline u: \rn \to [0, \infty)$ as
$$\overline u= (|u|-1)\chi_{\{|u|>1\}}.$$
%Here, and in what follows, $\chi_E$ denotes the characteristic function of a set $E\subset\rn$.
One can verify that
$$|u(x) -u(y)| \geq |\overline u(x) - \overline u(y)| \quad \text{for $x, y \in \rn$.}$$
Thus, $\overline u \in V^{s,A}(\rn)$, and $|\{\overline u >0\}|<\infty$. By Proposition \ref{poinc0}, $\overline u \in L^A(\rn)$, and since $|\{\overline u >0\}|<\infty$ the second embedding in \eqref{B.3'} ensures that  $\overline u \in L^1(\rn)$. On the other hand, $|u| \leq \overline u +1$, whence
$$\int _{E} |u|\, dx \leq \int_{E} \overline u \, dx + |E| < \infty
$$
for  any set $E\subset \rn$ with $|E|< \infty$.
Hence, \eqref{aug20} follows.
\end{proof}

Our proof of Theorem~\ref{T:1} makes use of a classical approach in the theory of fractional Sobolev norms of functions in $\rn$, which consists in an extension of the relevant functions to $\mathbb R^{n+1}$.  In particular, we follow  the outline of the proof of \cite[Theorem 2]{MaSh1}. However, specific ad hoc Orlicz space techniques and sharp one-dimensional Hardy type inequalities in Orlicz spaces, presented in Section \ref{1d}, have to be exploited in the present setting.

\begin{proof}[Proof of Theorem~\ref{T:1}] Let $u \in \mathcal M_d(\rn)$. If the right-hand side of inequality \eqref{E:3} is infinite for a certain constant $C$ to be specified later, then the conclusion is trivially true. We may thus assume that it is finite. Hence, in particular, $u \in V^{s,A}_d(\rn)$.
%The outline of the proof is similar to that of~\cite[Theorem~2]{MaSh1}.  Let  $u: \rn \to \R$ be a measurable function satisfying condition \eqref{aug15} and making the right-hand side of inequality \eqref{E:3} finite.
Owing to Theorem  \ref{P:L1},       the integral on the right-hand side of inequality \eqref{E:3} is still finite   if $u$ is replaced by   $u^\bigstar$.   By Proposition \ref{P:aux},  applied with $u$ replaced by $u^\bigstar$, we have that $u^\bigstar \in L^1_{\rm loc}(\rn)$. Consequently,
\begin{equation}\label{aug16}
\int _{\{|u|>1\}} |u|\, dx = \int _{\{u^\bigstar >1\}} u^\bigstar\, dx < \infty.
\end{equation}
%
%
%For the time being, let us assume that $u$ satisfies the additional property that
%\begin{equation}\label{aug16}
%\int _{\{|u|>1\}}|u(y)|\, dy < \infty\,.
%\end{equation}
In particular, $u\in L^1_{\rm loc}(\rn)$.
%\todo[inline]{A: change the  names of variables as follows: $h \to y$, $z\to t$, $\eta \to z$ // Angela : DONE!}
Let $\psi : \rn \to [0, \infty)$ be the function defined as
\begin{equation*}
	\psi(y)=\frac{(n+1)} {\omega_n}(1-|y|)\chi_{\{|y|>1\}} \quad \text{for $y \in \rn$.}
\end{equation*}
%where $|\mathbb S^{n-1}|=\mathcal H^{n-1}(\mathbb S^{n-1})$.
The function $U: \rn \times [0, \infty) \to \R$, given by
\begin{equation}\label{E:15}
	U(x,t)=\int_{\R^n}\psi(y)u(x+ty)\,dy\quad\text{for $(x,t)\in\R^n\times[0,\infty)$,}
\end{equation}
is thus well defined.
%
%
%\todo[inline]{A: the function $U$ is well defined if $u\in L^1_{\rm{loc}}(\rn)$. It seems to me that additional assumptions on $u$ are only needed  for the validity of the second inequality in \eqref{E:27}} where
%\begin{equation*}
%	\psi(h)=|S^{n-1}|^{-1}n(n+1)(1-|h|)_+.
%\end{equation*}
%Here, $|\mathbb S^{n-1}|=\mathcal H^{n-1}(\mathbb S^{n-1})$.
One has that
\begin{align}\label{E:16}
	|\nabla U(x,t)|
		 \le \frac{(n+1)(n+2)}{t\,\omega_n}\int_{\{|y|<1\}} |u(x+ty)-u(x)|\,dy
%			\\
%		& = \frac{K}{z\,\omega _n}\int_{\{|h|<1\}} |u(x+zh)-u(x)|\,dh
			\quad\text{for a.e. $(x,t)\in\R^n\times[0,\infty)$.}
\end{align}
Inequality \eqref{E:16} can be verified as follows. We have that
\begin{equation}\label{dec40}
\int_{\rn} \psi \Big(\frac{z-x}t\Big)\, \frac{dz}{t^n} = \int _{\rn}\psi (y)\, dy = 1 \quad \text{for $(x,t)\in\R^n\times[0,\infty)$.}
\end{equation}
Differentiating the leftmost side of equation \eqref{dec40} with respect to $x$ and with respect to $t$ yields
\begin{equation}\label{dec41}
\int_{\rn}\nabla  \psi \Big(\frac{z-x}t\Big)\, dz = 0 \quad \text{for a.e. $(x,t)\in\R^n\times[0,\infty)$,}
\end{equation}
and
\begin{equation}\label{dec42}
\int_{\rn} \nabla \psi \Big(\frac{z-x}t\Big)\cdot \frac{z-x}t\, \frac{dz}{t^{n+1}} + n \int _{\rn}  \psi \Big(\frac{z-x}t\Big) \frac{dz}{t^{n+1}}=0 \quad \text{for a.e. $(x,t)\in\R^n\times[0,\infty)$,}
\end{equation}
respectively, where $\lq\lq \cdot "$ stands for scalar product. Therefore,
\begin{equation}\label{dec43}
\nabla _x U(x,t) = \int_{\rn}(u(z) - u(x))\nabla  \psi \Big(\frac{z-x}t\Big)\, \frac{dz}{t^{n+1}} \quad \text{for a.e. $(x,t)\in\R^n\times[0,\infty)$,}
\end{equation}
and
\begin{align}\label{dec44}
U_t(x,t) & = - \int_{\rn} (u(z) - u(x)) \nabla \psi \Big(\frac{z-x}t\Big)\cdot \frac{z-x}t\, \frac{dz}{t^{n+1}}
    \\
    & \quad - n \int _{\rn} (u(z) - u(x)) \psi \Big(\frac{z-x}t\Big) \frac{dz}{t^{n+1}}
\quad \quad \quad  \text{for a.e. $(x,t)\in\R^n\times[0,\infty)$,} \nonumber
\end{align}
where $\nabla _xU$ and $U_t$ denote the vector of the derivatives of $U$ with respect to $x$, and the derivative of $U$ with respect to $t$.
Since
$$\psi \Big(\frac{z-x}t\Big) \leq \frac{n+1}{\omega _n} \chi_{\{|z-x|<t\}}(z)$$
and
$$\Big| \nabla  \psi \Big(\frac{z-x}t\Big)\Big|\leq \frac{n+1}{\omega _n} \chi_{\{|z-x|<t\}}(z), \quad  \Big|\nabla \psi \Big(\frac{z-x}t\Big)\cdot \frac{z-x}t\Big| \leq  \frac{n+1}{\omega _n} \chi_{\{|z-x|<t\}}(z),$$
for a.e. $(x,t)\in\R^n\times[0,\infty)$, we deduce from inequalities \eqref{dec43} and \eqref{dec44} that
\begin{align*}
|\nabla U(x,t)|& \leq |\nabla _xU(x,t)|+ |U_t(x,t)| \leq  \frac{(n+1)(n+2)}{\omega _n} \int_{\{|z-x|<t\}} |u(z) - u(x)| \, \frac{dz}{t^{n+1}}
\\ & = \frac{(n+1)(n+2)}{t\,\omega_n}\int_{\{|y|<1\}} |u(x+ty)-u(x)|\,dy  \quad \text{for a.e. $(x,t)\in\R^n\times[0,\infty)$,}
\end{align*}
namely \eqref{E:16}. 
On setting $K=(n+1)(n+2)$,  one has that
\begin{align}\label{E:17}
	\int_{0}^{\infty}&\int_{\R^n}t^{-1} A\left(t^{1-s}|\nabla U(x,t)|\right)\,dx\,dt
			\\
		& \le
			\int_{0}^{\infty}\int_{\R^n}t^{-1} A\left(\frac{K}{t^{s}\omega _n}\int_{\{|y|<1\}} |u(x+ty)-u(x)|\,dy\right)\,dx\,dt\nonumber
				%\quad \text{(since $\omega_n=|\{|h|<1\}|$)}
			\\
%		& = \int_{0}^{\infty}\int_{\R^n}z^{-1} A\left(\frac{K}{|\{|h|<1\}|z^{s}}\int_{\{|h|<1\}} |u(x+zh)-u(x)|\,dh\right)\,dx\,dz\nonumber
%		\\
		& \le \int_{0}^{\infty}\int_{\R^n}t^{-1} \frac{1}{\omega_n}\int_{\{|y|<1\}}A\left(\frac{K |u(x+ty)-u(x)|}{t^{s}}\right)\,dy\,dx\,dt\nonumber
			%	\quad \text{(by Jensen's inequality, since $\omega_n=|\{|y|<1\}|$)}
			\\
		& = \frac{1}{\omega_n}\int_{0}^{\infty}t^{-1}\int_{\{|y|<1\}}\int_{\R^n}A\left(\frac{K |u(x+ty)-u(x)|}{t^{s}}\right)\,dx\,dy\,dt\nonumber
				%\quad \text{(by Fubini's theorem)}
			\\
		& = \frac{1}{\omega_n}\int_{0}^{\infty}t^{-1}\int_{\{|z|<t\}}\int_{\R^n}A\left(\frac{K
		|u(x+z)-u(x)|}{t^{s}}\right)\frac{dx}{t^n}\,dz\,dt\nonumber
			%	\quad \text{(by change of variables $z=ty$)}
			\\
		& = \frac{1}{\omega_n}\int_{\R^n}\int_{\R^n}\int_{|z|}^{\infty}t^{-1-n}A\left(\frac{K
		|u(x+z)-u(x)|}{t^{s}}\right)\,dt\,dz\,dx \nonumber
			%	\quad \text{(by Fubini's theorem)},
\end{align}
where the first inequality holds by \eqref{E:16}, the second inequality by
Jensen's inequality (since $\omega_n=|\{|y|<1\}|$),  the first equality by Fubini's theorem, the second equality by the change of variables $z=ty$, and the last one by Fubini's theorem again.
Now, note that
\begin{equation}\label{E:18}
	\int_{t}^{\infty}\tau^{-1-n}A\left(\frac{r}{\tau^{s}}\right)\,d\tau
	= \frac{r^{-\frac{n}{s}}}{s}\int_{0}^{\frac{r}{t^s}}\tau^{\frac{n}{s}-1}A(\tau)\,d\tau
		= \frac{1}{st^n} F\left(\frac{r}{t^s}\right) \quad \text{for $r,t >0$,}
\end{equation}
where $F$ is the Young function defined as
\begin{equation*}
	F(t) = t^{-\frac{n}{s}}\int_{0}^{t}\tau^{\frac{n}{s}-1}A(\tau)\,d\tau \quad \text{for $t >0$.}
\end{equation*}
We claim that
\begin{equation}\label{E:19}
	\text{$F$ is equivalent to $A$.}
\end{equation}
Precisely, since the function $A(t)/t$ is non-decreasing,
\begin{equation}\label{aug10}
	F(t) \le t^{-\frac{n}{s}}\frac{A(t)}{t}\int_{0}^{t}\tau^{\frac{n}{s}}\,d\tau = \frac{s}{n+s}A(\tau) \quad \text{for $t >0$,}
\end{equation}
and
\begin{equation*}
	F(t) \ge t^{-\frac{n}{s}}\int_{\frac{t}{2}}^{t}\tau^{\frac{n}{s}-1}A(\tau)\,d\tau
			\ge t^{-\frac{n}{s}} \frac{A(\frac{t}{2})}{\frac{t}{2}}\int_{\frac{t}{2}}^{t}\tau^{\frac{n}{s}}\,d\tau
			= \frac{2s}{n+s}\left(1-\frac{1}{2^{\frac{n}{s}+1}}\right)A\left(\frac{t}{2}\right)
			\geq
		\frac{s}{n+s}A\left(\frac{t}{2}\right)
\end{equation*}
for $t>0$.
From~\eqref{E:17},~\eqref{E:18}, \eqref{E:19} and \eqref{aug10} one obtains that
\begin{align}\label{E:20}
	\int_{0}^{\infty}\int_{\R^n}t^{-1}A\left(t^{1-s}|\nabla U(x,t)|\right)\,dx\,dt
		\le \frac{1}{ n\omega_n}\int_{\R^n}\int_{\R^n}
			A\left(K\frac{|u(x)-u(y)|}{|x-y|^s}\right)\, \frac{dx\,dy}{|x-y|^n}.
\end{align}
Next,
\begin{equation}\label{E:21}
	\frac{|u(x)|}{|x|} = \frac{|U(x,0)|}{|x|} \le \frac{|U(x,t)|}{|x|} + \frac{1}{|x|}\int_{0}^{t}\left|\frac{\partial U}{\partial \tau}(x,\tau)\right|\,d\tau
\quad\text{for $(x,t) \in\R^{n}\times (0,\infty)$.}
\end{equation}
Observe that if $G$ is a~Young function, then
%\begin{equation}\label{E:22}
%	\int_{0}^{r}z^{-1}G\left(z^{1-s}t\right)\,dz
%		\ge \frac{1}{1-s}G\left(Cr^{1-s}t\right)
%			\quad\text{for $r,t\in(0,\infty)$},
%\end{equation}
%for some absolute constant $C$. Actually,
%since $\frac{G(t)}{t}$ is increasing,
%on setting
%\begin{equation*}
%	\overline{G}(t)=\int_{0}^{t}\frac{G(z)}{z}\,dz \quad\text{for $t\in(0,\infty)$},
%\end{equation*}
%one has that
%\begin{equation}\label{E:23}
%	\overline{G}(t)\ge G(t/2)  \quad \text{for $t \geq 0$.}
%\end{equation}
% Thus, by~\eqref{E:23},
\begin{equation}\label{E:24}
	\int_{0}^{r}\tau^{-1}G\left(\tau^{1-s}\eta\right)\,d\tau
		= \frac{1}{1-s}\int_{0}^{r^{1-s}\eta}\frac{G(\tau)}{\tau}\,d\tau
			\ge \frac{1}{1-s}G\left(r^{1-s}\eta/2\right) \quad\text{for $r, \eta>0$,}
\end{equation}
where the last inequality holds since the function $G(\tau)/\tau$ is non-decreasing.
Equation ~\eqref{E:24}, applied with $G={\widehat A}$, $r=|x|$, $\eta=\frac{|u(x)|}{|x|}$, and equation \eqref{E:21} yield
\begin{align}\label{E:25}
	{\widehat A}  \left(\frac{|u(x)|}{2|x|^s}\right)
		&\le (1-s)\int_{0}^{|x|}t^{-1}{\widehat A}\left(t^{1-s}\frac{|u(x)|}{|x|}\right)\,dt
			\\
		& \le (1-s)\int_{0}^{|x|}t^{-1}{\widehat A}\left(t^{1-s}\frac{|U(x,t)|}{|x|}
			+ \frac{t^{1-s}}{|x|}\int_{0}^{t}\left|\frac{\partial U}{\partial\tau}(x,\tau)\right|\,d\tau\right)\,dt
			\nonumber
			\\
%		& \le \frac{(1-s)}{2}\int_{0}^{|x|}z^{-1}{\widehat A}\left(2z^{1-s}\frac{|U(x,z)|}{|x|}\right)\,dz
%			+ \frac{(1-s)}{2}\int_{0}^{|x|}z^{-1}{\widehat A}\left(\frac{2z^{1-s}}{|x|}\int_{0}^{z}\left|\frac{\partial U}{\partial\tau}(x,\tau)\right|\,d\tau\right)\,dz
%			\nonumber
%			\\
		& \le (1-s) \int_{0}^{|x|}t^{-1}{\widehat A}\left(2t^{1-s}\frac{|U(x,t)|}{|x|}\right)\,dt
		\nonumber
		\\ &	\quad +  (1-s)\int_{0}^{|x|}t^{-1}{\widehat A}\left(2t^{-s}\int_{0}^{t}\left|\frac{\partial U}{\partial\tau}(x,\tau)\right|\,d\tau\right)\,dt \quad\text{for $(x,t) \in\R^{n}\times (0,\infty)$.}
			\nonumber
\end{align}
Thus, owing Lemma~\ref{L:1}, applied with $A$ replaced by ${\widehat A}$, and Proposition C
\begin{align}\label{E:25a}
	{\widehat A} \left(\frac{|u(x)|}{2|x|^s}\right)
		  &\le (1-s) \int_{0}^{|x|}t^{-1}{\widehat A}\left(2t^{1-s}\frac{|U(x,t)|}{|x|}\right)\,dt
			+(1-s) \int_{0}^{|x|}t^{-1}{\widehat A}\left(\frac{2}{s}t^{1-s}\left|\frac{\partial
			U(x,t)}{\partial t}\right|\right)\,dt
\\
		& \le (1-s)\int_{0}^{|x|}t^{-1}{\widehat A}\left(2t^{1-s}\frac{|U(x,t)|}{|x|}\right)\,dt
		+ (1-s)\int_{0}^{|x|}t^{-1}A\left(\frac{C}{s}t^{1-s}|\nabla U(x,t)|\right)\,dt \nonumber
\end{align}
for $(x,t) \in\R^{n}\times (0,\infty)$, and for some constant $C=C(n,s)$. Hence,
by~\eqref{E:25a},
\begin{align}\label{E:26}
	\int_{\R^n}{\widehat A}  \left(\frac{|u(x)|}{2|x|^s}\right)\,dx
		& \le (1-s) \int_{\R^n}\int_{0}^{|x|}t^{-1}{\widehat A}\left(2t^{1-s}\frac{|U(x,t)|}{|x|}\right)\,dt\,dx
			\\
		& \quad + (1-s) \int_{\R^n}\int_{0}^{\infty}t^{-1}A\left(\frac{C}{s}t^{1-s}|\nabla U(x,t)|\right)\,dt\,dx. \nonumber
\end{align}
We now make use of polar coordinates  $(\varrho, \theta, \varphi)$ in $\R^n\times (0,\infty)$, with $\varrho \in (0, \infty)$, $\theta \in (0, \tfrac \pi 2)$, $\varphi \in Q$, where $Q$ is a parallelepiped in $\R^{n-1}$. In particular,
$\varrho = \sqrt{|x|^2+t^2}$ and $\cos \theta = t/\varrho$.  From Lemma \ref{L:2} one infers that
\begin{align}\label{E:27}
\int_{\R^n} \int_{0}^{|x|}&t^{-1}{\widehat A}\left(\frac{t^{1-s}}{\sqrt{2}}\frac{|U(x,t)|}{|x|}\right)\,dt\,dx \leq  \int_{\R^n} \int_{0}^{\infty} t^{-1}{\widehat A}\left(\frac{t^{1-s}}{\sqrt{|x|^2+t^2}}|U(x,t)|\right)\,dt\,dx
\\ \nonumber  &
= \int_Q \int _0^{\frac \pi2}\int_{0}^{\infty}(\cos\theta)^{-1}\varrho^{-1}
	{\widehat A}\left((\cos\theta)^{1-s}\varrho^{-s}|\widehat U(\varrho,\theta, \varphi)|\right)\varrho^{n}\,J(\theta, \varphi) d\varrho\,d\theta\, d\varphi
\\ \nonumber &\leq
\int_Q \int _0^{\frac \pi2}\int_{0}^{\infty}(\cos\theta)^{-1}\varrho^{-1}	A\left(C(\cos\theta)^{1-s}\varrho^{1-s}\left|\frac{\partial \widehat U}{\partial \varrho}(\varrho,\theta, \varphi)\right|\right)\varrho^{n}\,J(\theta, \varphi)\, d\varrho\,d\theta\, d\varphi
\\
\nonumber & \leq
\int_{0}^{\infty}\int_{\R^n}  t^{-1}A\left(Ct^{1-s}|\nabla U(x,t)|\right)\,dt\,dx
\end{align}
for some constant $C=C(n,s)$. Here,  $\widehat U$ denotes the expression of $U$ in polar coordinates $(\varrho, \varphi, \theta)$. Also, observe that the present application of Lemma \ref{L:2} relies upon the equality
\begin{equation}\label{aug11} \widehat U(\varrho,\theta, \varphi) = \int _\varrho^\infty
\frac{\partial \widehat U}{\partial r}(r,\theta, \varphi)\, dr \quad \text{for $r>0$,}
\end{equation}
and for a.e. $(\theta, \varphi)$. Equality \eqref{aug11} holds provided that
\begin{equation}\label{aug12}
\lim _{\varrho \to \infty}\widehat U(\varrho,\theta, \varphi) = 0 \quad \text{for  $(\theta, \varphi)\in (0, \tfrac \pi 2) \times Q$.}
\end{equation}
Equation \eqref{aug12} can be verified as follows. We have that
\begin{align}\label{aug13}
    |U(x,t)| = \frac 1{t^n}\bigg|\int_{\rn} \psi\bigg( \frac{y-x}t\bigg) u(y)\, dy\bigg| \leq \frac C{t^n}\int_{|y-x|<t} |u(y)|\, dy \quad \text{for $(x,t) \in \rn\times (0, \infty)$,}
\end{align}
for some constant $C=C(n)$. Thus, for each $(\theta, \varphi)$,  there exists a constant $C=C(n, \theta)$ such that, for every $\varrho >0$, there exists a  ball $B_\varrho \subset \rn$ of radius $\varrho$ satisfying
\begin{align}\label{aug14}
    |U(\varrho, \theta, \varphi)|  \leq   \frac C{|B_\varrho|}\int_{B_\varrho} |u(y)|\, dy.
\end{align}
Set
$$\mu (\tau) = |\{|u|>\tau\}| \quad \text{for $\tau>0$,}$$
and observe that $\mu(\tau)<\infty$ for every $\tau>0$, since $u \in \mathcal M_d(\rn)$.
Define the non-decreasing function $g: (0, \infty) \to [0, \infty)$ as
$$g(\tau) = \begin{cases} \frac 1{\mu (\tau)} & \text{if $\tau \in (0,1]$}
\\ \frac 1{\mu (1)} & \text{if $\tau \in (1, \infty)$.}
\end{cases}
$$
Then the function $G: [0, \infty) \to [0, \infty)$, given by
$$G(\tau) = \int_0^\tau g(r)\, dr \quad \text{for $\tau \geq 0\,,$}
$$
is a Young function such that $G(\tau)>0$ for $\tau>0$. Consequently,
\begin{equation}\label{aug17}
\lim_{\tau\to 0}\frac {\widetilde G (\tau)}\tau =0.
\end{equation}
Now, we claim that $u \in L^G(\rn)$. To verify this claim, note that
\begin{equation}\label{aug18}
\int_{\rn}G(|u|)\, dx = \int _0^\infty g(\tau)\mu(\tau)\, d\tau = \int _0^1d\tau + \frac 1{\mu(1)}\int_1^\infty \mu(\tau)\, d\tau \leq 1 + \frac 1{\mu(1)} \int _{\{|u|>1\}}|u(y)|\, dy < \infty,
\end{equation}
where the last inequality holds by property \eqref{aug16}.
Thus, owing to equations \eqref{aug14}, \eqref{holder} and \eqref{aug17}
\begin{align}\label{aug19}
    \lim _{\varrho \to \infty}|U(\varrho, \theta, \varphi)| & \leq   \lim_{\varrho \to \infty}\frac C{|B_\varrho|}\int_{B_\varrho} |u(y)|\, dy
    \leq \lim_{\varrho \to \infty} \frac {2C}{|B_\varrho|}\|u\|_{L^G(\rn)} \|1\|_{L^{\widetilde G}(B_\varrho)}
    \\ \nonumber & = \lim_{\varrho \to \infty} \frac {2C}{|B_\varrho|}\|u\|_{L^G(\rn)}\frac 1{\widetilde G ^{-1}(1/|B_\varrho|)} =0,
\end{align}
whence \eqref{aug12} follows.
Inequalities~\eqref{E:26} and~\eqref{E:27} imply that
\begin{equation}\label{E:28}
	\int_{\R^n}{\widehat A}\left(\frac{|u(x)|}{|x|^s}\right)\,dx
		\le (1-s)\int_{\R^n} \int_{0}^{\infty} t^{-1}A\left(Ct^{1-s}|\nabla U(x,t)|\right)\,dt\,dx
\end{equation}
for some constant $C=C(n,s)$, with the property that $\lim _{s\to 1^-} C(n,s)< \infty$. Inequality~\eqref{E:3} is a consequence of equations \eqref{E:20} and~\eqref{E:28}.
\\ Inequality \eqref{E:3bis} can be deduced on applying inequality \eqref{E:3}  with $u$ replaced by $u/\lambda$ for any $\lambda >0$.
\end{proof}

\section{Sobolev embeddings: case $s\in (0,1)$}\label{s<1}

The Orlicz-Sobolev embedding for the space $V^{s,A}_d(\R^n)$ of order $s\in (0,1)$, with optimal Orlicz target space, reads as follows.

\begin{theorem}{\rm{\bf [Optimal Orlicz target space]}}\label{T:C2}
Let $n\in\N$ and $s\in(0,1)$. Assume that  $A$ is a~Young function satisfying conditions \eqref{E:0'} and \eqref{E:0''}, and  let $A_{\frac{n}{s}}$ be the Young function defined as in \eqref{Ans}.
Then,
\begin{equation}\label{E:41emb}
V^{s,A}_d(\R^n) \to L^{A_{\frac{n}{s}}}(\R^n),
\end{equation}
and
 there exists a~constant $C=C(n,s)$ such that
\begin{equation}\label{E:41}
	\|u\|_{L^{A_{\frac{n}{s}}}(\R^n)}
		\le C |u|_{s,A,\R^n}
\end{equation}
for every function $u \in V^{s,A}_d(\R^n)$.
%
%
% \lq\lq vanishing at infinity" \todo[inline]{A: this has to be made precise}.
Moreover, $L^{A_{\frac{n}{s}}}(\R^n)$ is the optimal target space in inequality \eqref{E:41} among all Orlicz spaces.
\end{theorem}

%\todo[inline]{A:
%The function $A_{\frac ns}$ depends continuously  on $s$. However, $A_{\frac ns}$ cannot converge uniformly to $A_{\frac n{s_0}}$   in $[0, \infty)$ as $s \to s_0$.}

A counterpart of embedding \eqref{E:41emb}, with  an improved target space which is optimal in the broader class of all rearrangement-invariant spaces, is stated in the next result.

\begin{theorem}{\rm{\bf [Optimal rearrangement-invariant target space]}}\label{T:C1} Assume that $n\in\N$, $s\in(0,1)$ and  $A$ are as  in Theorem~\ref{T:C2}.
Let ${\widehat A}$ be the Young function given by \eqref{E:1} and let  $L({\widehat A},\frac{n}{s})(\R^n)$ be the  Orlicz-Lorentz space defined as in \eqref{aug300}.
%
%
% by
%\begin{equation}\label{E:1}
%	{\widehat A}(t)=\int_0^tb(\tau)\,d\tau\quad\text{for $t\in[0,\infty)$},
%\end{equation}
%where
%\begin{equation}\label{E:2}
%	b^{-1}(r) = \left(\int_{a^{-1}(r)}^{\infty}
%		\left(\int_0^t\left(\frac{1}{a(\varrho)}\right)^{\frac{s}{n-s}}\,d\varrho\right)^{-\frac{n}{s}}\frac{dt}{a(t)^{\frac{n}{n-s}}}
%				\right)^{\frac{s}{s-n}}
%					\quad\text{for $r\ge0$}.
%\end{equation}
%Let $L({\widehat A},\frac{n}{s})(\R^n)$ be the  Orlicz-Lorentz space equipped with the norm
%\begin{equation}\label{E:29}
%	\|u\|_{L({\widehat A},\frac{n}{s})(\R^n)}
%		= \|r^{-\frac{s}{n}}u^{*}(r)\|_{L^{{\widehat A}}(0,\infty)}.
%\end{equation}
Then
\begin{equation}\label{E:30emb}
V^{s,A}_d(\R^n) \to L({\widehat A},\tfrac{n}{s})(\R^n),
\end{equation}
and
there exists a~constant $C=C(n,s)$ such that
\begin{equation}\label{E:30}
	\|u\|_{L({\widehat A},\frac{n}{s})(\R^n)}
		\le C |u|_{s,A, \R^n}
\end{equation}
for every function $u \in V^{s,A}_d(\R^n)$.
Moreover, $L({\widehat A},\frac{n}{s})(\R^n)$ is the optimal target space in inequality \eqref{E:30} among all rearrangement-invariant spaces.
\end{theorem}

We emphasize that
assumption  \eqref{E:0''} on the Young function $A$, appearing in Theorems \ref{T:C2} and \ref{T:C1}, is necessary  for an embedding of the space $V^{s,A}_d(\R^n)$ to hold into any rearrangement-invariant space. This is the content of the following proposition.

\begin{proposition}\label{propindisp}
Let $n \in \N$ and $s \in (0,1)$, and let $A$ be a Young function.  Assume that % the inequality
\begin{equation}\label{aug349}
	V^{s,A}_d(\R^n) \to Y(\R^n)
%
%\|u\|_{Y(\R^n)}
%		\le C |u|_{s,A, \R^n}
\end{equation}
  for some rearrangement-invariant space $Y(\rn)$.
%, for some constant $C$ and
%for every function $u \in V^{s,A}_d(\R^n)$.
Then condition \eqref{E:0''} is fulfilled.
%\begin{equation}\label{aug348}
%\int_{0}\left(\frac{t}{A(t)}\right)^{\frac{s}{n-s}}\,dt < \infty\,.
%\end{equation}
\end{proposition}

\begin{example}\label{ex01}
Let $A$ be a Young function as in \eqref{dec250}. Assume that the parameters $p$, $p_0$, $\alpha$ and $\alpha_0$ fulfill conditions \eqref{dec251} and \eqref{dec252}. Theorem \ref{T:C2} then tells us that embedding  \eqref{E:41emb} and inequality \eqref{E:41} hold if
\begin{equation}\label{dec255}
A_{\frac ns}(t)\,\,  \text{is equivalent to}\,\, \begin{cases} t^{\frac {n{p_0}}{n-s{p_0}}} (\log \frac 1t)^{\frac {n\alpha_0}{n-s{p_0}}} & \quad \text{ if $1\leq {p_0}< \frac ns$ }
\\
e^{-t^{-\frac{n}{s(\alpha_0 +1)-n}}} & \quad  \text{if ${p_0}=\frac ns$ and $\alpha_0 > \frac ns -1$}
\end{cases} \quad \text{near zero,}
\end{equation}
and
\begin{equation}\label{dec256}
A_{\frac ns}(t) \,\, \text{is equivalent to} \,\, \begin{cases} t^{\frac {np}{n-sp}} (\log t)^{\frac {\alpha p}{n-sp}} & \quad  \text{ if $1\leq p< \frac ns$ }
\\
e^{t^{\frac{n}{n-(\alpha +1)s}}}&  \quad \text{if  $p=\frac ns$ and $\alpha < \frac ns -1$}
\\
e^{e^{t^{\frac n{n-s}}}} &  \quad \text{if  $p=\frac ns$ and $\alpha = \frac ns -1$}
\end{cases} \quad \text{near infinity.}
\end{equation}
Moreover, the target space in the resultant embedding and inequality is optimal among all Orlicz spaces.
\\
In the special case when
\begin{equation}\label{jan1}p={p_0}<\tfrac ns \quad \text{ and} \quad  \alpha=\alpha_0 =0,
\end{equation}
this recovers inequality \eqref{intro2} for the classical fractional space $W^{s,p}(\rn)$. In the borderline situation corresponding to
\begin{equation}\label{jan2}
p={p_0}=\tfrac ns, \quad \alpha =0 \quad \text{and} \quad \alpha_0 > \tfrac ns -1,
\end{equation}
a fractional embedding of Pohozhaev-Trudinger-Yudovich  type  \cite{Poh, Tru, Yu} is established -- see also the recent paper \cite{PaRu} in this connection.
%\note[inline]{Lubos: I changed everywhere the spelling to `Pohozhaev' (because that is how we quote it in references) except in the title of Oton--Serra paper which quotes it spelled differently. MathSciNet says `Poho\v zaev' for this particular paper so I am not sure how to handle this properly.}
\\
On the other hand, Theorem \ref{T:C1} provides us with the optimal embedding \eqref{E:30emb} and inequality \eqref{E:30},  with a Young function $\widehat A$ whose behaviour is described in \eqref{dec253} and \eqref{dec254}.
The specific choices \eqref{jan1} yield inequality \eqref{intro3} -- a fractional extension of results of \cite{Oneil, Peetre} -- since the Orlicz-Lorentz target space \eqref{E:30emb} coincides with the standard Lorentz space $L^{\frac {np}{n-p},p}(\rn)$ in this case. Also, when the parameters $p,{p_0},\alpha, \alpha_0$ are as in \eqref{jan2}, inequality  \eqref{E:30} takes the form of a fractional inequality of Brezis-Wainger-Hansson type \cite{BW, Han}.
\end{example}

Lemma \ref{P:S1} below is critical in the proof of the optimality of the target spaces in Theorems \ref{T:C2} and \ref{T:C1}, and in the proof of Proposition \ref{propindisp}.

%\note[inline]{Lubo\v s: I wasn't quite sure where to put this, I guess Proposition~\ref{P:S1} will eventually disappear somewhere else.}
% \todo[inline]{A: I have realized that we can in fact prove also the sufficiency of the Hardy inequality. Thus we have a full reduction principle for embeddings of $V^{s,A}(\rn)$ into rearrangement-invariant spaces.}

\begin{lemma}%{\rm{\bf [Reduction principle, case $s \in (0,1)$]}}
\label{P:S1}
Let $n\in \N$ and  $s \in (0,1)$.  Let $A$ be a Young function and let $Y(\rn)$ be a rearrangement-invariant space.  Assume that there exists a constant $C$ such that
\begin{equation}\label{E:S1}
\|u\|_{Y(\rn)} \leq C |u|_{s,A, \R^n}
\end{equation}
for every function $u \in V^{s,A}_d(\R^n)$. Then
there exists a constant $C'$ such that
\begin{equation}\label{E:S2}
	\bigg\|\int_{r}^{\infty}f(\varrho)\varrho^{-1+\frac{s}{n}}\,d\varrho \bigg\|_{\overline Y(0, \infty)}
		\leq C'\|f\|_{L^A(0, \infty)}
\end{equation}
for every   function $f\in L^A(0,\infty)$.
\end{lemma}

\begin{remark}\label{suff}
{\rm Under the assumption that the function $A$ fulfills conditions \eqref{E:0'}--\eqref{E:0''}, a converse of Lemma \ref{P:S1} also holds. Namely,  inequality \eqref{E:S2} is a sufficient condition for inequality \eqref{E:S1}. To verify this assertion, recall from Theorem B that
the space $L(\widehat A, \frac ns)(0, \infty)$ is optimal in inequality \eqref{E:30}. Thereby,   if inequality \eqref{E:S2} holds, then  $L(\widehat A, \frac ns)(0, \infty) \to Y(0, \infty)$. Hence, $L(\widehat A, \frac ns)(\rn) \to Y(\rn)$ as well, and  inequality \eqref{E:S1} is thus a consequence of  \eqref{E:30}.
}
\end{remark}

\begin{proof}[Proof of Lemma \ref{P:S1}]
 In what follows, the relation $\lesssim$ between two expressions will be used to denote that the former is bounded by the latter, up to a positive constant depending only on $n$ and $s$. The relation $\approx$ means that the two expressions are bounded by each other,  up to positive constants depending only on $n$ and $s$.
\\ Assume that inequality \eqref{E:S1} holds.
Owing to \cite[Theorem 1.1]{pesa}, inequality \eqref{E:S2} holds   for every  function $f\in \mathcal M_+(0,\infty)$ if and only if it just holds for every non-increasing function $f: (0,\infty) \to [0, \infty)$. It thus suffices to prove inequality \eqref{E:S2} for this class of functions $f$. Given any  $f$ of this kind, define the function $u: \rn \to [0,\infty)$ as
\begin{equation*}
	u(x) = \int_{\omega_n|x|^n}^{\infty}f(r)r^{-1+\frac{s}{n}}\,dr \quad \text{for $x\in\R^n$}.
\end{equation*}
Let $x,y\in\R^n$. Suppose first that  $|y|\ge2|x|$. Then
\begin{align*}
	\frac{|u(x)-u(y)|}{|x-y|^{s}}
		& = \frac{\left|\int_{\omega_n|x|^n}^{\infty}f(r)r^{-1+\frac{s}{n}}\,dr-\int_{\omega_n|y|^n}^{\infty}f(r)r^{-1+\frac{s}{n}}\,dr\right|}{|x-y|^{s}}
			\\
		& = \frac{\int_{\omega_n|x|^n}^{\omega_n|y|^n}f(r)r^{-1+\frac{s}{n}}\,dr}{|x-y|^{s}}
			\lesssim \frac{\int_{\omega_n|x|^n}^{\omega_n|y|^n}f(r)r^{-1+\frac{s}{n}}\,dr}{|y|^{s}}.
\end{align*}
Thus,
\begin{align}\label{nov200}
	\int_{\R^n}\int_{|y|\ge2|x|}
		& A\left(\frac{|u(x)-u(y)|}{|x-y|^{s}}\right)\frac{dx\,dy}{|x-y|^n}
			\lesssim \int_{\R^n}\int_{\{|y|\ge2|x|\}}
			A\left(C\frac{\int_{\omega_n|x|^n}^{\omega_n|y|^n}f(r)r^{-1+\frac{s}{n}}\,dr}{|y|^{s}}\right)\frac{dy}{|y|^n}\,dx
				\\  \nonumber
		& \lesssim \int_{\R^n}\int_{\{|y|\ge2|x|\}}
			\int_{\omega_n|x|^n}^{\omega_n|y|^n}A\left(C'f(r)\right)r^{-1+\frac{s}{n}}\,dr\frac{dy}{|y|^{n+s}}\,dx
				\\  \nonumber
		& \lesssim \int_{0}^{\infty}A\left(C'f(r)\right)r^{-1+\frac{s}{n}}
				\int_{\omega_n|x|^n<r\}}\int_{\{|y|\ge2|x|\}}\frac{dy}{|y|^{n+s}}\,dx\,dr
				\\  \nonumber
		& \lesssim \int_{0}^{\infty}A\left(C'f(r)\right)r^{-1+\frac{s}{n}}
				\int_{\{\omega_n|x|^n<r\}}|x|^{-s}\,dx\,dr
				\\  \nonumber
		& \lesssim \int_{0}^{\infty}A\left(C'f(r)\right)r^{-1+\frac{s}{n}}r^{1-\frac{s}{n}}
				\,dr
				 \approx \int_{0}^{\infty}A\left(C'f(r)\right)\,dr,
\end{align}
for some positive constants $C$ and $C'$ depending on $n$ and $s$.
In particular, the second inequality in chain \eqref{nov200} relies upon Jensen's inequality.
\\ Now assume that $|x|\le|y|\le2|x|$. Thereby,
\begin{align*}
	\frac{|u(x)-u(y)|}{|x-y|^{s}}
		& \le \frac{\int_{\omega_n|x|^n}^{\omega_n|y|^n}f(r)r^{-1+\frac{s}{n}}\,dr}{|x-y|^{s}}
			\lesssim f(\omega_n|x|^n)\frac{|y|^s-|x|^s}{|x-y|^{s}}
				\lesssim f\left(\frac{\omega_n}{2^n}|y|^n\right)\frac{|y|^s-|x|^s}{|x-y|^{s}}
					\\
		& \lesssim f\left(\frac{\omega_n}{2^n}|y|^n\right)\frac{|y|-|x|}{|y|^{1-s}|x-y|^{s}}
			\lesssim f\left(\frac{\omega_n}{2^n}|y|^n\right)\frac{|x-y|}{|y|^{1-s}|x-y|^{s}}
				\lesssim f\left(\frac{\omega_n}{2^n}|y|^n\right)|x-y|^{1-s}|y|^{s-1}.
\end{align*}
Note that
\begin{equation*}
	|x-y|^{1-s}|y|^{s-1} \le \left(|x|^{1-s}+|y|^{1-s}\right)|y|^{s-1} \le 2.
\end{equation*}
Thus, there exists a constant $C=C(n,s)$ such that
\begin{align*}
	A\left(\frac{|u(x)-u(y)|}{|x-y|^s}\right)
		& \leq A\left(2Cf\left(\frac{\omega_n}{2^n}|y|^n\right)\frac{|x-y|^{1-s}|y|^{s-1}}{2}\right)
			\lesssim |x-y|^{1-s}|y|^{s-1}A\left(2Cf\left(\frac{\omega_n}{2^n}|y|^n\right)\right),
\end{align*}
where the last inequality holds thanks to property \eqref{kt}.
Therefore,
\begin{align}\label{nov201}
	\int_{\R^n}\int_{\{|x|\le|y|\le2|x|\}}
		& A\left(\frac{|u(x)-u(y)|}{|x-y|^{s}}\right)\frac{dx\,dy}{|x-y|^n}
			\\ \nonumber
		& \lesssim \int_{\R^n}\int_{\{|x|\le|y|\le2|x|\}} A\left(2Cf\left(\frac{\omega_n}{2^n}|y|^n\right)\right)|y|^{s-1}|x-y|^{1-s-n}\,dy\,dx
			\\ \nonumber
		& \lesssim \int_{\R^n}A\left(2Cf\left(\frac{\omega_n}{2^n}|y|^n\right)\right)|y|^{s-1}\int_{\{|x|\le|y|\}} |x-y|^{1-s-n}\,dx\,dy
			\\ \nonumber
		& \lesssim \int_{\R^n}A\left(2Cf\left(\frac{\omega_n}{2^n}|y|^n\right)\right)
			|y|^{s-1}\int_{\{x\in\R^n:|x-y|\le 2|y|\}} |x-y|^{1-s-n}\,dx\,dy
			\\ \nonumber
		& \lesssim \int_{\R^n}A\left(2Cf\left(\frac{\omega_n}{2^n}|y|^n\right)\right)
			|y|^{s-1}|y|^{1-s}\,dy
			= \int_{\R^n}A\left(2Cf\left(\frac{\omega_n}{2^n}|y|^n\right)\right)\,dy
			\\ \nonumber
		& = \int_{0}^{\infty}A\left(2Cf\left(\frac{r}{2^n}\right)\right)\,dr
			\approx \int_{0}^{\infty}A\left(2Cf\left(r\right)\right)\,dr.
\end{align}
Coupling inequality \eqref{nov200} with  \eqref{nov201} yields
\begin{equation*}
	\int_{\R^n}\int_{\{|y|\ge|x|\}}
		A\left(\frac{|u(x)-u(y)|}{|x-y|^{s}}\right)\frac{dx\,dy}{|x-y|^n}
			\lesssim \int_{0}^{\infty}A\left(Cf\left(r\right)\right)\,dr
\end{equation*}
for some positive constant $C=C(n,s)$.
Exchanging the roles of $x$ and $y$ tells us that
\begin{equation*}
	\int_{\R^n}\int_{\{|x|\ge|y|\}}
		A\left(\frac{|u(x)-u(y)|}{|x-y|^{s}}\right)\frac{dx\,dy}{|x-y|^n}
			\lesssim \int_{0}^{\infty}A\left(Cf\left(r\right)\right)\,dr.
\end{equation*}
Altogether,
\begin{equation}\label{july1}
	\int_{\R^n}\int_{\R^n}
		A\left(\frac{|u(x)-u(y)|}{|x-y|^{s}}\right)\frac{dx\,dy}{|x-y|^n}
			\lesssim \int_{0}^{\infty}A\left(Cf\left(r\right)\right)\,dr
\end{equation}
for some constant $C=C(n,s)$.
On replacing $f$ by $f/\lambda$ for any $\lambda >0$ in inequality \eqref{july1} one deduces  that
\begin{equation}\label{july2}
|u|_{s,A, \rn}
			\leq C \|f\|_{L^A(0, \infty)}
\end{equation}
for some positive constant $C=C(n,s)$.
On the other hand,
\begin{equation}\label{nov202}
	u^*(r)=\int_{r}^{\infty}f(\varrho)\varrho^{-1+\frac{s}{n}}\,d\varrho \quad \text{for $r>0$,}
\end{equation}
and
\begin{equation}\label{nov203}
	\|u\|_{Y(\rn)} =  \|u^*\|_{\overline Y(0,\infty)}.
\end{equation}
Inequality \eqref{E:S2} follows from equations \eqref{july2}--\eqref{nov203}.
\end{proof}

\begin{proof}[Proof of  Theorem~\ref{T:C1}]
Inequality~\eqref{E:28'} ensures that
\begin{equation}\label{E:31}
	\int_{\R^n}\int_{\R^n} A\left(\frac{|u(x)-u(y)|}{|x-y|^{s}}\right)\frac{dx\,dy}{|x-y|^n}
		\ge \int_{\R^n}\int_{\R^n} A\left(\frac{|u^{\bigstar}(x)-u^{\bigstar}(y)|}{|x-y|^{s}}\right)\frac{dx\,dy}{|x-y|^n}.
\end{equation}
Since
\begin{equation}\label{E:32}
	\int_{\R^n}{\widehat A}\left(\frac{|u^{\bigstar}(x)|}{|x|^s}\right)\,dx
		= \int_{0}^{\infty}{\widehat A}\left(\frac{\omega_n^{\frac sn}|u^{\ast}(r)|}{r^{\frac{s}{n}}}\right)\,dr,
\end{equation}
an application of inequality ~\eqref{E:3} to the function $u^{\bigstar}$ yields
\begin{equation}\label{E:33}
	\int_{0}^{\infty}{\widehat A}\left(C\frac{|u^{\ast}(r)|}{ r^{\frac{s}{n}}}\right)\,dr
		\le   \int_{\R^n}\int_{\R^n} A\left(\frac{|u^{\bigstar}(x)-u^{\bigstar}(y)|}{|x-y|^{s}}\right)\frac{dx\,dy}{|x-y|^n}
\end{equation}
for a suitable positive constant $C=C(n,s)$.
From inequalities \eqref{E:31} and \eqref{E:33} we deduce that
\begin{equation}\label{E:33'}
	\int_{0}^{\infty}{\widehat A}\left(C\frac{|u^{\ast}(r)|}{r^{\frac{s}{n}}}\right)\,dr
		\le   \int_{\R^n}\int_{\R^n} A\left(\frac{|u (x)-u(y)|}{|x-y|^{s}}\right)\frac{dx\,dy}{|x-y|^n}.
\end{equation}
Inequality~\eqref{E:33'} is a  version of~\eqref{E:30} in integral form. Inequality~\eqref{E:30} follows on applying~\eqref{E:33'} with $u$ replaced by $u/\lambda$ for any $\lambda >0$.
%$A$ replaced by $\frac{A(t)}{M}$, where
%\begin{equation*}
%	M = \int_{\R^n}\int_{\R^n} A\left(\frac{|u (x)-u (y)|}{|x-y|^{s}}\right)\frac{dx\,dy}{|x-y|^n},
%\end{equation*}
%and with ${\widehat A}$ replaced by the function defined as in~\eqref{E:1}--\eqref{E:2}, with $A$ replaced by $\frac{A(t)}{M}$.
\\ It remains to prove that
 the target space  $L({\widehat A},\frac{n}{s})(\R^n)$ is optimal in  inequality \eqref{E:30}. To this purpose, assume that $Y(\rn)$ is a rearrangement-invariant space which renders inequality \eqref{E:30} true. Thus, by Proposition \ref{P:S1}, inequality \eqref{E:S2} holds.  The conclusion then follows via Theorem B.
\end{proof}

\begin{proof}[Proof of Theorem~\ref{T:C2}]
Inequality  \eqref{E:41} can be deduced via Theorem~\ref{T:C1} and Proposition~\ref{P:1}.
The optimality of the space $L^{{A_{\frac{n}{s}}}}(\rn)$ is a consequence  of Proposition \ref{P:S1} and of Theorem A.
\end{proof}

\begin{proof}[Proof of Proposition \ref{propindisp}]
Assume that embedding \eqref{aug349} holds for some rearrangement-invariant space $Y(\rn)$, namely that inequality \eqref{E:S1} holds. Then, by Proposition \ref{P:S1}, inequality
\eqref{E:S2}  holds as well. A necessary condition for one-dimensional Hardy type inequalities -- a dual version of \cite[Lemma~1]{EGP}, see e.g. \cite[Lemma 2]{cianchi_JFA} -- implies that
\begin{equation}\label{aug350}
\sup _{r>0}\|1\|_{\overline{Y}(0,r)}\|\varrho ^{-1+\frac sn}\|_{L^{\widetilde A}(r, \infty)} <\infty.
\end{equation}
Computations show that the second norm on the left-hand side of equation \eqref{aug350} is finite for any $r>0$ if and only if
\begin{equation}\label{aug351}
\int_0 \frac{\widetilde A(t)}{t^{1+\frac n{n-s}}}\, dt < \infty,
\end{equation}
see e.g. \cite[Lemma 3]{cianchi_ASNS}.
Condition \eqref{aug351} is in its turn equivalent to \eqref{aug349} -- see \cite[Lemma 2.3]{cianchi-ibero}.
\end{proof}

\section{Sobolev embeddings: case $s>1$}\label{s>1}

The results of the previous section are extended here to any fractional-order power $s\in (0, n)$. The optimal Orlicz target space for embeddings of the space $V^{s,A}_d(\R^n)$ is exhibited in the following theorem.

\begin{theorem}{\rm{\bf [Higher-order optimal Orlicz target space]}}\label{T:C2h}
Let $n\in\N$ and $s\in (0, n) \setminus \N$.  Assume that $A$ is a~Young function  fulfilling conditions \eqref{E:0'} and \eqref{E:0''}, and let
 $A_{\frac{n}{s}}$ be the Young function defined as  in \eqref{Ans}--\eqref{H}.
Then
\begin{equation}\label{E:41hemb}
V^{s,A}_d(\R^n) \to L^{A_{\frac{n}{s}}}(\R^n),
\end{equation}
and
there exists a~constant $C$ such that
\begin{equation}\label{E:41h}
	\|u\|_{L^{A_{\frac{n}{s}}}(\R^n)}
		\le C \big|\nabla ^{[s]}u\big|_{\{s\},A, \R^n}
\end{equation}
for every function $u \in V^{s,A}_d(\R^n)$.
%
%
% \lq\lq vanishing at infinity" together with all its weak derivatives up to the order $[s]$.
Moreover, $L^{A_{\frac{n}{s}}}(\R^n)$ is the optimal target space in inequality \eqref{E:41h} among all Orlicz spaces.
\end{theorem}

The next result enhances Theorem \ref{T:C2h} and provides us with the optimal rearrangement-invariant target space for embeddings of $V^{s,A}_d(\R^n)$.

\begin{theorem}{\rm{\bf [Higher-order optimal rearrangement-invariant target space]}}\label{T:C1h}
Let $n$, $s$ and $A$ be as in Theorem~\ref{T:C2h}.
Let ${\widehat A}$ be the Young function defined as in \eqref{E:1}--\eqref{E:2}, and let $L({\widehat A},\frac{n}{s})(\R^n)$ be the Orlicz-Lorentz~space equipped with the norm given by \eqref{E:29}.
Then
\begin{equation}\label{E:30hemb}
V^{s,A}_d(\R^n) \to L({\widehat A},\tfrac{n}{s})(\R^n),
\end{equation}
and
 there exists a~constant $C$ such that
\begin{equation}\label{E:30h}
	\|u\|_{L({\widehat A},\frac{n}{s})(\R^n)}
		\le C \big|\nabla ^{[s]}u\big|_{\{s\},A, \R^n}
\end{equation}
for every function $u \in V^{s,A}_d(\R^n)$
%such that  $|\{|\nabla ^{k} u|>t\}|<\infty$ for every $t>0$ and for every $k=0, \dots [s]$.
Moreover, $L({\widehat A},\frac{n}{s})(\R^n)$ is the optimal target space in inequality \eqref{E:30h} among all rearrangement-invariant spaces.
\end{theorem}

\begin{remark}\label{indsiph}
{\rm Assumption  \eqref{E:0''} on the Young function $A$ in Theorems \ref{T:C2h} and \ref{T:C1h} is  necessary   for an embedding
of the form
$$ V^{s,A}_d(\R^n) \to Y(\rn)$$
to hold for some rearrangement-invariant space, also if $s\in (0, n) \setminus \N$ . Indeed, Proposition \ref{propindisp} continues to hold  for $s$ in this range, with a completely analogous proof which makes use of Lemma~\ref{T:reduction_principle} below, a higher-order analogue of Lemma~\ref{P:S1}.
}
\end{remark}

As a consequence of Theorem  \ref{T:C1h}, we can derive a higher-order version of the Hardy type inequality \eqref{E:3bis}.

\begin{theorem}{\rm{\bf[Higher-order fractional Orlicz--Hardy inequality]}}\label{T:1higher}
Let $n$, $s$, $A$ and ${\widehat A}$ be as in Theorem~\ref{T:C1h}. Then there exists a~constant
$C$ such that
%\begin{equation}\label{E:3higher}
%	\int_{\R^n}{\widehat A}\left(\frac{|u(x)|}{|x|^s}\right)\,dx
%		\le  \int_{\R^n} \int_{\R^n}
%			A\left(C\frac{\big|\nabla ^{[s]}u(x)-\nabla ^{[s]}u(y)\big|}{|x-y|^{\{s\}}}\right)\,\frac{dx\, dy}{|x-y|^n}
%\end{equation}
%for every function $u \in V^{[s],A}(\R^n)$ such that $|\nabla ^k u| \in \mathcal M_d(\rn)$ for $k=0, 1, \dots , [s]$.
%\\ Moreover,
\begin{equation}\label{E:3higherbis}
	\left\|\frac{|u(x)|}{|x|^s}\right\|_{L^{\widehat A}(\rn)} \leq C
		\big|\nabla ^{[s]}u\big|_{\{s\},A,\rn}
\end{equation}
for every   function $u \in   V^{s,A}_d (\rn)$.
\end{theorem}

%\todo[inline]{A: examples to be added here}

\begin{example}\label{ex>1}
 Assume that $A$ is a Young function as in \eqref{dec250}, with  parameters $p$, ${p_0}$, $\alpha$ and $\alpha_0$ fulfilling conditions \eqref{dec251} and \eqref{dec252}. Then Theorem \ref{T:C2h} yields embedding \eqref{E:41hemb} and inequality \eqref{E:41h} for every $s\in (0,n)\setminus \N$, where the Young function $A_{\frac ns}$ obeys equations \eqref{dec255} and \eqref{dec256}. For the same range of exponents $s$,    Theorem \ref{T:C1h}  tells us that embedding \eqref{E:30hemb} and inequality \eqref{E:30h} hold,  where the Young function $\widehat A$ fulfills  \eqref{dec253} and \eqref{dec254}. Moreover, the target spaces in the relevant embeddings and inequalities are optimal in their respective classes.
\\ The same kind of Young function $\widehat A$ renders the Hardy inequality \eqref{E:3higherbis} true.
 \\ The special choices of the parameters $p,{p_0},\alpha, \alpha_0$ as in \eqref{jan1} or \eqref{jan2} produce higher-order versions of inequalities \eqref{intro2}
and \eqref{intro3}, or of their limiting versions in the spirit of Pohozhaev-Trudinger-Yudovich and Brezis-Wainger-Hansson, respectively.
\end{example}

The proof of the optimality of the target spaces in Theorems \ref{T:C2h} and \ref{T:C1h} relies up the following lemma.

\begin{lemma}\label{T:reduction_principle}%{\rm{\bf [Higher-order reduction principle]}}
Let $n\in \N$ and  $s \in (0,n)\setminus
\N$.  Let $A$ be a Young function and let $Y(\rn)$ be a rearrangement-invariant space.
Assume that there exists a constant $C$ such that
\begin{equation}\label{E:i1}
\|u\|_{Y(\rn)} \leq C  \big|\nabla ^{[s]}u\big|_{\{s\},A, \R^n}
\end{equation}
for every function $u \in V^{s,A}_d(\rn)$. Then
there exists a constant $C'$ such that
\begin{equation}\label{E:i2}
	\bigg\|\int_{r}^{\infty}f(\varrho)\varrho^{-1+\frac{s}{n}}\,d\varrho \bigg\|_{\overline{Y}(0, \infty)}
		\leq C'\|f\|_{L^A(0, \infty)}
\end{equation}
 for every  function $f\in L^A(0,\infty)$.
\end{lemma}

\begin{remark}\label{suffh}
{\rm The implication of
Lemma \ref{T:reduction_principle} can be reversed
if the function $A$ satisfies conditions \eqref{E:0'}--\eqref{E:0''}. This has been pointed out in  Remark \ref{suff} for the case when $s \in (0,1)$. The argument supporting this assertion is completely analogous to that presented in that remark.
}
\end{remark}

An algebraic inequality to be used in the proof of Lemma \ref{T:reduction_principle}  is the subject of the next result.

\begin{lemma}\label{L:lemma}
Let $x$, $y\in \Rn$, $n \geq 1$ be such that $0<|x|\leq |y| \leq 2|x|$. Let $i \in \N \cup \{0\}$ and let $\beta \in \R$ be such that $i\leq n$ and $\beta+i \geq 0$. Assume that  $\alpha_1, \alpha_2, \dots, \alpha_i \in \{1, 2, \dots, n\}$. Then
\begin{equation}\label{E:est}
|x_{\alpha_1} \cdots x_{\alpha_i} |x|^\beta - y_{\alpha_1} \cdots y_{\alpha_i}|y|^{\beta}| \lesssim |x-y| |x|^{\beta+i-1},
\end{equation}	
up to a multiplicative constant depending on $i$ and $\beta$.  Here, the products $x_{\alpha_1} \cdots x_{\alpha_i} $ and $y_{\alpha_1} \cdots y_{\alpha_i}$ have to be interpreted as $1$ if $i=0$.
\end{lemma}

\begin{proof} Fix $x$ and $y$ as in the statement.
 If $i=0$, then inequality  \eqref{E:est}  reads
$$||x|^\beta - |y|^{\beta}| \lesssim |x-y| |x|^{\beta-1}.$$
This inequality holds, for instance, as a consequence of the mean value theorem for functions of several variables.
\\ Let us now assume that
 $i \geq 1$. Notice that inequality~\eqref{E:est} will follow if we show that
\begin{equation}\label{nov204}
\sum_{\alpha_1=1}^n \dots \sum_{\alpha_i=1}^{n} (x_{\alpha_1} \dots x_{\alpha_i}|x|^\beta - y_{\alpha_1} \dots y_{\alpha_i}|y|^\beta)^2 \leq C |x-y|^2 |x|^{2(\beta+i-1)}
\end{equation}
for some constant $C>0$ depending on $i$ and $\beta$. Thanks to homogeneity of inequality \eqref{nov204}, we may assume, without loss of generality, that $|x|=1$, and hence that $1\leq |y| \leq 2$. Inequality \eqref{nov204} then turns  into
$$
\sum_{\alpha_1=1}^n \dots \sum_{\alpha_i=1}^{n} (x_{\alpha_1} \dots x_{\alpha_i} - y_{\alpha_1} \dots y_{\alpha_i}|y|^\beta)^2 \leq C |x-y|^2,
$$
which can be   rewritten as
$$
|x|^{2i} -2|y|^{\beta} (x\cdot y)^i + |y|^{2\beta +2i} \leq C(|x|^2-2x\cdot y +|y|^2),
$$
namely
\begin{equation}\label{nov205}
1 -2|y|^{\beta} (x\cdot y)^i + |y|^{2\beta +2i} \leq C(1-2x\cdot y +|y|^2),
\end{equation}
since we are assuming that $|x|=1$.
%Here, $x\cdot y$ stands for the scalar product of   $x$ and $y$.
\\
Inequality \eqref{nov205} can, in its turn, be rewritten as
$$
1+ |y|^{2\beta +2i} +2x\cdot y (C-|y|^\beta (x\cdot y)^{i-1}) \leq C(1+|y|^2).
$$
Observe that $|x\cdot y| \leq |x| |y| =|y| \leq 2$. Furthermore,  inasmuch as $[-|y|,|y|] \subseteq [-2, 2]$, the function
$$
[-|y|,|y|] \ni r \mapsto r (C-|y|^\beta r^{i-1})
$$
is non-decreasing, provided that
$C>i 2^{\beta+i-1}$.
Altogether, we deduce that
$$
x\cdot y (C-|y|^\beta (x\cdot y)^{i-1}) \leq |y| (C-|y|^{\beta+i-1}).
$$
It thus suffices to show that
$$
1+ |y|^{2\beta +2i} +2|y| (C-|y|^{\beta+i-1}) \leq C(1+|y|^2),
$$
or, equivalently,
$$
(|y|^{\beta+i}-1)^2 \leq C (|y|-1)^2,
$$
namely,
$$
|y|^{\beta+i}-1 \leq \sqrt{C} (|y|-1).
$$
This inequality clearly holds if $1 \leq |y| \leq 2$, provided that $C$ is sufficiently large, depending on $\beta$ and $i$.
\end{proof}

\begin{proof}[Proof of Lemma~\ref{T:reduction_principle}]
We focus on the case when $s\in (1,\infty) \setminus \N$, and hence $n\geq 2$, since the result for $s\in (0,1)$ has already been proved in Lemma \ref{P:S1}.
Assume that inequality \eqref{E:i1} holds.  By the reason explained in the proof of Lemma \ref{P:S1} in connection with  inequality  \eqref{E:S2}, it suffices to prove inequality \eqref{E:i2}
for every non-increasing function $f: (0,\infty) \to [0, \infty)$.  Moreover, on replacing, if necessary, $f$ by $f\chi_{(0,L)}$ for $L>0$, and letting $L \to \infty$, we may also assume that $f$ has a bounded support. Passage to the limit as $L\to \infty$ in inequality \eqref{E:i2} applied with $f$ replaced by $f\chi_{(0,L)}$ is legitimate owing to the Fatou property of rearrangement-invariant norms \cite[Chapter 1, Definition 1.1]{BS}. Denote, for simplicity,
 $[s]=m$. Given any function $f$ as above,   define the function $u:  \rn \to [0, \infty)$ as
\begin{equation}\label{test}
u(x)=\int_{\omega_n |x|^n}^\infty \int_{r_1}^\infty \dots \int_{r_m}^\infty f(r_{m+1}) r_{m+1}^{-m-1+\frac{s}{n}}\,dr_{m+1} \dots dr_{1} \quad \text{for $x\in \R^n$}.
\end{equation}
It is easily verified that $u$ is $m$-times weakly differentiable, and that $|\{ |\nabla^k u|>t\}| <\infty$ for every $k=0,1,\dots, m$ and every $t>0$. From Fubini's theorem, one can deduce that
\begin{equation}\label{nov207}
u(x)
=\frac{1}{m!}\int_{\omega_n |x|^n}^\infty f(r) r^{-m-1+\frac{s}{n}} (r-\omega_n |x|^n)^m \,dr
\gtrsim \int_{2\omega_n |x|^n}^\infty f(r) r^{-1+\frac{s}{n}}\,dr \quad \text{for $x\in \R^n$.}
\end{equation}
Throughout this proof, the relations $\gtrsim $, $\lesssim$ and $\approx$ hold up to multiplicative constants depending on $n$, $m$ and $s$. The same dependence concerns all constants appearing explicitly.
Inequality \eqref{nov207}, combined with the boundedness of the dilation operator on rearrangement-invariant spaces, implies that
\begin{equation}\label{E:lower}
\|u\|_{Y(\Rn)}\gtrsim \left\|\int_{2\omega_n |x|^n}^\infty f(r) r^{-1+\frac{s}{n}}\,dr \right\|_{Y(\rn)} =
\left\|\int_{2r}^\infty f(\varrho) \varrho^{-1+\frac{s}{n}}\,d\varrho\right\|_{\overline Y(0,\infty)}
 \gtrsim \left\|\int_{r}^\infty f(\varrho) \varrho^{-1+\frac{s}{n}}\,d\varrho\right\|_{\overline Y(0,\infty)}
\end{equation}
for any rearrangement-invariant space $Y(\Rn)$.
\\
Let $g:(0, \infty) \to [0, \infty)$ be the function   defined  as
$$g(r) = \frac{1}{m!}\int_{\omega_n r^n}^\infty f(\varrho) \varrho^{-m-1+\frac{s}{n}} (\varrho-\omega_n r^n)^m \,d\varrho \quad \text{for $r>0$.}$$
 One can verify (see  \cite[Proof of Theorem 3.1]{ACPS})  that any  $m$-th order derivative of $u$ is a linear combination of terms of the form
$$
\frac{x_{\alpha_1} \dots x_{\alpha_i} g^{(k)}(|x|)}{|x|^{m-k+i}},
$$
where $i=0,1,\dots,m$, $k=1,\dots,m$ and $\alpha_1, \dots, \alpha_i \in \{1,\dots, n\}$. Furthermore, $g^{(k)}(r)$ is a linear combination of functions of the form
$$
r^{jn-k}\int_{\omega_n r^n}^\infty \int_{r_{j+1}}^\infty \dots \int_{r_m}^\infty f(r_{m+1}) r_{m+1}^{-m-1+\frac{s}{n}}\,dr_{m+1} \dots dr_{j+1},
$$
where $j=1,\dots,k$. Note that, if $j=k=m$, then the last expression has to be interpreted as
$$
r^{mn-m}\int_{\omega_n r^n}^\infty f(r_{m+1}) r_{m+1}^{-m-1+\frac{s}{n}}\,dr_{m+1}.
$$
Altogether, we deduce that any $m$-th order derivative of $u$ agrees with a linear combination of terms of the form
$$
x_{\alpha_1} \dots x_{\alpha_i} |x|^{jn-m-i} \int_{\omega_n |x|^n}^\infty \int_{r_{j+1}}^\infty \dots \int_{r_m}^\infty f(r_{m+1}) r_{m+1}^{-m-1+\frac{s}{n}}\,dr_{m+1} \dots dr_{j+1},
$$
where $i=0,1,\dots,m$, $j=1,\dots,m$ and $\alpha_1, \dots, \alpha_i \in \{1,\dots, n\}$.
\\
For a.e. $x, y\in \Rn$ we have what follows. Assume first that $|y| \geq 2|x|$. Then
\begin{align}\label{dec27}
|\nabla^m u(x) - \nabla^m u(y)|
&\leq |\nabla^m u(x)| + |\nabla^m u(y)|\\  \nonumber
&\lesssim \sum_{j=1}^m |x|^{jn-m} \int_{\omega_n |x|^n}^{\infty} \int_{r_{j+1}}^\infty \dots \int_{r_m}^\infty f(r_{m+1}) r_{m+1}^{-m-1+\frac{s}{n}}\,dr_{m+1} \dots dr_{j+1}\\ \nonumber
&+\sum_{j=1}^m |y|^{jn-m} \int_{\omega_n |y|^n}^{\infty} \int_{r_{j+1}}^\infty \dots \int_{r_m}^\infty f(r_{m+1}) r_{m+1}^{-m-1+\frac{s}{n}}\,dr_{m+1} \dots dr_{j+1}\\ \nonumber
&\lesssim \sum_{j=1}^m |x|^{jn-m} \int_{\omega_n |x|^n}^\infty f(r) r^{-j-1+\frac{s}{n}}\,dr
+\sum_{j=1}^m |y|^{jn-m} \int_{\omega_n |y|^n}^\infty f(r) r^{-j-1+\frac{s}{n}}\,dr\\ \nonumber
&\lesssim |x|^{s-m} f(\omega_n |x|^n) + |y|^{s-m} f(\omega_n |y|^n),
\end{align}
where the third inequality follows via an analogue of equation \eqref{nov207}, and the last one  thanks to the monotonicity of $f$. Thus,
\begin{align*}
\int_{\Rn} &\int_{|y|\geq 2|x|} A\left(\frac{|\nabla^m u(x) - \nabla^m u(y)|}{|x-y|^{s-m}}\right) \frac{\,dx \,dy}{|x-y|^n}\\
&\lesssim  \int_{\Rn} \int_{|y| \geq 2|x|} A \left( C |x|^{s-m} |y|^{m-s} f(\omega_n|x|^n) \right) \frac{dy}{|y|^n}\,dx
+ \int_{\Rn} \int_{|y| \geq 2|x|} A \left(C f(\omega_n|y|^n) \right) \,dx\frac{dy}{|y|^n}\\
&\lesssim \int_{\Rn} \int_{2|x|}^\infty A \left( C'r^{m-s} |x|^{s-m}  f(\omega_n|x|^n) \right) \frac{dr}{r}\,dx
+ \int_{\Rn} A \left( C' f(\omega_n|y|^n) \right) \,dy
\end{align*}
for some constants $C$ and $C'$.
The change of variables $t=C' r^{m-s}|x|^{s-m}  f(\omega_n|x|^n)$ yields
$$
\int_{2|x|}^\infty A \left( C'r^{m-s} |x|^{s-m}  f(\omega_n|x|^n) \right) \frac{dr}{r}\,dx
\approx \int_0^{C' f(\omega_n |x|^n)} A(t) \frac{dt}{t}
\leq A\left(C' f(\omega_n |x|^n)\right),
$$
where the last inequality holds due to the monotonicity of the function $t\mapsto \frac{A(t)}{t}$.
Therefore,
\begin{align}\label{nov209}
\int_{\Rn} \int_{|y|\geq 2|x|} A\left(\frac{|\nabla^m u(x) - \nabla^m u(y)|}{|x-y|^{s-m}}\right) \frac{\,dx \,dy}{|x-y|^n}
\lesssim \int_{\Rn} A \left(C f(\omega_n|x|^n) \right) \,dx
= \int_0^\infty A \left(C f(r) \right) \,dr
\end{align}
for some constant $C$.
\\ Let us now assume that $|x| \leq |y| \leq 2|x|$. Then
\begin{align*}
&|\nabla^m u(x) - \nabla^m u(y)|\\
&\lesssim \sum_{i=0}^m \sum_{(\alpha_1, \dots, \alpha_i) \in \{1, \dots, n\}^i} \sum_{j=1}^m \bigg|x_{\alpha_1} \cdots x_{\alpha_i} |x|^{jn-m-i} \int_{\omega_n |x|^n}^{\infty} \int_{r_{j+1}}^\infty \dots \int_{r_m}^\infty f(r_{m+1}) r_{m+1}^{-m-1+\frac{s}{n}}\,dr_{m+1} \dots dr_{j+1}\\
& \quad - y_{\alpha_1} \cdots y_{\alpha_i} |y|^{jn-m-i} \int_{\omega_n |y|^n}^{\infty} \int_{r_{j+1}}^\infty \dots \int_{r_m}^\infty f(r_{m+1}) r_{m+1}^{-m-1+\frac{s}{n}}\,dr_{m+1} \dots dr_{j+1}\bigg|\\
&\lesssim \sum_{i=0}^m \sum_{(\alpha_1, \dots, \alpha_i) \in \{1, \dots, n\}^i} \sum_{j=1}^m \bigg|x_{\alpha_1} \cdots x_{\alpha_i} |x|^{jn-m-i} \int_{\omega_n |x|^n}^{\omega_n |y|^n} \int_{r_{j+1}}^\infty \dots \int_{r_m}^\infty f(r_{m+1}) r_{m+1}^{-m-1+\frac{s}{n}}\,dr_{m+1} \dots dr_{j+1}\bigg|\\
&\quad +\sum_{i=0}^m \sum_{(\alpha_1, \dots, \alpha_i) \in \{1, \dots, n\}^i} \sum_{j=1}^m\big|x_{\alpha_1} \cdots x_{\alpha_i} |x|^{jn-m-i} - y_{\alpha_1} \cdots y_{\alpha_i} |y|^{jn-m-i}\big|
\\ & \quad \quad \quad  \quad \quad \quad \quad \times
 \int_{\omega_n |y|^n}^{\infty} \int_{r_{j+1}}^\infty \dots \int_{r_m}^\infty f(r_{m+1}) r_{m+1}^{-m-1+\frac{s}{n}}\,dr_{m+1} \dots dr_{j+1} \\
&\lesssim \sum_{j=1}^m |x|^{jn-m} \int_{\omega_n |x|^n}^{\omega_n |y|^n} \int_{r_{j+1}}^\infty \dots \int_{r_m}^\infty f(r_{m+1}) r_{m+1}^{-m-1+\frac{s}{n}}\,dr_{m+1} \dots dr_{j+1}\\
& \quad + \sum_{j=1}^m |x-y| |x|^{jn-m-1} \int_{\omega_n |y|^n}^{\infty} \int_{r_{j+1}}^\infty \dots \int_{r_m}^\infty f(r_{m+1}) r_{m+1}^{-m-1+\frac{s}{n}}\,dr_{m+1} \dots dr_{j+1}\\
&\lesssim \sum_{j=1}^m |x|^{jn-m} \int_{\omega_n |x|^n}^{\omega_n |y|^n} \int_{r}^\infty f(\varrho)\varrho^{-j-2+\frac{s}{n}}\,d\varrho\,dr
+\sum_{j=1}^m |x-y| |x|^{jn-m-1} \int_{\omega_n |y|^n}^{\infty} f(r) r^{-j-1+\frac{s}{n}}\,dr\\
&\lesssim \sum_{j=1}^m |x|^{jn-m} (|y|^n - |x|^n) \int_{\omega_n |x|^n}^{\infty} f(r) r^{-j-2+\frac{s}{n}}\,dr
+\sum_{j=1}^m |x-y| |x|^{jn-m-1} \int_{\omega_n |x|^n}^{\infty} f(r) r^{-j-1+\frac{s}{n}}\,dr\\
&\lesssim |y-x| |x|^{-m-1+s} f(\omega_n |x|^n).
\end{align*}
Observe that the third inequality holds owing to Lemma~\ref{L:lemma}, the fourth one by an analogue of equation \eqref{nov207},  and the last one since $f$ is non-increasing.
Therefore,
\begin{align*}
\int_{\Rn} &\int_{|x| \leq |y| \leq 2|x|} A\left(\frac{|\nabla^m u(x) - \nabla^m u(y)|}{|x-y|^{s-m}}\right) \frac{\,dx \,dy}{|x-y|^n}\\
&\lesssim  \int_{\Rn} \int_{|x| \leq |y| \leq 2|x|} A\left(C|x|^{s-m-1}|y-x|^{m-s+1}f(\omega_n |x|^n)\right)\frac{dy}{|y-x|^n}\,dx\\
&= \int_{\Rn} \int_{|x| \leq |z+x| \leq 2|x|} A\left(C|x|^{s-m-1}|z|^{m-s+1}f(\omega_n |x|^n)\right)\frac{dz}{|z|^n}\,dx\\
&\leq\int_{\Rn} \int_{|z|\leq 3|x|} A\left(C|x|^{s-m-1}|z|^{m-s+1}f(\omega_n |x|^n)\right)\frac{dz}{|z|^n}\,dx\\
&\lesssim \int_{\Rn} \int_{0}^{3|x|} A\left(C|x|^{s-m-1}r^{m-s+1}f(\omega_n |x|^n)\right)\frac{dr}{r}\,dx
\end{align*}
for some positive constant $C$.
The change of variables $t=C|x|^{s-m-1}f(\omega_n |x|^n)r^{m-s+1}$ yields
$$
\int_{0}^{3|x|} A\left(C|x|^{s-m-1}r^{m-s+1}f(\omega_n |x|^n)\right)\frac{dr}{r}
\approx \int_0^{C f(\omega_n |x|^n)} A(t) \frac{\,dt}{t}
\leq A(C f(\omega_n |x|^n).
$$
Thereby,
\begin{align}\label{nov210}
\int_{\Rn} \int_{|x|\leq|y|\leq 2|x|} A\left(\frac{|\nabla^m u(x) - \nabla^m u(y)|}{ |x-y|^{s-m}}\right) \frac{\,dx \,dy}{|x-y|^n}
\lesssim \int_{\Rn} A \left( C f(\omega_n|x|^n) \right) \,dx
= \int_0^\infty A \left(C f(r) \right) \,dr.
\end{align}
Coupling equations \eqref{nov209} and \eqref{nov210} tells us that
\begin{align}\label{nov211}
\int_{\Rn} \int_{|x|\leq |y|} A\left(\frac{|\nabla^m u(x) - \nabla^m u(y)|}{ |x-y|^{s-m}}\right) \frac{\,dx \,dy}{|x-y|^n}
\lesssim \int_0^\infty A \left(C f(r) \right) \,dr,
\end{align}
for some positive constant $C$.
Adding inequality \eqref{nov211} to a parallel inequality obtained by exchanging the roles of $x$ and $y$, and applying the resultant inequality with $f$ replaced by $f/\lambda$ for any $\lambda >0$ yield
$$
\int_{\Rn} \int_{\Rn} A\left(\frac{|\nabla^m u(x) - \nabla^m u(y)|}{\lambda |x-y|^{s-m}}\right) \frac{\,dx \,dy}{|x-y|^n}
\lesssim \int_0^\infty A \left(\frac{C}{\lambda} f(r) \right) \,dr
$$
for some constant $C$. Now recall that $m=[s]$ and $s-m=\{s\}$  to infer that
\begin{equation}\label{E:upper}
\big|\nabla ^{[s]}u\big|_{\{s\},A, \rn} \lesssim \|f\|_{L^A(0,\infty)}.
\end{equation}
Combining inequalities  \eqref{E:lower} and~\eqref{E:upper}   shows that~\eqref{E:i1} implies~\eqref{E:i2}.
\end{proof}

We have now all the preliminaries at our disposal to accomplish the proof of Theorem~\ref{T:C1h}.

\begin{proof}[Proof of Theorem~\ref{T:C1h}] In this proof we  need  to make use of the function $\widehat A$, defined as in \eqref{E:1}--\eqref{E:2}, also with $s$ replaced by $\{s\}$. For clarity of notation, we shall denote by $\widehat A_{s}$ and $\widehat A_{\{s\}}$ the functions defined by \eqref{E:1}--\eqref{E:2} with $s$ and $\{s\}$, respectively.
 Theorem \ref{T:C1}, applied with $u$ replaced by $\nabla ^{[s]}u$, and with $s$ replaced by $\{s\}$, tells us that
\begin{equation}\label{aug3}
	\|\nabla ^{[s]}u\|_{L(\widehat A_{\{s\}},\frac{n}{\{s\}})(\R^n)}
		\le C \big|\nabla ^{[s]}u\big|_{s,A,\rn}
% \|u\|_{W^{s,A}(\R^n)}
\end{equation}
for some constant $C$ and for every function $u \in V^{s,A}_d(\rn)$. Inequality \eqref{E:30h} will thus follow if we show that
\begin{equation}\label{aug4}
	\|u\|_{L({\widehat A}_s,\frac{n}{s})(\R^n)} \leq C\|\nabla ^{[s]}u\|_{L(\widehat A_{\{s\}},\frac{n}{\{s\}})(\R^n)}
\end{equation}
for some constant $C$ and for every function $u \in V^{s,A}_d(\rn)$.
In order to prove inequality \eqref{aug4}, we make use of Theorems B and  E.
%sharp iteration principle
%of Theorem D.
% for optimal rearrangement-invariant target spaces in Hardy type inequalities.
% Define, for $\alpha \in (0,n)$, the operator
%$$T_\alpha f(r) = \int _r^\infty \varrho ^{-1+\frac \alpha n} f(\varrho)\, d\varrho \quad \text{for $r\geq 0$.}$$
% By \cite[???]{cianchi-ibero},
By Theorem B,
\begin{equation}\label{aug5}
	\|T_{\{s\}}f\|_{L(\widehat A_{\{s\}},\frac{n}{\{s\}})(0,\infty)}
		\le C \|f\|_{L^A(0, \infty)},
\end{equation}
for some constant $C$ and every function $f\in L^A(0, \infty)$,
where $T_{\{s\}}$ is the operator defined as in \eqref{T}. Furthermore,
$L(\widehat A_{\{s\}},\frac{n}{\{s\}})(0,\infty)$ is the optimal rearrangement-invariant target space in \eqref{aug5}. The same result also tells us that
\begin{equation}\label{aug6}
	\|T_{s}f\|_{L({\widehat A}_s,\frac{n}{s})(0,\infty)}
		\le C \|f\|_{L^A(0, \infty)},
\end{equation}
for some constant $C$ and every function  $f\in L^A(0, \infty)$,
and that $L({\widehat A},\frac{n}{s})(0,\infty)$ is the optimal rearrangement-invariant target space in \eqref{aug6}. Denote by $X_{[s]}(0, \infty)$ the optimal rearrangement-invariant target space in the inequality
\begin{equation}\label{aug7}
	\|T_{[s]}f\|_{X_{[s]}(0,\infty)}
		\le C \| f\|_{L(\widehat A_{\{s\}},\frac{n}{\{s\}})(0,\infty)}\,
\end{equation}
for some constant $C$ and for every function $f\in L(\widehat A_{\{s\}},\frac{n}{\{s\}})(0,\infty)$.
Hence, since $s=[s]+\{s\}$, from Theorem E, part (ii), one can deduce that
%the iteration principle of \cite[Theorem~3.5]{CP-Trans},
%\todo[inline]{A: we have to check if this works for non-integer iteration and on infinite intervals}
%\note[inline]{Lubo\v s: Yes, this works, see Theorem A above in magenta color. Please let me know whether I should write more detailed proof (no problem to write it up, but it would be nothing really new).}
\begin{equation}\label{aug8}
X_{[s]}(0, \infty) = L({\widehat A}_s,\tfrac{n}{s})(0,\infty).
\end{equation}
Therefore,
\begin{equation}\label{aug9}
	\|T_{[s]}f\|_{L({\widehat A}_s,\frac{n}{s})(0,\infty)}
		\le C \| f\|_{L(\widehat A_{\{s\}},\frac{n}{\{s\}})(0,\infty)}
\end{equation}
for some constant $C$ and for every function $f\in L(\widehat A_{\{s\}},\frac{n}{\{s\}})(0,\infty)$.
Owing to the reduction principle for integer-order Sobolev inequalities in the version of \cite[Theorem~3.3]{Mih},
%\todo[inline]{A: the exact statement in \cite{Mih} and its application have to be checked}
 inequality \eqref{aug9} implies inequality \eqref{aug4}.
%\todo[inline]{A: the proof of optimality is missing. Maybe this can be accomplished by considering test functions as in our higher-order paper, with $m=[s]$, and arguing as in the proof for $s\in (0,1)$. Lenka, would you have a look at this issue?}
\\ The optimality of the space $L({\widehat A},\frac{n}{s})(\R^n)$ in inequality \eqref{E:30h} follows from Lemma \ref{T:reduction_principle} and Theorem B.
\end{proof}

\begin{proof}[Proof of Theorem~\ref{T:C2h}]
Inequality \eqref{E:41h} follows from Theorem~\ref{T:C1h} and Proposition~\ref{P:1}. The optimality of the space $L^{A_{\frac ns}}(\Omega)$ is a consequence of Lemma \ref{T:reduction_principle} and Theorem A.
\end{proof}

\begin{proof}[Proof of Theorem~\ref{T:1higher}]
%Since the constant $C$ in inequality \eqref{E:30h} is independent of $A$, given $M>0$, one can apply that inequality with  $A(t)$ replaced by the function $A_M$ defined as  $A_M(t)= A(t)/M$, for $t \geq 0$. One can verify that, after this replacement, the function $\widehat A (t)$ turns into $\widehat A (t)/M$. Moreover, the choice
%$$M= \int_{\R^n} \int_{\R^n}
%			A\left(\frac{|\nabla ^{[s]}u(x)-\nabla ^{[s]}u(y)|}{|x-y|^{\{s\}}}\right)\,\frac{dx\, dy}{|x-y|^n}$$
%yields $\big|\nabla  ^{[s]}u\big|_{s, A_M, \rn} \leq 1$. The resultant inequality thus implies that
%\begin{equation}\label{oct100}
%	\int_0^\infty {\widehat A}\big(r^{-\frac sn} u^*(r)\big)\, dr \leq C
%		  \int_{\R^n} \int_{\R^n}
%			A\left(\frac{|\nabla ^{[s]}u(x)-\nabla ^{[s]}u(y)|}{|x-y|^{\{s\}}}\right)\,\frac{dx\, dy}{|x-y|^n}.
%\end{equation}
Property \eqref{HLA} ensures that
\begin{equation}\label{oct101bis}
	\left\|\frac{|u(x)|}{|x|^s}\right\|_{L^{\widehat A}(\rn)}
		\le \| \omega_n ^{\frac sn} r^{-\frac sn} u^*(r)\|_{L^{\widehat A}(0, \infty)}  = \omega_n ^{\frac sn} \|u\|_{L(\widehat A, \frac ns)}.
\end{equation}
%
%
%\begin{equation}\label{oct101}
%	\int_{\R^n}{\widehat A}\left(\frac{|u(x)|}{|x|^s}\right)\,dx
%		\le  \int_0^\infty {\widehat A}\big(\omega_n ^{\frac sn} r^{-\frac sn} u^*(r)\big)\, dr.
%\end{equation}
Coupling inequality  \eqref{oct101bis} with \eqref{E:30h}  yields  \eqref{E:3higherbis}.
\end{proof}

\section{Embeddings on domains}\label{domains}

So far, we have been dealing with embeddings and corresponding Sobolev--Poincar\'e inequalities for functions defined in the whole of $\rn$. This section is devoted to their counterparts for fractional Orlicz-Sobolev spaces on open subsets of $\rn$.
\par The open sets that will be considered are bounded Lipschitz domains according to the following definition. If $n \geq 2$,
 we make use of the notation $x=(x', x_n)$ for $x \in \rn$, where $x'\in \mathbb R^{n-1}$ and $x_n\in \mathbb R$, and
set $\mathcal{Q}= \{x=(x', x_n) \in \R^{n-1}\times\R: |x'|<1, |x_n|<1\}$, $\mathcal{Q}_+= \{x\in \mathcal Q: x_n>0\}$ and $\mathcal{Q}_0= \{x\in \mathcal Q: x_n=0\}$.
A set $\o \subset \rn$ is called a bounded Lipschitz domain if it is a bounded open set and  there exists  a finite number of balls $\{B_j\}_{j=1}^k$ such that $\partial \Omega \subset\bigcup_{j=1}^k B_j$, and corresponding  Lipschitz continuous homeomorphisms with Lipschitz continuous inverses $T_j: \mathcal{Q}\to B_j$, such that $T_j (\mathcal{Q}_+)=  B_j\cap \o$ and $T_j (\mathcal{Q}_0)=  B_j\cap \partial \o$. If $n=1$, then  a bounded Lipschitz domain $\o \subset \mathbb R$ is just the  union of a finite family of bounded intervals at positive distance from each other.

\smallskip
\par
As in the case of embedding in $\rn$, we premise our results in the basic case of spaces of order $s \in (0,1)$.

\begin{theorem}
{\rm{\bf[Optimal embeddings of order $s\in (0,1)$ on domains]}}
\label{sobperp}
Let $\Omega$ be a bounded Lipschitz domain in $\rn$. Assume that $s\in (0,1)$ and that $A$ is a Young function satisfying conditions \eqref{E:0'}--\eqref{E:0''}.
\\ (i)
One has that
\begin{equation}\label{embom1}
W^{s,A}(\o) \to L^{A_{\frac{n}{s}}}(\o),
\end{equation}
 and $L^{A_{\frac{n}{s}}}(\o)$ is the optimal Orlicz target space in \eqref{embom1}.
Moreover,
there exists a~constant $C=C(n,s,\o)$ such that
\begin{equation}\label{perp1}
	%\|u-u_\o\|_{L^{A_{\frac{n}{s}}}(\o)}
		\|u\|_{L^{A_{\frac{n}{s}}}(\o)} \le C |u|_{s,A,\o}
\end{equation}
for every function $u \in V^{s,A}_\perp (\o)$.
\\ (ii) One has that
\begin{equation}\label{embom2}
W^{s,A}(\o) \to L({\widehat A},\tfrac ns) (\o),
\end{equation}
and $L({\widehat A},\frac ns) (\o)$ is the optimal rearrangement-invariant target space in \eqref{embom2}.  Moreover,
there exists a~constant $C=C(n,s,\o)$ such that
\begin{equation}\label{perp2}
	%\|u-u_\o\|_{L({\widehat A},\frac ns) (\o)}
		\|u\|_{L({\widehat A},\frac ns) (\o)}\le C |u|_{s,A,\o}
\end{equation}
for every function $u \in V^{s,A}_\perp(\o)$.
%\todo[inline]{A: The proof of the optimality is missing, and the optimality itself has to be checked. Its proof should follow via a necessity-reduction principle analogous to  that established in $\rn$. The  proof of this principle might make use of similar trial functions $u$, but built upon functions $f$ supported in a small interval near $0$.}
\end{theorem}

Optimal arbitrary-order fractional Orlicz-Sobolev embeddings on domains are stated in the next theorem.

%
%
%\par\noindent More generally, if $s\in (0, \infty) \setminus \N$,  define
%\begin{equation}\label{sep100}
%V^{s,A}_\perp (\o) = \{u \in V^{s,A}(\o): (\nabla ^ku)_\o =0, \, k=1, \dots , [s]\}.
%\end{equation}

\begin{theorem}{\rm{\bf[Higher-order optimal embeddings  on domains]}}\label{omhigher}
Let $\Omega$ be a bounded Lipschitz domain in $\rn$. Assume that $s\in (0,n)\setminus \N$ and that $A$ is a Young function satisfying conditions \eqref{E:0'}--\eqref{E:0''}.
\\ (i) One has that
%There exists a~constant $C=C(n,s,\o)$ such that
%\begin{equation}\label{perp1}
%	\|u-u_\o\|_{L^{A_{\frac{n}{s}}}(\o)}
%		\le C [u]_{s,A,\o}
%\end{equation}
%for every function $u \in V^{s,A}(\o)$. In particular,
\begin{equation}\label{omhigher1}
W^{s,A}(\o) \to L^{A_{\frac{n}{s}}}(\o),
\end{equation}
and $L^{A_{\frac{n}{s}}}(\o)$ is the optimal Orlicz target space in \eqref{omhigher1}.
Moreover, there exists a constant $C$ such that
\begin{equation}\label{sep101}
	\|u\|_{L^{A_{\frac{n}{s}}}(\o)}
		\le C \big|\nabla ^{[s]}u\big|_{\{s\},A,\o}
\end{equation}
for every $u \in V^{s,A}_\perp (\o)$.
%and $L^{A_{\frac{n}{s}}}(\o)$ is the optimal Orlicz target space in \eqref{omhigher1}.
\\ (ii) One has  that
%There exists a~constant $C=C(n,s,\o)$ such that
%\begin{equation}\label{perp2}
%	\|u-u_\o\|_{L({\widehat A},-\frac ns) (\o)}
%		\le C [u]_{s,A,\o}
%\end{equation}
%for every function $u \in V^{s,A}(\o)$. In particular,
\begin{equation}\label{omhigher2}
W^{s,A}(\o) \to L({\widehat A},\tfrac ns) (\o),
\end{equation}
 and $L({\widehat A},\frac ns) (\o)$ is the optimal rearrangement-invariant target space in \eqref{omhigher2}.
Moreover, there exists a constant $C$ such that
\begin{equation}\label{sep102}
	\|u\|_{L({\widehat A},\frac ns) (\o)}
		\le C \big|\nabla ^{[s]}u\big|_{\{s\},A,\o}
\end{equation}
for every $u \in V^{s,A}_\perp (\o)$.
%todo[inline]{A:  optimality of the spaces?}
%and $L({\widehat A},-\frac ns) (\o)$ is the optimal rearrangement-invariant target space in \eqref{omhigher2}.
\end{theorem}

%\todo[inline]{A: examples to be inserted here}

\begin{example}\label{exdomains}
Assume that $\o$ is a bounded Lipschitz domain in $\rn$.
Let $s\in (0,n) \setminus \mathbb N$. Consider   a Young function $A$ as in \eqref{dec250} under   assumptions \eqref{dec251}   and \eqref{dec252} on the parameters $p$,  ${p_0}$, $\alpha$, $\alpha_0$.
By Theorem \ref{omhigher}, Part (i),  embedding \eqref{omhigher1} holds with $A_{\frac ns}$ fulfilling condition   \eqref{dec256}. Note that, since $|\o|<\infty$, only the behaviour near infinity of the function  $A_{\frac ns}$ is relevant now. Thus, embedding \eqref{omhigher1} reads
\begin{equation}\label{dec260}
W^{s,A}(\o) \to \,\, \begin{cases} L^{\frac {np}{n-sp}} (\log L)^{\frac {\alpha p}{n-sp}} (\o)& \quad  \text{ if $1\leq p< \frac ns$ }
\\
\exp {L^{\frac{n}{n-(\alpha +1)s}}}(\o) &  \quad \text{if  $p=\frac ns$ and $\alpha < \frac ns -1$}
\\
\exp \exp {L^{\frac n{n-s}}}(\o)  &  \quad \text{if  $p=\frac ns$ and $\alpha = \frac ns -1$,}
\end{cases}
\end{equation}
and the target spaces are optimal in the class of all Orlicz spaces. Embedding \eqref{dec260} reproduces or extends to the fractional case various results scattered in the literature. The case corresponding to \eqref{jan1} is classical. Integer-order Sobolev embeddings parallel to \eqref{dec260} are special instances of the general results of \cite{cianchi_forum}, which, in their turn, include various  borderline cases established in \cite{EdGuOp, FuLiSb, Poh, Tru, Yu}. In fact, the paper \cite{EdGuOp}, and some sequel contributions by the same authors, also deal with fractional embeddings, but defined in terms of potentials instead of difference quotients.
\\
As far as augmented embeddings with sharp rearrangement-invariant target spaces are concerned,
Theorem \ref{omhigher}, Part (ii), tells us that embedding \eqref{omhigher2} holds with $\widehat A$ obeying \eqref{dec253} and \eqref{dec254}. In this case, the resultant space $L(\widehat A, \frac ns)(\o)$ agrees (up to equivalent norms) with a (generalized) Lorentz-Zygmund space. Thus, embedding \eqref{omhigher2} can be written as
\begin{equation}\label{dec261}
W^{s,A}(\o) \to \,\, \begin{cases} L^{\frac {np}{n-sp}, p; \frac \alpha p} (\o)& \quad  \text{ if $1\leq p< \frac ns$ }
\\
L^{\infty, \frac ns; \frac {\alpha s}n-1} (\o) &  \quad \text{if  $p=\frac ns$ and $\alpha < \frac ns -1$}
\\
L^{\infty, \frac ns; - \frac sn, -1} (\o)  &  \quad \text{if  $p=\frac ns$ and $\alpha = \frac ns -1$,}
\end{cases}
\end{equation}
all target spaces being optimal among all rearrangement-invariant spaces. Embedding \eqref{dec261} is   well known in the integer-order case -- see \cite{cianchi_forum}.  The results of the latter paper encompass, in particular,  classical embeddings
of \cite{Oneil, Peetre}  and of \cite{BW, Han} under \eqref{jan1} and \eqref{jan2}, respectively.

\end{example}

Theorems \ref{sobperp} and \ref{omhigher} rely upon their analogues in $\rn$ and on the following extension domain for fractional Orlicz-Sobolev spaces on Lipschitz domains.

\begin{theorem}{\rm{\bf[Extension operator for fractional Orlicz-Sobolev spaces]}}\label{ext}
Let $\Omega$ be a bounded Lipschitz domain in $\rn$,  with $n\geq 1$.
%\todo[inline]{check the proof if $n=1$}
Assume that $s\in (0, 1)$ and let $A$ be a Young function. Then, there exist a
linear extension operator $\mathcal{E}: W^{s, A}(\o) \to W^{s, A}(\Rn)$ and a constant $C=C(s,\Omega)$  such that
\begin{equation}\label{a200}
    \mathcal{E}(u) =u \quad \text{in}\;\; \o\,,
\end{equation}
and
\begin{equation}\label{ext1}
 \|\mathcal{E} (u)\|_{W^{s,A}(\rn)}\leq C  \|u\|_{W^{s,A}(\Omega)}
\end{equation}
for  every $u \in W^{s,A}(\Omega)$.
\\
Moreover, there exists a constant $C=C(s,\Omega)$ such that
\begin{equation}\label{ext2}
|\mathcal{E} (u)|_{s,A, \rn}\leq C
|u|_{s,A, \o}
\end{equation}
for every
 $u \in V^{s,A}_{\perp}(\Omega)$.
\end{theorem}

%Also, given a Young function $A$ and $s \in (0,1)$, we define
%$$V^{s,A}_{\perp}(\Omega) = \{u\in V^{s,A}(\Omega): u \in L^1(\Omega), \, u_\o =0\}.$$

The following Poincar\'e type inequality, of independent interest, is one ingredient in the proof of Theorem \ref{ext}.

\begin{proposition}
{\rm{\bf[Fractional Orlicz--Poincar\'e inequality]}}
\label{poinc} Let $\Omega$ be a bounded open set in $\rn$, with $n \geq 1$. Assume that $s\in (0,1)$ and that $A$ is a Young function. If $u \in V^{s,A}(\Omega)$, then $u \in L^A(\Omega)$. Moreover, there exists a constant $C=C(s, \Omega)$ such that
\begin{equation}\label{sep1}
\|u- u_\o\|_{L^A(\Omega)} \leq C |u|_{s,A, \Omega}
\end{equation}
for every function $u \in V^{s,A}(\Omega)$. In particular,
\begin{equation}\label{dec210}
\int_\Omega A\big(\big|u(x)-u_\Omega \big|\big)\, dx
\leq \int _\Omega   \int_\Omega A\left(\frac{C|u(x)-u(y)|}{|x-y|^s}\right)\frac{\,dx\,dy}{|x-y|^n}
\end{equation}
for every function $u\in \mathcal M_d(\o)$.
\end{proposition}
\begin{proof} Let $u \in V^{s,A}(\o)$. Suppose, for the time being,   that we already know that $u \in L^1(\Omega)$. Hence, $u_\Omega$ is well defined. Since $\Omega$ is bounded, we have that $|x-y|\leq C$ for some constant $C=C(\Omega)$. Thus,  there exist  constants $C=C(s, \Omega)$ and $C'=C'(s, \Omega)$ such that
\begin{align}\label{sep2}
\int_\Omega A\big(\big|u(x)-u_\Omega \big|\big)\, dx & = \int_\Omega A\bigg(\bigg|\frac 1{|\Omega|}\int _\Omega (u(x)-u(y))dy\bigg|\bigg) \, dx  %\\ \nonumber & \leq
\int_\Omega \frac 1{|\Omega|}\int _\Omega A(|u(x)-u(y)|)dy \, dx \\ \nonumber & \leq
\frac C{|\Omega|}\int _\Omega   \int_\Omega A\left(\frac{C|u(x)-u(y)|}{|x-y|^s}\right)\frac{\,dx\,dy}{|x-y|^n}
%\\ \nonumber &
\leq \int _\Omega   \int_\Omega A\left(\frac{C'|u(x)-u(y)|}{|x-y|^s}\right)\frac{\,dx\,dy}{|x-y|^n}
\end{align}
Note that the first inequality is due to Jensen inequality and the last one holds owing to property \eqref{kt}.
This established inequality \eqref{dec210}.
Inequality \eqref{sep1} follows on applying  \eqref{sep2}with $u$ replaced by  $u/\lambda$ for any $\lambda >0$.
\\It remains to show that, if
$u$ is any function  in $V^{s,A}(\Omega)$, then $u \in L^1(\o)$.
%Define the function $v_t : \o \to \rn$ as
%$$v_t = T_t(u) - (T_t(u))_\o,$$
Given $t>0$, denote by $T_t: \R \to \R$ the function defined as $T_t(r) = \min\{|r|, t\}{\rm sign} (r)$ for $r \in\R$.
 One can verify that
$$|T_t(u) (x) - T_t(u)(y)| \leq |u(x)-u(y)| \quad \text{for $x, y \in \o$.}$$
Since $T_t(u) \in L^\infty(\o)$, and hence $T_t(u)\in L^1(\o)$, we may apply   inequality \eqref{sep2} with $u$ replaced by $T_t(u)$ and deduce  that
\begin{align}\label{sep3}
\int_\Omega A\big(\big|T_t(u) (x)- (T_t(u))_\o \big|\big)\, dx & \leq
\int _\Omega   \int_\Omega A\left(\frac{C'|T_t(u) (x)-T_t(u) (y)|}{|x-y|^s}\right)\frac{\,dx\,dy}{|x-y|^n}
\\ \nonumber & \leq \int _\Omega   \int_\Omega A\left(\frac{C'|u(x)-u(y)|}{|x-y|^s}\right)\frac{\,dx\,dy}{|x-y|^n}
\end{align}
for  $t>0$. Next, denote by $\med(u)$ the median of $u$ given by
$\med(u)= \inf\{\tau\in \R: |\{u>\tau\}|\leq |\o|/2\},$
and observe that
\begin{align}\label{sep4}\med(T_t(u))= \med(u)\quad \text{if   $t > |\med(u)|$. }
\end{align}
 Also, there exists a constant $C=C(|\Omega|)$ such that
\begin{align}\label{sep5}
\int_\Omega A(|v(x)- \med(v)|)\, dx \leq \int_\Omega A(C|v(x)- v_\Omega|)\, dx
\end{align}
for every function $v \in L^1(\Omega)$, see e.g. \cite[Lemma 2.1]{EMilman}. From inequalities  \eqref{sep3}--\eqref{sep5} we infer that
\begin{align}\label{sep6}
\int_\Omega A\big(\big|T_t(u) (x)- \med(u) \big|\big)\, dx
 \leq \int _\Omega   \int_\Omega A\left(\frac{C|u(x)-u(y)|}{|x-y|^s}\right)\frac{\,dx\,dy}{|x-y|^n}
\end{align}
for some constant $C$ and for every $t> |\med(u)|$. Since $\lim _{t\to \infty}T_t(u)  = u$ a.e. in $\Omega$, passing to the limit as $t\to \infty$ in inequality \eqref{sep6} yields, by Fatou's lemma,
\begin{align}\label{sep7}
\int_\Omega A\big(\big|u(x)- \med(u) \big|\big)\, dx   \leq \int _\Omega   \int_\Omega A\left(\frac{C|u(x)-u(y)|}{|x-y|^s}\right)\frac{\,dx\,dy}{|x-y|^n}.
\end{align}
Hence, given any $\lambda >0$,
\begin{align}\label{sep8}
\int_\Omega A(|u(x)|/\lambda)\, dx   & \leq  \int_\Omega A\left(\frac{2|u(x)-\med(u) |}\lambda \right)\, dx   + |\o|A(2|\med(u)|/\lambda )\\ & \nonumber \leq
\int _\Omega   \int_\Omega A\left(\frac{2C|u(x)-u(y)|}{\lambda |x-y|^s}\right)\frac{\,dx\,dy}{|x-y|^n}  + |\o|A(2\med(u) |/\lambda) < \infty.
\end{align}
Since $u \in V^{s,A}(\o)$, the double integral  in equation \eqref{sep8} is finite provided that $\lambda $ is sufficiently large.
This shows that $u \in L^A(\Omega)$, and hence, owing to the second embedding in \eqref{B.3'},  $u \in L^1(\Omega)$.
\end{proof}

\begin{proof}[Proof of Theorem \ref{ext}]
%Let $u \in  W^{A,s}(\Omega)$.
%\todo[inline]{A: in order to prove \eqref{ext1}, we may try to follow the argument of \cite{DPV}}
Inequality \eqref{ext2} is a consequence of  \eqref{ext1} and \eqref{sep1}.
\\
The proof of inequality \eqref{ext1} is patterned on that of \cite[Theorem 5.4 ]{DPV}, and is split in  steps. We focus the case when $n \geq 2$, the one-dimensional case being analogous, and even simpler.
%\todo[inline]{A: case $n=1$ to be checked}

\smallskip
\noindent
\emph {Step1.} Let $E$  be a compact set such that $E\subset \o$. Let $\mathcal E_0$ be the linear operator defined by \eqref{E0}.
%%%%%%%%%%%%%%%%%%%%%%%%%%%%%%%%%%%%%%%%%%%%%%%%%%%%%%%
%In order to prove inequality \eqref{ext1}, let us verify the following preliminary lemmas.
%\begin{lemma}\label{lemma5.1_DNPV}
%Let $\Omega$ be an open set in $\rn$. Let $u$ be a function in $W^{s, A}(\Omega)$ with $s\in (0, 1)$ and $A$ a Young function. Assume that there exists a compact subset $K\subset\Omega$ such that $u\equiv0$ in $\Omega\setminus K$. Define the extension function $\widetilde{u}$ as
%\begin{equation}\label{a1}
%\mathcal{E}_0{u}(x)=
%\begin{cases}
%u(x) & \text{if $x\in \Omega$}
%\\
%0 & \text{if $x\in \rn \setminus\Omega$}\,,
%\end{cases}
%\end{equation}
%for any function $u:\o \to \R$.
Then there exists a constant $C=C(n, s, E, \Omega)$ such that, if $u\in W^{s,A}(\o)$
and  $u=0$ in $\o\setminus E$, then
 $\mathcal{E}_0(u)\in W^{s, A}(\rn)$ and
\begin{equation}\label{a2}
\|\mathcal{E}_0(u)\|_{W^{s, A}(\rn)}\leq C \| u \|_{W^{s, A}(\Omega)}\,.
\end{equation}

\smallskip
\par\noindent
Plainly,
\begin{equation}\label{a3}
\|\mathcal{E}_0(u)\|_{L^A(\rn)} = \|u\|_{L^A(\Omega)}\,.
\end{equation}
It thus suffices to show that
\begin{equation}\label{a5}
\big|\mathcal{E}_0(u)\big|_{s, A, \rn} \leq C  \| u \|_{W^{s, A}(\Omega)}\,
\end{equation}
for some constant $C=C(n,s,E, \o)$. Since $\mathcal{E}_0(u)$ vanishes outside $E$,
\begin{align}\label{a6}
\int_{\rn} \int_{\rn} A\left( \frac{|\mathcal{E}_0(u)(x) - \mathcal{E}_0(u)(y)|}{|x-y|^s}\right) \,\frac{dx\, dy}{|x-y|^n}& =
\int_\Omega \int_\Omega A\left( \frac{|u(x) - u(y)|}{|x-y|^s}\right) \,\frac{dx\, dy}{|x-y|^n}
\\
& \quad + 2 \int_E \left(\int_{\rn\setminus \Omega} A\left( \frac{|u(x) |}{|x-y|^s}\right) \,\frac{ dy}{|x-y|^n}\right) \;dx. \nonumber
\end{align}
Set  $\ddd= \operatorname{dist}(E, \rn \setminus \Omega)$. Thereby,
\begin{equation}\label{a100}
A\left(\frac{|u(x)|}{|x-y|^s}\right)= A\left( \frac{|u(x)|}{\ddd ^s} \frac{\ddd ^s}{|x-y|^s}\right)\leq A\left( \frac{|u(x)|}{\ddd ^s}\right) \frac{\ddd ^s}{|x-y|^s}\quad \text{if $x \in E$ and $y\in \rn\setminus \Omega$.}
\end{equation}
Notice that the last inequality holds by property \eqref{kt}, inasmuch as
$\frac{\ddd ^s}{|x-y|^s}\leq 1$. Hence,
\begin{align}\label{a7}
\int_E \int_{\rn \setminus \Omega}  A\left( \frac{|u(x)|}{|x-y|^s}\right)\; \frac{dx\,dy}{|x-y|^n}& \leq \ddd ^s\int_E \left(\int_{\rn \setminus\Omega} \frac{dy}{|x-y|^{n+s}}\right) A\left( \frac{|u(x)|}{\ddd ^s}\right) \; dx
\\
&
\leq
\ddd ^s \left(\int_{\rn \setminus \Omega} \frac{dy}{{\rm dist}(y,  E)^{n+s}}\right)\; \int_\Omega  A\left( \frac{|u(x)|}{\ddd ^s}\right)\;dx\,.\nonumber
\end{align}
The last integral over $\rn \setminus \Omega$ in equation \eqref{a7} is convergent, since $n+s>n$ and ${\rm dist}(y,  E)\geq \ddd>0$ if $y \in \rn \setminus \Omega$. Inequalities \eqref{a6} and \eqref{a7}, applied with $u$ replaced by $u/{\lambda}$, yield
$$\int_{\rn} \int_{\rn} A\left( \frac{|\mathcal{E}_0(u)(x) - \mathcal{E}_0(u)(y)|}{\lambda |x-y|^s}\right) \,\frac{dx\, dy}{|x-y|^n} \leq
\int_\Omega \int_\Omega A\left( \frac{|u(x) - u(y)|}{\lambda |x-y|^s}\right) \,\frac{dx\, dy}{|x-y|^n}
+  C \int_\Omega  A\left( \frac{|u(x)|}{\lambda \ddd ^s}\right)\;dx$$
for some constant $C=C(n,s,E, \o)$.  Hence, inequality \eqref{a5} follows.

\smallskip
\noindent
\\ \emph {Step 2.}
Assume that $\Omega$ is symmetric about the hyperplane $\{x_n=0\}$. Set $\Omega_{+} =\{ x\in \Omega : x_n>0\}$ and $\Omega_{-}=\{x\in \Omega: x_n\leq 0\}$. Given any  function $u: \o_+ \to \R$,  define the function  $\mathcal E_1(u) : \o \to \R$ as
\begin{equation}\label{a9}
\mathcal E_1(u)(x)=\begin{cases}
 u(x', x_n) & \text{if $x_n\geq 0$}
\\
u(x', -x_n) & \text{if $x_n<0\,.$}
\end{cases}
\end{equation}
If $u \in W^{s, A}(\Omega_+)$,  then $\mathcal E_1(u) \in W^{s, A}(\Omega)$ and
\begin{equation}\label{a10}
\|\mathcal E_1(u)\|_{W^{s,A}(\Omega)}\leq 4 \, \|u\|_{W^{s,A}(\Omega_+)}\,.
\end{equation}
Clearly,
\begin{equation}\label{a12}
\|\mathcal E_1(u)\|_{L^A(\Omega)}\leq 2 \, \|u\|_{L^A(\Omega_+)}\,.
\end{equation}
On the other hand, given $\lambda >0$,
\begin{align}\label{a14}
\int_\Omega\int_\Omega A\left(\frac{|\mathcal E_1(u)(x) - \mathcal E_1(u)(y)|}{\lambda |x-y|^s}\right)\; \frac {dx\,dy}{|x-y|^n} &= \int_{\Omega_+}\int_{\Omega_+} A\left(\frac{|u(x) - u(y)|}{\lambda|x-y|^s}\right)\; \frac {dx\,dy}{|x-y|^n}
\\
& \quad + 2 \int_{\Omega_+}\int_{\Omega_-} A\left(\frac{|u(x) - u(y', -y_n)|}{\lambda |x-y|^s}\right)\; \frac {dx\,dy}{|x-y|^n}\nonumber
\\
&
 \quad + \int_{\Omega_-}\int_{\Omega_-} A\left(\frac{|u(x', -x_n) - u(y', -y_n)|}{\lambda|x-y|^s}\right)\; \frac {dx\,dy}{ |x-y|^n} \nonumber
\\
&\leq 4\,\int_{\Omega_+}\int_{\Omega_+} A\left(\frac{|u(x) - u(y)|}{\lambda |x-y|^s}\right)\; \frac {dx\,dy}{|x-y|^n}\,,\nonumber
\end{align}
where the last inequality is due to the fact that
$(x_n - y_n)^2 \geq (x_n+y_n)^2$ if $x_n \geq 0$ and $y_n\leq 0$.
Inequality \eqref{a14} implies that
\begin{equation}\label{a13}
\big|\mathcal E_1(u)\big|_{s,A, \Omega} \leq 4 |u|_{s,A, \Omega_+}.
\end{equation}
Inequality \eqref{a10} follows from \eqref{a12} and \eqref{a13}.

\smallskip
\noindent
\\ \emph {Step 3.}
%
%\end{proof}
%
Let $\zeta : \Omega \to [0, 1]$ be a Lipschitz continuous function whose Lipschitz constant agrees with $L$. Then there exists a constant $C=C(s, L, \o)$ such that for every
$u\in W^{s, A}(\Omega)$, one has that  $\zeta \, u  \in W^{s, A}(\Omega)$ and
\begin{equation}\label{a15}
\|\zeta\,u\|_{W^{s,A}(\Omega)} \leq C \, \|u\|_{W^{s,A}(\Omega)}.
\end{equation}

\smallskip
\noindent
Inasmuch as $0 \leq \zeta \leq 1$,
\begin{align}\label{a15bis}
\| \zeta \, u\|_{L^A(\Omega)}\leq \|u\|_{L^A(\Omega)}.
\end{align}
Moreover,
\begin{align}\label{a16}
\int_\Omega &\int_\Omega A\left(\frac{|\zeta(x)\, u(x) - \zeta(y)\, u(y)|}{|x-y|^s}\right)\, \frac{dx\,dy}{|x-y|^n}
% =
%\int_\Omega\int_\Omega A\left(\frac{|\psi(x)\, u(x) \pm \psi(x)\, u(y) - \psi(y)\, u(y)|}{|x-y|^s}\right)\, \frac{dx\,dy}{|x-y|^n}
%\\
%& =
%\int_\Omega\int_\Omega A\left(\frac{|u(y)(\psi(x) - \psi(y)) + \psi(x) (u(x) - u(y))|}{|x-y|^s}\right)\, \frac{dx\,dy}{|x-y|^n} \nonumber
\\
&\leq
\int_\Omega\int_\Omega A\left(\frac{2|u(y)||\zeta(x) - \zeta(y)|}{|x-y|^s}\right)\, \frac{dx\,dy}{|x-y|^n} +
\int_\Omega\int_\Omega A\left(\frac{2|\zeta(x)| |u(x) - u(y)||}{|x-y|^s}\right)\, \frac{dx\,dy}{|x-y|^n}\nonumber
\\
&\leq
 \int_\Omega\int_\Omega A\left(\frac{2|u(y)||\zeta(x) - \zeta(y)|}{|x-y|^s}\right)\, \frac{dx\,dy}{|x-y|^n} +
\int_\Omega\int_\Omega A\left(\frac{2 |u(x) - u(y)||}{|x-y|^s}\right)\, \frac{dx\,dy}{|x-y|^n}.\nonumber
\end{align}
Since $\zeta$ has Lipschitz constant $L$,
%\\
%Since $\psi\in\mathcal{C}^{0,1}(\Omega)$, it follows that
\begin{align}\label{a17}
&\int_\Omega\int_\Omega A\left(\frac{2|u(y)||\zeta(x) - \zeta(y)|}{|x-y|^s}\right)\, \frac{dx\,dy}{|x-y|^n}
\\
&\leq
\int_\Omega\int_{\Omega\cap \{|x-y|\leq 1\}} A\left(2\, L\, |u(y)||x-y|^{1-s}\right)\, \frac{dx\,dy}{|x-y|^n}
+ \int_\Omega\int_{\Omega\cap \{|x-y|> 1\}} A\left(\frac{2|u(y)||\zeta(x)-\zeta(y)|}{|x-y|^s}\right)\, \frac{dx\,dy}{|x-y|^n}
\nonumber
\\
&
\leq
\int_\Omega\int_{\Omega\cap \{|x-y|\leq 1\}} A\left(2\, L\, |u(y)||x-y|^{1-s}\right)\, \frac{dx\,dy}{|x-y|^n}
+\int_\Omega \int_{\Omega\cap \{|x-y|> 1\}} A\left(\frac{2|u(y)|}{|x-y|^s}\right)\; \frac{dx\,dy}{|x-y|^n}\nonumber
\\
&
\leq \int_\Omega \left(\int_{\Omega\cap \{|x-y|\leq 1\}}\frac{dx}{|x-y|^{n+s-1}}\right)\,  A(2\, L\,|u(y)|)\;dy
 + \int_\Omega \left( \int_{\Omega\cap \{|x-y|> 1\}} \;\frac {dx}{|x-y|^{n+s}}\right)\,A\left(2|u(y)|)\right)\;dy\nonumber
\\
&
\leq  C \int_\Omega A(C'\,|u(y)|) \; dy  \leq \int_\Omega A(C''\,|u(y)|) \; dy \nonumber
\end{align}
for some constants $C=C(s,\o)$, $C'=C'(L)$ and $C''=C''(s,L,\o)$. Observe that the second inequality holds since
 $|\zeta(x)-\zeta(y)|\leq 1$, the third one holds owing to property \eqref{kt}, and the fourth one since $n+s-1<n$ and $n+s>n$, and the last one by property \eqref{kt} again. From inequalities \eqref{a16} and \eqref{a17}, applied with $u$ replaced by $u/\lambda$ for any $\lambda >0$, we infer that
\begin{equation}\label{a17bis}
|\zeta\,u|_{s,A, \Omega} \leq C \, \|u\|_{W^{s,A}(\Omega)}.
\end{equation}
for some constant $C=C(s, L, \o)$. Coupling inequality \eqref{a15bis} with \eqref{a17bis} yields \eqref{a15}.

\smallskip
\noindent
\emph {Step4.}
Conclusion.

\smallskip
\noindent Let $\mathcal{Q}$, $\mathcal{Q}_+$, $\{B_j\}_{j=1}^k$ and $\{T_j\}_{j=1}^k$ be as in the definition of bounded Lipschitz domain at the beginning of this section.
%Set $\mathcal{Q}= \{x=(x', x_n) \in \R^{n-1}\times\R: |x'|<1, |x_n|<1\}$, $\mathcal{Q}_+= \{x\in \mathcal Q: x_n>0\}$ and $\mathcal{Q}_0= \{x\in \mathcal Q: x_n=0\}$.
%Since $\o$ is a bounded Lipschitz domain, there exists  a finite number of balls $\{B_j\}_{j=1}^k$ such that $\partial \Omega \subset\bigcup_{j=1}^k B_j$, and corresponding  Lipschitz continuous homeomorphisms with Lipschitz continuous inverses $T_j: \mathcal{Q}\to B_j$, such that $T_j (\mathcal{Q}_+)=  B_j\cap \o$ and $T_j (\mathcal{Q}_0)=  B_j\cap \partial \o$.
Since $\rn= \bigcup_{j=1}^k B_j \cup\left(\rn\setminus\partial \Omega\right)$,
there exists a smooth partition of unity $\{\zeta _j\}_{j=0}^k$ with respect to this covering of $\rn$ such that ${\rm supp }\, \zeta _0  \subset\rn\setminus\partial \Omega$, ${\rm supp}\, \zeta_j\subset B_j$ for  $j=1, \ldots,k$, $0\leq \zeta_j \leq 1$ for   $j=0, 1, \ldots, k$, and $\sum_{j=0}^k \zeta_j=1$ in $\rn$.
\\ Given any function $u \in W^{s,A}(\o)$, define the function $v_j : \mathcal{Q}_+ \to \R$, for $j=1, \dots ,k$  as
 \begin{equation*}
 v_j(\hat y)= u\big(T_j(\hat y)\big) \quad \text{for  $\hat y\in \mathcal{Q}_+$}\,.
 \end{equation*}
We claim  that $v_j\in W^{s,A}(\mathcal{Q}_+)$ for $j=1, \dots ,k$, and
\begin{equation}\label{nov400}
\|v_j\|_{W^{s,A}(\mathcal{Q}_+)}\leq C \, \|u\|_{W^{s, A}(B_j \cap \o)}\,
\end{equation}
for some constant $C=C(s, \Omega)$.
 This claim follows from the following chain:
%Indeed, on setting $x= T_j(\hat{x})$ and making the standard change of variable, we obtain
\begin{align}\label{a20}
\int_{\mathcal{Q}_+}&\int_{\mathcal{Q}_+} A\left(\frac{|v_j(\hat{x})-v_j(\hat{y})|}{|\hat{x}-\hat{y}|^s}\right)\; \frac{d\hat{x}\,d\hat{y}}{|\hat{x}-\hat{y}|^n}
%\\ &
= \int_{\mathcal{Q}_+}\int_{\mathcal{Q}_+} A\left(\frac{|u(T_j(\hat{x}))- u(T_j(\hat{y}))|}{|\hat{x}-\hat{y}|^s}\right)\; \frac{d\hat{x}\,d\hat{y}}{|\hat{x}-\hat{y}|^n}%\nonumber
\\
&= \int_{B_j\cap\Omega}\int_{B_j\cap\Omega} A\left(\frac{|u(x)- u(y)|}{|T_j^{-1}(x) - T_j^{-1}(y)|^s}\right)\, |{\rm det}(J(T_j^{-1}))|\; \frac{dx\,dy}{|{T_j^{-1}(x) - T_j^{-1}(y)}|^n}\nonumber
\\
& 
\leq
\int_{B_j\cap\Omega}\int_{B_j\cap\Omega} A\left(\frac{C|u(x)- u(y)|}{|x-y|^s}\right)\; \frac{dx\,dy}{|x-y|^n}\,, \nonumber
\end{align}
for some constant $C$ depending on the Lipschitz constant of $T_j$ and
% and $C'$ also depending on
the Lipschitz constant of $T_j^{-1}$. Here, ${\rm det}(J(T_j^{-1}))$ denotes the determinant of the Jacobian of the map $T_j^{-1}$. Note that the last inequality relies  upon property \eqref{kt} as well.
%where we have made use of the change of variables $x= T_j(\hat{x})$. Notice the last inequality is owing to the fact that $T_j$ is Lipschitz.
\\
Next, let $\overline v_j : \mathcal Q \to \R$ be the function obtained on extending $v_j$ to $\mathcal Q$   as in Step 2, namely $\overline v_j  = \mathcal E_1 (v_j)$. Therefore, $\overline v_j \in  W^{s, A}(\mathcal{Q})$, and
\begin{equation}\label{nov402}
\|\overline v_j\|_{W^{s,A}(\mathcal{Q})}\leq 4 \, \|v_j\|_{W^{s, A}(\mathcal{Q}_+)}\,
\end{equation}
for $j=1, \dots , k$.
Now, let $w_j: B_j \to \R$ be given by
\begin{equation}\label{a21}
w_j(x)=\overline v_j\big(T_j^{-1}(x)\big) \quad \text{for  $x\in B_j$}\,.
\end{equation}
A chain analogous to  \eqref{a20} ensures that  $w_j\in W^{s, A}(B_j)$ and
\begin{equation}\label{nov401}
\|w_j\|_{W^{s,A}(B_j)}  \leq C \|\overline v_j\|_{W^{s,A}(\mathcal{Q})}
\end{equation}
for some constant $C=C(s, \o)$.
Definition \eqref{a21} immediately tells us that $w_j= u$ on $B_j\cap\Omega$,  and hence $\zeta _j\, w_j =  \zeta_j \, u$ on $B_j\cap\Omega$. By Step 3, $\zeta _j \,w_j\in W^{s, A}(B_j)$ and
\begin{equation}\label{a21''}
\|\zeta_j\,w_j\|_{W^{s,A}(B_j)}  \leq C \|w_j\|_{W^{s,A}(B_j)}.
\end{equation}
for $j=1, \dots , k$, for some constant $C=C(s, \o)$.
On the other hand,
$\zeta_j\,w_j$ has compact support in $B_j$.
Hence, the extension $\mathcal E_0 (\zeta_j\,w_j) : \rn \to \R$ of $\zeta_j\,w_j$ to $\rn$, defined as in Step 1, is such that $\mathcal E_0 (\zeta_j\,w_j) \in W^{s,A}(\rn)$ and
\begin{align}\label{a22'}
\|\mathcal E_0 (\zeta_j\,w_j)\|_{W^{s,A}(\rn)} & \leq C
\|\zeta_j\,w_j\|_{W^{s,A}(B_j)}
\end{align}
for $j=1, \dots , k$, for some constant $C=C(s, \o)$. Also,
since $\zeta_0\, u=0$ in a neighborhood of $\partial \Omega$,   the  extension of $\zeta_0\, u$ to $\rn$ given by $\mathcal E_0 (\zeta_0\, u)$ belongs  to $W^{s,A}(\rn)$, and, by Steps 1 and 3,
 \begin{equation}\label{a19}
 \|\mathcal E_0 (\zeta_0\, u)\|_{W^{s, A}(\rn)} \leq C
\| \zeta_0\, u\|_{W^{s, A}(\Omega)}\leq C'\,\|u\|_{W^{s, A}(\Omega)}\,
 \end{equation}
for some constants $C=C(s,\o)$ and $C'=C'(s,\o)$.
Finally, consider  the extension $\mathcal E (u) : \rn \to \R$ of $u$ to $\rn$ given by
\begin{equation*}
\mathcal E(u)= \mathcal E_0 (\zeta_0\, u) + \sum_{j=1}^k \mathcal E_0 (\zeta_j\,w_j)\,.
\end{equation*}
Then $\mathcal E$ defines a linear operator on $W^{s, A}(\Omega)$ such that, $\mathcal E(u)=u$ in $\o$ and,
 owing to inequalities  \eqref{nov400}, \eqref{nov402}, \eqref{nov401}, \eqref{a21''}, \eqref{a22'} and \eqref{a19},
\begin{equation*}
\|\mathcal E(u)\|_{W^{s,A}(\rn)} \leq C \|u\|_{W^{s,A}(\Omega)}\,
\end{equation*}
for some constant $C=C(s,\o)$ and for every $u \in W^{s,A}(\Omega)$.
The proof is complete.
\end{proof}

As in the case of fractional Orlicz-Sobolev spaces in $\rn$, the validity of an embedding on a domain implies a corresponding one-dimensional Hardy type inequality. This is the content of the following lemma, to be used in the proof of the optimality of the target spaces in the embeddings of Theorems \ref{sobperp} and  \ref{omhigher}.

\begin{lemma}\label{T:reduction_omega}%{\rm{\bf [Higher-order reduction principle]}}
Let $n\in \N$ and  $s \in (0,n)\setminus
\N$.  Let $A$ be a Young function, let $\o$ be an open set in $\rn$ such that $|\o|<\infty$ and let $Y(\o)$ be a rearrangement-invariant space.
Assume that there exists a constant $C$ such that
\begin{equation}\label{E:1om}
\|u\|_{Y(\o)} \leq C \|u\|_{W^{s,A}(\Omega)}
\end{equation}
for every function $u \in W^{s,A}(\o)$. Then
there exists a constant $C'$ such that
\begin{equation}\label{E:2om}
	\bigg\|\int_{r}^{|\o|}f(\varrho)\varrho^{-1+\frac{s}{n}}\,d\varrho \bigg\|_{\overline{Y}(0, |\o|)}
		\leq C'\|f\|_{L^A(0, |\o|)}
\end{equation}
 for every  function $f\in L^A(0,|\o|)$.
\end{lemma}

\begin{proof}[Proof, sketched] The proof follows along the same lines as that of Lemma \ref{T:reduction_principle}. One can  assume, without loss of generality, that $0\in \o$. Let $\mathcal B$ be a ball, centered at $0$ and with measure $L$, contained in $\o$.  Consider trial functions $u$ in inequality \eqref{E:1om} of the form \eqref{test}, with $f$ supported in $(0,L)$. Then an analogous argument as in the proof of Lemma \ref{T:reduction_principle} tells us that
\begin{equation}\label{nov500}
	\bigg\|\int_{r}^{L}f(\varrho)\varrho^{-1+\frac{s}{n}}\,d\varrho \bigg\|_{\overline{Y}(0, L)}
		\leq C\bigg(\|f\|_{L^A(0, L)} + \bigg\|\int_{r}^{L}f(\varrho)\varrho^{-1+\frac{s}{n}}\,d\varrho \bigg\|_{L^A(0,L)}\bigg)
\end{equation}
for every function $f \in L^A(0, L)$. On the other hand, (the same proof of) \cite[Inequality (4.10)]{CP_Arkiv} yields
\begin{equation}\label{nov501}
	\bigg\|\int_{r}^{L}f(\varrho)\varrho^{-1+\frac{s}{n}}\,d\varrho \bigg\|_{L^A(0, L)}
		\leq C L^{\frac sn}\|f\|_{L^A(0, L)}
\end{equation}
for some constant $C=C(s,n)$ and for every function $f \in L^A(0,L)$.  Inequalities \eqref{nov500} and \eqref{nov501} imply that \eqref{E:2om} holds with $|\o|$ replaced by $L$. Inequality \eqref{E:2om} holds in its original version, and in fact with $|\o|$ replaced by any $L>0$, as a consequence  of the boundedness  of the dilation operator, defined as in \eqref{dilation}, in any Orlicz space.
\end{proof}

\begin{proof}[Proof of Theorem~\ref{sobperp}]
Inequality \eqref{perp1} follows via inequality \eqref{ext2} and  Theorem \ref{T:C2}.   Embedding \eqref{embom1} can be deduced from an application of inequality \eqref{perp1} with $u$ replaced by $u-u_\o$ for any function $u \in W^{s,A}(\o)$.
The optimality of the target space in embedding  \eqref{embom1} is a consequence of Lemma \ref{T:reduction_omega} and of Theorem A.
\\ Inequality
\eqref{perp2}, and hence embedding \eqref{embom2}, follow via inequality \eqref{ext2} and  Theorem \ref{T:C1}. The optimality of the target space in embedding  \eqref{embom2} is a consequence of Lemma \ref{T:reduction_omega} and of Theorem B.
\end{proof}

\begin{proof}[Proof of  Theorem \ref{omhigher}]
Let us begin by proving  inequality \eqref{sep102}. Let $u \in W^{s,A}(\o)$. As in the proof of Theorem \ref{T:C1h},   we denote by $\widehat A_{s}$ and $\widehat A_{\{s\}}$ the functions associated with $A$ as in \eqref{E:1}--\eqref{E:2} with $s$ and $\{s\}$, respectively.
%Denote by $v$ any $[s]$-th order weak derivative of $u$. Then $v \in W^{\{s\},A}(\o)$.
An application of embedding \eqref{embom2} to each $[s]$-th order weak derivative of $u$ tells us that
\begin{equation}\label{omhigher3}
\|\nabla ^{[s]}u\|_{L(\widehat A_{\{s\}},  \frac n{\{s\}})(\o)} \le C \|\nabla ^{[s]}u\|_{W^{\{s\},A}(\o)}
\end{equation}
for some constant $C=C(s,\o)$. On the other hand, owing to inequality \eqref{aug9} and  to the reduction principle for integer-order Sobolev inequalities of \cite[Theorem 6.1]{CPS},
\begin{equation}\label{omhigher4}
\|u\|_{L(\widehat A_{s},  \frac ns )(\o)}  \leq C \Big( \|\nabla ^{[s]}u\|_{L(\widehat A_{\{s\}},  \frac n{\{s\}})(\o)}+ \sum _{k=0}^{[s]-1}\|\nabla ^ku\|_{L^1(\o)}\Big)
\end{equation}
for some constant $C=C(s,\o)$.
Coupling inequalities \eqref{omhigher3} and  \eqref{omhigher4}, and making use of property \eqref{normineq}, implies that
\begin{equation}\label{omhigher5}
\|u\|_{L(\widehat A_{s}, \frac ns )(\o)}  \leq C \Big(\|\nabla ^{[s]}u\|_{W^{\{s\},A}(\o)} + \sum _{k=0}^{[s]-1}\|\nabla ^ku\|_{L^1(\o)}\Big) \leq C'  \|u\|_{W^{s,A}(\o)}
\end{equation}
for some constants $C$ and $C'$ depending on $s$ and $\o$. This establishes embedding \eqref{omhigher1}.
\\ As far as inequality \eqref{sep102} is concerned, assume that $u \in V^{s,A}_\perp (\o)$. Since $(\nabla ^{[s]}u)_\o =0$, inequality \eqref{perp2}, applied to each $[s]$-th order derivative of $u$ tells us that
\begin{equation}\label{omhigher6}
\|\nabla ^{[s]}u\|_{L(\widehat A_{\{s\}}, \frac n{\{s\}})(\o)} \le C  \big|\nabla ^{[s]}u\big|_{\{s\},A,\o}
\end{equation}
for some constant $C=C(s,\o)$. On the other hand, inasmuch as we are also assuming that $(\nabla ^ku)_\o =0$ for $k=0, \dots , [s]-1$, by  the integer-order result of \cite[Theorem 6.1]{CPS} and inequality \eqref{aug9} again,
\begin{equation}\label{omhigher7}
\|u\|_{L(\widehat A_{s}, \frac ns )(\o)}  \leq C \|\nabla ^{[s]}u\|_{L(\widehat A_{\{s\}}, \frac n{\{s\}})(\o)}
\end{equation}
for some constant $C=C(s,\o)$. Inequality \eqref{sep102} follows from \eqref{omhigher6} and \eqref{omhigher7}.
\\In order to deduce embedding \eqref{omhigher2} from inequality  \eqref{sep102}, observe that, for each function $u\in W^{s,A}(\o)$, there exists a (unique) polynomial $P_u$, of order at most $[s]$,   such that $u - P_u \in V^{s,A} _\bot (\o)$.
 Moreover, the
coefficients of $P_u$ are linear combinations of the components of
$\int _{\Omega} \nabla ^k u \,dx $, for $k=0, \dots , [s]$, with
coefficients depending on $[s]$, $k$ and $\o$.  Thus, there exist constants $C=C(s,\o)$ and $C'=C'(s,\o)$ such that
\begin{align}\label{dec215}
\|u\|_{L(\widehat A_{s}, \frac ns )(\o)}  & \leq \|u-P_u\|_{L(\widehat A_{s}, \frac ns )(\o)} + \|P_u\|_{L(\widehat A_{s}, \frac ns )(\o)}
\\ \nonumber  & \leq C \|\nabla ^{[s]}u\|_{L(\widehat A_{\{s\}}, \frac n{\{s\}})(\o)} + C \|1\|_{L(\widehat A_{s}, \frac ns )(\o)}\sum _{k=0}^{[s]}\int _\o| \nabla ^k u|\,dx
\\ \nonumber  &  \leq  C \|\nabla ^{[s]}u\|_{L(\widehat A_{\{s\}}, \frac n{\{s\}})(\o)} + 2 C\|1\|_{L(\widehat A_{s}, \frac ns )(\o)} \|1\|_{L^{\widetilde A}(\o)} \sum _{k=0}^{[s]} \|\nabla ^ku\|_{L^A(\o)} \leq C' \|u\|_{W^{s,A}(\o)},
\end{align}
where the second inequality holds owing to \eqref{omhigher7} applied to $u-P_u$, and the third inequality is due to \eqref{holder}. Embedding \eqref{omhigher2} is a consequence of inequality \eqref{dec215}.
\\ The optimality of the target space in embedding \eqref{omhigher2} is a consequence of Lemma \ref{T:reduction_omega} and of Theorem B.
\\ Embedding \eqref{omhigher1} and inequality \eqref{sep101} follow from embedding \eqref{omhigher2} and inequality \eqref{sep102}, respectively, via Proposition \ref{P:1}.
Lemma \ref{T:reduction_omega}  again and Theorem A imply that the target space is optimal  in embedding \eqref{omhigher1} among all Orlicz spaces.
\end{proof}

\section{Compact embeddings}\label{compact}

%Let $A$ an $B$ be Young functions.
%\note[inline]{Lubos: This definition is repeated after the theorem in a more general form. I suggest to delete this one and to move the more general one here.} We say that the embedding
%$$
%V^{s,A}_d(\rn) \to L^B_{\rm loc}(\rn)
%$$
%is compact if every bounded sequence in $V^{s,A}_d(\rn) $ has a subsequence whose restriction to $E$ converges in $L^B(E)$ for every compact set $E$  in $\rn$.

We conclude by criteria for the compactness of fractional Orlicz-Sobolev embeddings into Orlicz spaces and, more generally, into rearrangement-invariant spaces.
\par The results concerning spaces defined in the whole of $\rn$ have necessarily a local nature, in the following sense.
Given any non-integer positive number $s$,  a Young function $A$ and a rearrangement-invariant space $Y(\rn)$,  we say that the embedding
$$
V^{s,A}_d(\rn) \to Y_{\rm loc}(\rn)
$$
is compact if every bounded sequence in $V^{s,A}_d(\rn) $ has a subsequence whose restriction to $E$ converges in $Y(E)$ for every bounded measurable set $E$  in $\rn$. Here, $Y(E)$ denotes the rearrangement-invariant space given by the restriction of $Y(\rn)$ to $E$, defined as in \eqref{Xr}.

\par
A necessary and sufficient condition for compact embeddings into an Orlicz space amounts to requiring that the Young function that defines the latter space grows essentially more slowly near infinity (in the sense of  \eqref{nov 110}) than the Young function that defines the optimal Orlicz target for merely continuous embeddings. This is the content of the following theorem.

\begin{theorem}\label{T:orlicz_compactness}
Let $n\in\N$ and $s\in (0, n) \setminus \N$.  Let $A$ be a~Young function fulfilling conditions \eqref{E:0'} and \eqref{E:0''}, and let $A_{\frac ns}$ be the Young function defined by \eqref{Ans}.
Assume that $B$ is a Young function. The following properties are equivalent:
\\ \textup{(i)} $B$ grows essentially more slowly near infinity than $A_{\frac{n}{s}}$, namely
\begin{equation}\label{E:almost_compact_orlicz}
%$B$ grows essentially more slowly near infinity than $A_{\frac{n}{s}}$
\lim_{t\to \infty} \frac{B(\lambda t)}{A_{\frac{n}{s}}(t)} =0
\end{equation}
for every $\lambda>0$.
\\ \textup{(ii)} The embedding
\begin{equation}\label{comploc1}
V^{s,A}_d(\rn) \to L^B_{\rm loc}(\rn)
\end{equation}
is compact.
%
%Every sequence $\{u_{i}\}$ of functions in $V^{s,A}(\R^n)$ which satisfies $[\nabla^{[s]}u_i]_{\{s\},A,\Rn}\leq 1$ and $|\{|\nabla ^{k} u_i|>t\}|<\infty$ for every $i \in \N$, $t>0$ and $k=0, \dots [s]$ has a subsequence $\{u_{i_j}\}$ which is convergent in $L^B(S)$ for every bounded set $S\subseteq \Rn$.
\\ \textup{(iii)}
The embedding
\begin{equation}\label{comploc2}
W^{s,A}(\Omega) \to L^B(\o)
\end{equation}
is compact for every bounded Lipschitz domain $\Omega$ in $\Rn$.
\end{theorem}

\begin{example}\label{excompact}
Assume that $\o$ is a bounded Lipschitz domain in $\rn$.
Let $s\in (0,n) \setminus \mathbb N$ and let $A$ be a Young function obeying   \eqref{dec250},    \eqref{dec251}  and \eqref{dec252}.  An application of
Theorem \ref{T:orlicz_compactness} and the use of property \eqref{nov 110bis} tell us that the embedding
\begin{equation*}%\label{dec263}
W^{s,A}(\o) \to L^B (\o)
\end{equation*}
is compact if and only if $B$ is a Young function fulfilling
\begin{equation}\label{dec264}
 \begin{cases} \lim _{t\to \infty}\frac{t^{\frac {n-sp}{np}}(\log t)^{-\frac \alpha n}}{B^{-1}(t)} =0
& \quad  \text{ if $1\leq p< \frac ns$ }
\\
 \lim _{t\to \infty}\frac{ (\log t)^{\frac {n-(\alpha +1)s}n}}{B^{-1}(t)}
%\exp {L^{\frac{n}{n-(\alpha +1)s}}}(\o)
&  \quad \text{if  $p=\frac ns$ and $\alpha < \frac ns -1$}
\\
 \lim _{t\to \infty}\frac{ (\log \log t)^{\frac {n-s}n}}{B^{-1}(t)}
%\exp \exp {L^{\frac n{n-s}}}(\o)
&  \quad \text{if  $p=\frac ns$ and $\alpha = \frac ns -1$.}
\end{cases}
\end{equation}
A parallel conclusion holds, with $A$ and $B$ as above,  for the embedding $V^{s,A}_d(\rn) \to L^B_{\rm loc}(\rn)$.
%\todo[inline]{A: Lenka, in view of the optimal target spaces in \eqref{dec261}, what compact embeddings into r.i. space can we produce as examples here?}
%and the target spaces are optimal in the class of all Orlicz spaces.
%Owing to Theorem \ref{omhigher}, Part (ii), embedding \eqref{omhigher2} holds with $\widehat A$ obeying \eqref{dec253} and \eqref{dec254}. In this case, the resultant space $L(\widehat A, \frac ns)(\o)$ agrees (up to equivalent norms) with a (generalized) Lorentz-Zygmund space of the form \eqref{LZ}. Thus, embedding \eqref{omhigher2} can be written as
%\begin{equation}\label{dec261}
%W^{s}L^p (\log L)^\alpha(\o) \to \,\, \begin{cases} L^{\frac {np}{n-sp}, p; \frac \alpha p} (\o)& \quad  \text{ if $1\leq p< \frac ns$ }
%\\
%L^{\infty, \frac ns; \frac {\alpha s}n-1} (\o) &  \quad \text{if  $p=\frac ns$ and $\alpha < \frac ns -1$}
%\\
%L^{\infty, \frac ns; - \frac sn, -1} (\o)  &  \quad \text{if  $p=\frac ns$ and $\alpha = \frac ns -1$,}
%\end{cases}
%\end{equation}
%all target spaces being optimal among all rearrangement-invariant spaces.

\end{example}

\medskip

Theorem \ref{T:orlicz_compactness} will be deduced via the following
characterization of compact embeddings into rearrangement-invariant spaces. Due to the generality of the latter class of function spaces, such a characterization is naturally less explicit than \eqref{E:almost_compact_orlicz}, but still handy for applications to customary spaces.

\begin{theorem}\label{T:compactness_general}
Let $n\in\N$, $s\in (0, n) \setminus \N$ and let $A$ be a~Young function fulfilling conditions \eqref{E:0'} and \eqref{E:0''}.
Assume that $Y(\Rn)$ is a rearrangement-invariant space. The following properties are equivalent:
\\ \textup{(i)}
\begin{equation}\label{E:almost_compact}
\lim_{L\to 0^+} \sup_{\|f\|_{L^A(0,L)}\leq 1} \left\| \int_r^L f(\varrho) \varrho^{-1+\frac{s}{n}}\,d\varrho\right\|_{\overline Y(0,L)} =0.
\end{equation}
\\ \textup{(ii)} The embedding
\begin{equation}\label{comploc3}
V^{s,A}_d(\rn) \to Y_{\rm loc}(\rn)
\end{equation}
is compact.
\\ \\ \textup{(iii)}
The embedding
\begin{equation}\label{comploc4}
W^{s,A}(\Omega) \to Y(\o)
\end{equation}
is compact for every bounded Lipschitz domain $\Omega$ in $\Rn$.
\end{theorem}

%\todo[inline]{A: examples to be added here}

\begin{example}\label{excompact1}
Let $\o$, $s$ and $A$ be as in Example \ref{excompact}.
Assume that $1\leq r,q\leq \infty$ and $\gamma\in \R$ are such that either $r=q=1$ and $\gamma \geq 0$, or $1<r<\infty$, or $r=\infty$, $q <\infty$ and $\gamma+1/q<0$, or $r=q=\infty$ and $\gamma\leq 0$ (by~\cite[Theorem 9.10.4]{PKJF}, this assumption ensures that $L^{r,q;\gamma}(\Omega)$ is equivalent to a rearrangement-invariant space). Then the embedding
\begin{equation}\label{jan4}
W^{s,A} (\o) \to L^{r,q ;\gamma}(\o)
\end{equation}
is compact if and only if
%\todo[inline]{A: is the \lq\lq only if true"?}
%\todo[inline]{Lenka: I hope so. I believe that everything I used to obtain this example includes the \lq\lq only if'' part.}
either
\begin{equation}\label{jan10}
1\leq p<\frac{n}{s}, \quad \alpha \in \R\quad  \text{and}\quad
\begin{cases}
&r<\frac{np}{n-sp},\\
&r=\frac{np}{n-sp}, \quad p\leq q, \quad   \frac{\alpha}{p}>\gamma,\\
&r=\frac{np}{n-sp}, \quad p>q, \quad  \frac{\alpha}{p}+\frac{1}{p}>\gamma+\frac{1}{q},
\end{cases}
\end{equation}
or
\begin{equation}\label{jan11}
  p=\frac{n}{s}, \quad \alpha\leq \frac{n}{s}-1 \quad  \text{and}\quad
\begin{cases}
&r<\infty\\
&r=\infty, \quad  \frac{\alpha s+s}{n}-1 >\gamma+\frac{1}{q}.
\end{cases}
\end{equation}
If $1\leq p<n/s$, or $p=n/s$ and $\alpha< n/s-1$ then this fact follows  from a combination of Theorem~\ref{T:compactness_general}, \cite[Theorems 4.1 and 4.2 and Proposition 7.2]{S2} and of Example~\ref{exdomains}. In the case when $p =n/s$ and $\alpha= n/s-1$, one needs to additionally observe that the space $L^{\infty, \frac ns; - \frac sn, -1}(\Omega)$ is continuously embedded into $L^{r,q;\gamma}(\Omega)$  if and only if either $r<\infty$, or $r=\infty$ and $\gamma+1/q<0$ (see, e.g., \cite[Theorem 9.5.14]{PKJF}), and that this embedding is in fact almost-compact thanks to the strict inequality in the last condition.
\\ The embedding $V^{s,A}_d(\rn)  \to L^{r,q ;\gamma}_{\rm loc}(\rn)$ is compact  under the same conditions on the exponents $r,q;\gamma$ as in \eqref{jan10} or \eqref{jan11}.
\end{example}

\medskip

Our proof of Theorem \ref{T:compactness_general} makes use of the following lemma.

\begin{lemma}\label{L:pointwise}
Let $n\in\N$, $s\in (0, n) \setminus \N$ and let $A$ be a~Young function fulfilling conditions \eqref{E:0'} and \eqref{E:0''}.
\\ (i)  Any  bounded sequence in $V^{s,A}_d(\R^n)$ has a subsequence  which converges a.e.\ in $\Rn$.
\\ (ii) Assume that $\o$ is a bounded Lipschitz domain in $\rn$. Then
any bounded sequence in $W^{s,A}(\Omega)$ has a subsequence which converges a.e. in $\Omega$.
\end{lemma}

A first step in the proof Lemma \ref{L:pointwise} in its turn relies upon the next result on the almost-compact embedding (according to the notion recalled in Section \ref{rispaces}) of the optimal Orlicz target space in fractional Orlicz-Sobolev embeddings into $L^1$.

\begin{lemma}\label{L:ac_to_l1}
Let $n\in\N$, $s\in (0, n)$ and let $A$ be a~Young function fulfilling conditions \eqref{E:0'} and \eqref{E:0''}. Let $A_{\frac ns}$ be the function defined by \eqref{Ans}. Assume that $E$ is a measurable bounded set in $\Rn$.
Then the Orlicz space $L^{A_{\frac{n}{s}}}(E)$ is almost-compactly embedded into $L^1(E)$.
\end{lemma}

\begin{proof}
Since the function $t\mapsto \frac{t}{A(t)}$ is non-increasing,  one has that
\begin{equation}\label{incr}
\frac{t}{A(t)} \leq \frac{1}{A(1)} \quad  \text{for $t \geq 1$.}
\end{equation}
Hence
$$
H(t)^{\frac{n}{n-s}}
= \int_0^1 \left(\frac{\tau}{A(\tau)}\right)^{\frac{s}{n-s}}\,d\tau
+ \int_1^t \left(\frac{\tau}{A(\tau)}\right)^{\frac{s}{n-s}}\,d\tau
\lesssim t \quad \text{for $t>1$,}
$$	
up to a constant independent of $t$. Thus,
$$
H(t) \lesssim t^{\frac{n-s}{n}} \lesssim A(t)^{\frac{n-s}{n}} \quad \text{near infinity,}
$$
whence
\begin{equation}\label{limit}
A_{\frac{n}{s}}(t) =A(H^{-1}(t)) \gtrsim t^{\frac{n}{n-s}} \quad \text{near infinity.}
\end{equation}
In particular,
\begin{equation}\label{E:limit}
\lim_{t\to \infty} \frac{t}{A_{\frac{n}{s}}(t)} =0.
\end{equation}	
The conclusion hence follows, owing to property \eqref{dec230}.
\end{proof}

\begin{proof}[Proof of Lemma \ref{L:pointwise}]
We provide a proof of Part (i), the proof of Part (ii) being analogous.
 Let   $\mathcal B$ be an open ball in $\Rn$. It suffices to show that any bounded sequence $\{u_{i}\}$ in $V^{s,A}_d(\rn)$ has a subsequence which is convergent in $L^1(\mathcal B)$.
%
%We shall show that there exists a subsequence  of $\{u_{i}\}$ which is convergent in $L^1(\mathcal B)$, and  hence  there exists a subsequence  converging a.e. in $\mathcal B$.
 Assume, for the time being, that $s\in (0,1)$.  By the Riesz-Kolmogorov compactness theorem,
% (see, e.g.,~\cite[Corollary 4.27]{Brezis_book}),
 %it suffices  to show that
this conclusion will follow if we show that
$\{u_{i}\}$ is a bounded sequence in $L^1(\mathcal B)$ and that, for every $\varepsilon >0$,
\begin{equation}\label{nov410}
\text{there exists $\delta>0$ and a compact set $\widehat{\mathcal B} \subseteq \mathcal B$ such that
$\|u_i\|_{L^1(\mathcal B \setminus \widehat{\mathcal B})}<\varepsilon$}
\end{equation}
 and
\begin{equation}\label{nov411} \text{$\int_{\mathcal B} |u_i(x+h)-u_i(x)|\,dx<\varepsilon$ \quad  if $h \in \rn$, $|h|<\delta$ and $i\in \N$. }
\end{equation}
The boundedness of the sequence $\{u_{i}\}$  in $V^{s,A}_d(\R^n)$ amounts to the existence of a positive  constant $C$ such that
\begin{equation}\label{nov416}
\int_{\Rn} \int_{\Rn} A \left( \frac{|u_i(x)-u_i(y)|}{C\,|x-y|^{s}}\right) \,\frac{dy\,dx}{|x-y|^n} \leq 1
\end{equation}
for $i \in \N$. This piece of information
implies, via Theorem~\ref{T:C2h}, that $\|u_i\|_{L^{A_{\frac{n}{s}}}(\mathcal B)} \leq C$ for some constant $C$ independent of $i$. The almost-compact embedding of $L^{A_{\frac{n}{s}}}(\mathcal B)$ into  $L^1(\mathcal B)$ established in Lemma~\ref{L:ac_to_l1} then ensures that
the sequence $\{u_{i}\}$ is bounded in $L^1(\mathcal B)$, and that property \eqref{nov410} holds.
It remains to prove property \eqref{nov411}. To this purpose, we shall show   that
\begin{equation}\label{nov412}
\int_\mathcal B |u_i(x+h)-u_i(x)| \,dx \lesssim |h|^s
\end{equation}
for $h \in \rn$ and $i \in \N$.
Here, and in the remaining part of this proof, the relations $ \lesssim$ and $\approx$ hold up to constants depending on $s$, $n$, $A$ and on the constant $C$ appearing in equation \eqref{nov416}.
Fix $h \in \rn$. We have that
\begin{align}\label{nov415}
\int_{\mathcal B} |u_i(x+h)-u_i(x)|\,dx
&=\int_{\mathcal B \cap \{x:~|u_i(x+h)-u_i(x)| \leq 2^{s+1}|h|^s\}} |u_i(x+h)-u_i(x)|\,dx\\ \nonumber
& \quad +\int_{\mathcal B \cap \{x:~|u_i(x+h)-u_i(x)| > 2^{s+1}|h|^s\}} |u_i(x+h)-u_i(x)|\,dx\\ \nonumber
&\leq 2^{s+1}|h|^s |\mathcal B| + \int_{\mathcal B \cap \{x:~|u_i(x+h)-u_i(x)| > 2^{s+1}|h|^s\}} |u_i(x+h)-u_i(x)|\,dx
\end{align}
for  $i \in \N$.
Fix $i \in \N$, and assume that $x\in \mathcal B$, $y, h\in \Rn$.
%, $|u_i(x+h)-u_i(x)| \geq 2^{s+1}|h|^s$ and
%\todo[inline]{A: why is this assumption needed ?}
%\todo[inline]{Lenka: This is a misprint, it was supposed to be $|x-y|\leq |h|$. But it is not needed in fact. I also think that the assumption $|u_i(x+h)-u_i(x)| \geq 2^{s+1}|h|^s$ is not necessary for the calculation below.}
%{\color{red} $|u_i(x)-u_i(y)|\leq |h|$}.
Since
$$
|u_i(x+h)-u_i(x)|\leq |u_i(x+h)-u_i(y)| + |u_i(y)-u_i(x)|,
$$
one has that
$$
\text{either}\quad |u_i(x+h)-u_i(y)| \geq \tfrac{1}{2}|u_i(x+h)-u_i(x)|\quad \text{or} \quad |u_i(y)-u_i(x)| \geq \tfrac{1}{2}|u_i(x+h)-u_i(x)|.
$$
Therefore, either
\begin{equation}\label{E:1bis}
\big|\{y \in \mathcal B_{|h|}(x): ~|u_i(x+h)-u_i(y)| \geq \tfrac{1}{2}|u_i(x+h)-u_i(x)|\} \big|\geq \frac{\omega_n|h|^n}{2},
\end{equation}
or
\begin{equation}\label{E:1bis'}
\big|\{y \in \mathcal B_{|h|}(x): ~|u_i(y)-u_i(x)| \geq \tfrac{1}{2}|u_i(x+h)-u_i(x)|\}\big| \geq \frac{\omega_n|h|^n}{2}.
\end{equation}
Here, $\mathcal B_{|h|}(x)$ denotes the ball in $\Rn$ centered at $x$ and having radius $|h|$.   In what follows, we assume  that \eqref{E:1bis} is in force, the argument in the case when \eqref{E:1bis'} holds  being  completely analogous. Set
$$S(x,h)=\{y \in \mathcal B_{|h|}(x): ~|u_i(x+h)-u_i(y)| \geq \tfrac{1}{2}|u_i(x+h)-u_i(x)|\}.$$
 The following chain holds:
\begin{align}\label{nov417}
& \int_{\{x\in \mathcal B :~|u_i(x+h)-u_i(x)| > 2^{s+1}|h|^s\}} |u_i(x+h)-u_i(x)|\,dx\\ \nonumber
&= \int_{\{x\in \mathcal B :~|u_i(x+h)-u_i(x)| > 2^{s+1}|h|^s\}} \frac{1}{|S(x,h)|} \int_{S(x,h)} |u_i(x+h)-u_i(x)|\,dy\,dx\\ \nonumber
&\lesssim \frac{1}{|h|^n} \int_{ \{x \in \mathcal B :~|u_i(x+h)-u_i(x)| > 2^{s+1}|h|^s\}} \int_{S(x,h)} |u_i(x+h)-u_i(y)|\,dy\,dx\\ \nonumber
&\lesssim \frac{1}{|h|^n} \int \int_{\{(x,y): ~x\in \mathcal B,~y\in \mathcal B_{|h|}(x), ~|u_i(x+h)-u_i(y)| >2^s|h|^s\}} |u_i(x+h)-u_i(y)|\,dy\,dx\\ \nonumber
&= \frac{1}{|h|^n} \int \int_{\{(x,y): ~x\in \mathcal B,~y\in \mathcal B_{|h|}(x), ~|u_i(x+h)-u_i(y)| >2^s|h|^s\}} \frac{|u_i(x+h)-u_i(y)|}{|x+h-y|^{n+s}} |x+h-y|^{n+s} \,dy\,dx\\ \nonumber
&\lesssim |h|^s \int \int_{\{(x,y): ~x\in \mathcal B,~y\in \Rn, ~|u_i(x+h)-u_i(y)| >|x+h-y|^s\}} \frac{|u_i(x+h)-u_i(y)|}{|x+h-y|^{s}} \,\frac{dy\,dx}{|x+h-y|^n}\\ \nonumber
&\lesssim |h|^s \int \int_{\{(x,y): ~x\in \Rn,~y\in \Rn, ~|u_i(x)-u_i(y)| >|x-y|^s\}} \frac{|u_i(x)-u_i(y)|}{|x-y|^{s}} \,\frac{dy\,dx}{|x-y|^n}\\ \nonumber
&\lesssim |h|^s  \int_{\Rn} \int_{\Rn} A \left( \frac{|u_i(x)-u_i(y)|}{|x-y|^{s}}\right) \,\frac{dy\,dx}{|x-y|^n}
\lesssim |h|^s.
\end{align}
Note that the last inequality holds thanks to property \eqref{incr}.
An application of inequalities \eqref{nov415} and \eqref{nov417} with $u$ replaced by $u/C$, where $C$ is the constant appearing in equation \eqref{nov416}, yields \eqref{nov411}.
\\
Let us next  consider  the case when $s\in (1,n) \setminus \N$. The argument above applied with the sequence $\{u_{i}\}$ replaced by $\{\nabla^{[s]}u_{i}\}$ tells us that there exists a subsequence of $\nabla^{[s]}u_{i}$, still indexed by $i$,  which converges in $L^1(\mathcal B)$.
 Moreover, by Theorem~\ref{T:C1}, the sequence $\{\nabla^{[s]}u_{i}\}$ is bounded in the space $L(\widehat A_{\{s\}},\frac{n}{\{s\}})(\Rn)$, where we are adopting the notation $\widehat A_{\{s\}}$ introduced in the proof of Theorem~\ref{T:C1h}. Making use of inequality~\eqref{aug4} with $s$ replaced by $k+\{s\}$, implies  that the sequence $\{\nabla^{[s]-k}u_{i}\}$ is bounded in $L(\widehat A_{\{s\}+k},\frac{n}{\{s\}+k})(\Rn)$ for $k=1,2,\dots,[s]$. In particular, the sequence $\{\nabla^k u_{i}\}$ is bounded in $L^1(\mathcal B)$ for  $k=0,1,\dots,[s]-1$. On taking, if necessary,   a subsequence  we may also assume that the sequence $\big\{\int_\mathcal B \nabla^k u_{i}dx\big\}$ converges. From  an application of the  Poincar\'e inequality in $W^{1,1}(\mathcal B)$, one can infer that  the sequence $\{\nabla^{[s]-1} u_{i}\}$ converges in $L^1(\mathcal B)$.  An iteration of the same argument  implies that the sequence
$\{\nabla^k u_{i}\}$ converges in $L^1(\mathcal B)$ for $k=0,1,\dots,[s]-1$. The convergence of a subsequence of $\{u_i\}$ in $L^1(\mathcal B)$, hence follows via the choice $k=0$.
\\ The existence of a subsequence of $\{u_i\}$ that converges a.e. on the whole of $\Rn$, follows from  a diagonal argument, by an iterated application of the  result established above   to the sequence of balls $\{\mathcal B_j\}$, centered at $0$, with radius $j\in \N$.
\end{proof}

We conclude with proofs of the main results of this section.

\begin{proof}[Proof of Theorem \ref{T:compactness_general}]
We begin by proving that property \textup{(i)} implies \textup{(ii)}.
 Let $\{u_i\}$ be a bounded sequence in $V^{s,A}_d(\Rn)$.
By Lemma~\ref{L:pointwise}, there  exists subsequence  of $\{u_{i}\}$, still denoted by  $\{u_{i}\}$, which converges a.e. in $\rn$ to some function $u$. Moreover, Theorem~\ref{T:C1h} guarantees that $\{u_{i}\}$ is  bounded   in $L(\widehat A,\frac{n}{s})(\Rn)$. By  Fatou's lemma, $u$  belongs to $L(\widehat A,\frac{n}{s})(\Rn)$ as well. Hence,  $\{u_{i}-u\}$ is a bounded sequence in $L(\widehat A,\frac{n}{s})(\Rn)$. Owing to property \eqref{E:36}, to the definition of
 the Orlicz-Lorentz space $L[\tilde{A},-\frac{n}{s}](0,L)$ and to  the fact that  $L^{\widetilde{A}}(0,L)=(L^A)'(0,L)$ (up to equivalent norms), one has that
\begin{equation}\label{a1_bis}
\|f\|_{L(\widehat A,\frac{n}{s})'(0,L)} \approx \|r^{\frac{s}{n}} f^{**}(r)\|_{(L^A)'(0,L)}.
\end{equation}
Throughout this proof, the relations $ \lesssim$ and $\approx$ hold up to constants depending on $s$, $n$ and $A$.
Thanks to~\cite[Theorem 4.2]{S2},
 assumption~\eqref{E:almost_compact} implies that the space $L(\widehat A,\frac{n}{s})(E)$ is almost compactly embedded into  $Y(E)$
for any bounded measurable set $E$ in $\rn$. An application of property~\eqref{E:convergence} thus tells us that the sequence  $\{(u_{i}-u)\chi_E\}$ converges to $0$ in $Y(E)$, and hence the sequence $\{u_{i}\chi_E\}$  converges to $u\chi_E$ in  $Y(E)$.
\\ Let us next show that, conversely, \textup{(ii)} implies \textup{(i)}. Property \textup{(ii)}
implies  that
$$
\|u\|_{Y(\Rn)} \leq C\big|\nabla^{[s]}u\big|_{\{s\},A,\Rn}
$$
for some constant $C$ and every function $u\in V^{s,A}_d(\R^n)$. Thus, by Lemma~\ref{T:reduction_principle}, the limit in \eqref{E:almost_compact} is finite. Let $T_s$ be the operator defined by  \eqref{T}, with $L=\infty$.
%\todo[inline]{A: in what follows, I have replaced norms on $(0,\infty)$ by norms on $(0, \frac 1i)$, to make the proof consistent with the current formulation of property (i). Please, check!}
For each $i\in \N$, choose a function $f_i \in \mathcal M_+(0,\infty)$,  supported in the interval $[0,\frac{1}{i})$, and such that
$\|f_i\|_{L^A(0,\frac{1}{i})} \leq 1$ and
\begin{equation}\label{E:saturation}
\sup_{\|f\|_{L^A(0,\frac{1}{i})}\leq 1} \left\|T_s\big(\chi_{(0,\frac{1}{i})}f\big) \right\|_{\overline{Y}(0,\frac{1}{i})}
<\left\|T_s f_i\right\|_{\overline{Y}(0,\frac{1}{i})}+\frac{1}{i}.
\end{equation}
%\todo[inline]{Lenka: Do we denote representation spaces with a bar, or without a bar?}
%\todo[inline]{A: yes, I would use the bar}
Set $m=[s]$ and, for $i \in \N$, let $u_i : \rn \to [0, \infty)$ be the function defined as
$$
u_i(x)=\int_{\omega_n |x|^n}^\infty \int_{r_1}^\infty \dots \int_{r_m}^\infty f_i(r_{m+1}) r_{m+1}^{-m-1+\frac{s}{n}}\,dr_{m+1} \dots dr_{1} \quad \text{for $x\in \R^n$}.
$$
Equation \eqref{E:upper}, with $u$ and $f$ replaced by $u_i$ and $f_i$, tells us that
$$
\big|\nabla^{m}u_i\big|_{\{s\},A,\Rn} \lesssim \|f_i\|_{L^A(0,\frac{1}{i})} \leq 1
$$
for $i \in\N$.
In addition, we have $|\{|\nabla ^{k} u_i|>t\}|<\infty$ for every $t>0$ and $k=0, \dots m$ since $u_i$ is compactly supported. Therefore, the function $u_i \in V^{s,A}_d(\rn)$.  Since the supports of the functions $u_i$ are uniformly bounded for $i\in\N$, assumption~\textup{(ii)} ensures that there exists  a subsequence of $\{u_i\}$, still denoted by  $\{u_i\}$, which is convergent in $Y(\Rn)$. Thanks to the properties of $f_{i}$,  one has that $\lim_{i\to \infty} u_{i}=0$, whence
\begin{equation}\label{E:lim1}
\lim_{i\to \infty}\|u_{i}\|_{Y(\Rn)} =0.
\end{equation}
On the other hand, by inequality \eqref{nov207}, with $u$ and $f$ replaced by $u_i$ and $f_i$,
$$
u_{i}(x) \gtrsim \int_{2\omega_n |x|^n}^\infty f_{i}(r) r^{-1+\frac{s}{n}}\,dr \quad \text{for $x \in \rn$.}
$$
Consequently,
\begin{equation}\label{E:lim2}
\|u_{i}\|_{Y(\Rn)} \gtrsim \|T_sf_{i}\|_{\overline{Y}(0,\frac{1}{i})}
\end{equation}
for $i \in\N$.
Coupling equation ~\eqref{E:lim1} with~\eqref{E:lim2} yields
$$
\lim_{i\to \infty} \|T_sf_{i}\|_{\overline{Y}(0,\frac{1}{i})} =0.
$$
Property (i) hence follows, via  equation~\eqref{E:saturation} and the monotonicity with respect to $L$  of the expression under the limit.
\\ The proof of the equivalence of properties
 \textup{(i)} and  \textup{(iii)}  is analogous to that of the equivalence of \textup{(i)} and \textup{(ii)}, and is omitted for brevity.
\end{proof}

\begin{proof}[Proof of Theorem \ref{T:orlicz_compactness}] As in the previous proof, we limit ourselves to showing the equivalence of properties \textup{(i)} and \textup{(ii)}. We first prove that
 \textup{(i)} implies \textup{(ii)}.   Let $\{u_i\}$ be a bounded sequence in $V^{s,A}_d(\R^n)$. By
 Lemma~\ref{L:pointwise}, there exists a subsequence  of $\{u_{i}\}$, still denoted by $\{u_i\}$, which converges a.e. in $\rn$ to some function $u$. Furthermore,  assumption \textup{(i)}, coupled  with Theorem~\ref{T:C2h}, ensures that $\{u_i\}$ is a bounded sequence in $L^{A_{\frac{n}{s}}}(\Rn)$. By Fatou's lemma,  the function $u$   belongs to $L^{A_{\frac{n}{s}}}(\Rn)$, and hence $\{u_{i}-u\}$ is a bounded sequence in $L^{A_{\frac{n}{s}}}(\Rn)$. Assumption~\eqref{E:almost_compact_orlicz}, combined with \cite[Theorem 4.17.7]{PKJF},  tells us that the space  $L^{A_{\frac{n}{s}}}(E)$ is almost compactly embedded into  $L^B(E)$ for any bounded measurable set $E$ in $\rn$.  By applying property~\eqref{E:convergence}, we thus obtain that the sequence $\{u_{i}-u\}$ converges to $0$ in $L^B(E)$, whence the sequence $\{u_{i}\chi_E\}$ is convergent in $L^B(E)$.
\\ We conclude by proving that \textup{(ii)} implies \textup{(i)}.
Assume that property \textup{(ii)} holds. Thanks to Theorem~\ref{T:compactness_general}, this piece of information  ensures that condition~\eqref{E:almost_compact} holds with $\overline{Y}(0,L)=L^B(0,L)$. Owing to Theorem B and \cite[Theorem 4.2]{S2}, this condition  in its  turn implies  that the space  $L(\widehat A,\frac{n}{s})(0,1)$ is almost-compactly embedded into $L^B(0,1)$. On testing this almost-compact embedding on characteristic functions of intervals of the form $(0,L)$ with $L \in (0,1)$, one infers that
\begin{equation}\label{a2_bis}
\lim_{L\to 0^+} \frac{\|\chi_{(0,L)}\|_{L^B(0,1)}}{\|\chi_{(0,L)}\|_{L(\widehat A,\frac{n}{s})(0,1)}} =0.
\end{equation}
By Theorem B and Theorem E,
%By~\eqref{aug301},
%the operator $T$, defined on $\M(0,1)$ by
%\begin{equation*}
%    Tg(r)=\int_{r}^{1}\varrho^{-1+\frac{s}{n}}g(\varrho)\,d\varrho, \quad r\in(0,1),
%\end{equation*}
%is bounded from $L^{A}(0,1)$ into $L(\widehat A,\frac{n}{s})(0,1)$, and the target space is optimal space that renders the boundedness of $T$ true. Therefore, by Theorem~E, one has
\begin{equation*}
    \|f\|_{L(\widehat A,\frac{n}{s})'(0,1)} \approx \|r^{\frac{s}{n}}f^{**}(r)\|_{L^{\widetilde A}(0,1)}
\end{equation*}
for every function $f\in\M(0,1)$. Here, and in what follows,  equivalence is up to multiplicative constants depending only on $n$, $s$ and $A$.
In particular,
\begin{equation}\label{dec240}
    \|\chi_{(0,L)}\|_{L(\widehat A,\frac{n}{s})(0,1)} \approx \frac{L}{\|\chi_{(0,L)}\|_{L(\widehat A,\frac{n}{s})'(0,1)}}\approx
\frac{L}{\|r^{\frac{s}{n}}\chi_{(0,L)}^{**}(r)\|_{L^{\widetilde A}(0,1)}} \quad \text{for $L\in(0,1)$,}
\end{equation}
up to equivalence constants independent of $L$. Notice that the first equivalence holds thanks to a general property of rearrangement-invariant norms \cite[Chapter 2, Theorem 5.2]{BS}.
On the other hand,
\begin{equation}\label{dec241}
    \|r^{\frac{s}{n}}\chi_{(0,L)}^{**}(r)\|_{L^{\widetilde A}(0,1)}
        \approx \|r^{\frac{s}{n}}\chi_{(0,L)}(r)\|_{L^{\widetilde A}(0,1)}
            + L\|r^{-1+\frac{s}{n}}\chi_{(L,1)}(r)\|_{L^{\widetilde A}(0,1)} \quad \text{for $L\in(0,1)$.}
\end{equation}
In particular, if  $L\in(0,\tfrac{1}{2})$, then
\begin{align}\label{dec242}
\|r^{-1+\frac{s}{n}}\chi_{(L,1)}(r)\|_{L^{\widetilde A}(0,1)}
    & \ge \|r^{-1+\frac{s}{n}}\chi_{(L,2L)}(r)\|_{L^{\widetilde A}(0,1)}
    \ge (2L)^{-1+\frac{s}{n}}\|\chi_{(L,2L)}\|_{L^{\widetilde A}(0,1)} \\ \nonumber
    & = (2L)^{-1+\frac{s}{n}}\|\chi_{(0,L)}\|_{L^{\widetilde A}(0,1)}
        \ge 2^{-1+\frac{s}{n}}L^{-1}\|r^{\frac{s}{n}}\chi_{(0,L)}(r)\|_{L^{\widetilde A}(0,1)}.
\end{align}
Equations \eqref{dec240}--\eqref{dec242} tell us that
\begin{equation}\label{a3_bis}
\|\chi_{(0,L)}\|_{L(\widehat A,\frac{n}{s})(0,1)} \approx \frac 1{\|r^{-1+\frac sn}\chi_{(L, 1)}(r)\|_{L^{\widetilde A}(0, 1)}}\quad \text{as $L \to 0^+$.}
\end{equation}
Furthermore, owing to \cite[Lemma~1]{cianchi_pacific},
\begin{equation}\label{a4_bis}
\|r^{-1+\frac sn}\chi_{(L, 1)}(r)\|_{L^{\widetilde A}(0, 1)} \approx A_{\frac ns}^{-1}(1/L)  \quad \text{as $L \to 0^+$.}
\end{equation}
On the other hand, $\|\chi_{(0,L)}\|_{L^B(0,1)} = \frac 1{B^{-1}(1/L)}$ for $L\in (0,1)$. This equality, and equations \eqref{a3_bis} and \eqref{a4_bis} entail that
 condition~\eqref{a2_bis}
is equivalent to
\begin{equation}\label{a5_bis}
\lim_{t\to \infty} \frac{A_{\frac ns}^{-1}(t)}{B^{-1}(t)} =0,
\end{equation}
and the latter is in its turn equivalent to \eqref{E:almost_compact_orlicz}.
%
%\textup{(ii)} $\Leftrightarrow$ \textup{(iii)} The proof is analogous to that of the equivalence of \textup{(i)} and \textup{(iii)} and is omitted for brevity.
\end{proof}

\section*{Compliance with Ethical Standards}\label{conflicts}

\subsection*{Funding}

This research was partly funded by:

\begin{enumerate}
\item Research Project 201758MTR2  of the Italian Ministry of Education, University and
Research (MIUR) Prin 2017 ``Direct and inverse problems for partial differential equations: theoretical aspects and applications'';
\item GNAMPA of the Italian INdAM -- National Institute of High Mathematics
(grant number not available);
\item Grant P201-18-00580S of the Czech Science Foundation;
\item Deutsche Forschungsgemeinschaft (DFG, German Research Foundation) under Germany's Excellence Strategy - GZ 2047/1, Projekt-ID 390685813.
\end{enumerate}

\subsection*{Conflict of Interest}

The authors declare that they have no conflict of interest.

\end{document}